\documentclass[11pt,a4paper,reqno]{amsart}
\usepackage[applemac]{inputenc}
\usepackage[T1]{fontenc}
\usepackage{amsmath}
\usepackage{amsthm}
\usepackage{amsfonts}
\usepackage{amssymb}
\usepackage{graphicx}
\usepackage{amsbsy}
\usepackage{mathrsfs}
\usepackage{bbm}
\newtheorem{theorem}{Theorem}[section]
\newtheorem{lemma}[theorem]{Lemma}

\newtheorem{proposition}[theorem]{Proposition}

\newtheorem{corollary}[theorem]{Corollary}
\newtheorem{definition}[theorem]{Definition}
\theoremstyle{remark}
\newtheorem{remark}[theorem]{\it \bf{Remark}\/}

\numberwithin{equation}{section}
\catcode`@=11
\def\section{\@startsection{section}{1}%
  \z@{1.5\linespacing\@plus\linespacing}{.5\linespacing}%
  {\normalfont\bfseries\large\centering}}
\catcode`@=12
\newcommand{\be}{\begin{equation}}
\newcommand{\ee}{\end{equation}}
\newcommand{\bea}{\begin{eqnarray}}
\newcommand{\eea}{\end{eqnarray}}
\newcommand{\bee}{\begin{eqnarray*}}
\newcommand{\eee}{\end{eqnarray*}}

\def\pa{\partial}

\def\fref#1{{\rm (\ref{#1})}}

\def\Bb{\bar B}

\catcode`@=11
\def\supess{\mathop{\operator@font Sup\,ess}}
\catcode`@=12

\def\bt{\tilde{b}}

\def\RR{\mathbb{R}}

\def\e{\varepsilon}

\def\bar#1{{\overline #1}}
\def\fref#1{{\rm (\ref{#1})}}

\def\R2+{\RR ^2_+}

\def\lsl{\frac{\lambda_s}{\lambda}}

\def\pa{\partial}

\def\lim{\mathop{\rm lim}}

\def\sup{\mathop{\rm sup}}

\def\l{\lambda}

\def\log{{\rm log}}

\def\et{\tilde{\e}}

\def\lsl{\frac{\lambda_s}{\lambda}}
\def\xsl{\frac{x_s}{\lambda}}

\def\cal{\mathcal}

\def\matchal{\mathcal}

\def\pa{\partial}

\def\tb{\tilde{b}}
\def\et{\tilde{\e}}

\def\pa{\partial}

\def\eoy{e^{\omega y}}

\def\eoyt{e^{\frac \omega 3 y}}
\def\NNb{\overline{\mathcal N }}
\def\NNbl{\overline{\mathcal N}_{\rm loc}}

\title[Blow up for the critical gKdV II]{Blow up for the critical gKdV equation. \\ II: Minimal mass dynamics}
\author[Y. Martel]{Yvan Martel}
\address{Universit\'e de Versailles St-Quentin and Institut Universitaire de France, LMV  CNRS UMR8100.
Current address : Ecole polytechnique, CMLS UMR7640}
\email{yvan.martel@polytechnique.edu}
\author[F. Merle]{Frank Merle}
\address{Universit\'e de Cergy Pontoise and Institut des Hautes \'Etudes Scientifiques, AGM CNRS UMR8088}
\email{merle@math.u-cergy.fr}
\author[P. Rapha\"el]{Pierre Rapha\"el}
\address{Universit\'e Paul Sabatier and  Institut Universitaire de France, IMT CNRS UMR          5219.
Current address : Universit\'e de Nice Sophia-Antipolis, Laboratoire J.A. Dieudonn\'e
 CNRS UMR7351
}
\email{pierre.raphael@unice.fr}
\begin{document}

\begin{abstract}
We consider the mass critical (gKdV) equation $u_t + (u_{xx} + u^5)_x =0$ for initial data in $H^1$. We first prove the existence and uniqueness in the energy space of a minimal mass blow up solution and give a sharp description of the corresponding blow up soliton-like bubble. We then show that this solution is the universal attractor of all solutions near the ground state which have  a defocusing behavior. This allows us to sharpen the description of  near soliton dynamics obtained in \cite{MMR1}. \end{abstract}

\maketitle

\section{Introduction}

\subsection{Setting of the problem}
We continue the study of the mass critical generalized Korteweg--de Vries equation:
\begin{equation}\label{kdv}
{\rm (gKdV)}\quad  \left\{\begin{array}{ll}
 u_t + (u_{xx} + u^5)_x =0, \quad & (t,x)\in [0,T)\times\RR, \\
 u(0,x)= u_0(x), & x\in {\mathbb R},
\end{array}
\right.
\end{equation}
initiated in  Part I \cite{MMR1}.
The Cauchy problem is locally well posed in the energy space $H^1$ from Kenig, Ponce and Vega \cite{KPV,KPV2}, and given $u_0 \in H^1$, there exists a unique\footnote{in a certain sense} maximal solution $u(t)$ of \eqref{kdv} in $C([0,T), H^1)$ and 
\be
\label{blowucifoi}
T<+\infty \ \  \mbox{implies} \ \ \lim_{t\to T} \|u_x(t)\|_{L^2} = +\infty.
\ee
The Cauchy problem for \eqref{kdv} is also locally well-posed in $L^2$ and given  $u_0 \in L^2$, there exists a unique maximal solution $u(t)$ of \eqref{kdv} in $C([0,T), L^2)$ with either $T=+\infty$ or $$T<+\infty \ \ \mbox{and then}\ \  \|u\|_{L^5_xL^{10}_{(0,T)}}=\infty.$$
Moreover, $H^1$ solutions satisfy the conservation of mass and energy: $$M(u(t))=\int u^2(t)= M_0, \ \ E(u(t))= \frac 12 \int u_x^2(t) - \frac 16 \int u^6(t)= E_0.$$ The symmetry group of \fref{kdv} is continuous in $H^1$ and given by
$$\epsilon_0\lambda_0^{\frac12}u(\lambda_0^3 (t-t_0),\lambda_0 (x -x_0)), \ \ (\epsilon_0,\l_0,x_0,t_0)\in \{-1,1\}\times\RR^*_+\times\RR\times\RR.$$
In particular the scaling symmetry leaves the $L^2$ norm invariant and hence the problem is mass or $L^2$ critical.\\
Travelling wave solutions play a distinguished role in the analysis $$u(t,x)=Q(x-t)$$ where $Q$ is the ground state solitary wave $$Q(x) =   \left(\frac {3}{\cosh^{2}\left( 2 x\right)}\right)^{\frac14}$$ which attains the sharp Gagliardo-Nirenberg inequality, \cite{W1983}:
 \be\label{gn}
\forall v\in H^1,\quad
\int |v|^6 \leq \int v_x^2 \left(\frac {\int v^2}{\int Q^2}\right)^2.
\ee
The conservation of mass and energy and the blow up criterion \fref{blowucifoi} ensure that  $H^1$ initial data with subcritical mass $\|u_0\|_{L^2}<\|Q\|_{L^2}$ generate global in time solutions.

\subsection{The flow near the ground state} 
In the series of   works \cite{MMjmpa,MMgafa,Mjams,MMannals, MMjams}, Martel and Merle obtain the first qualitative information on the flow for small super critical mass initial data $\|Q\|_{L^2}<\|u_0\|_{L^2}<\|Q\|_{L^2}+\alpha^*$, $0<\alpha^*\ll1$, in particular the existence of finite time blow up solutions for $E_0<0$ and the classification of $Q$ as the unique global attractor of all $H^1$ blow up solutions.\\

In Part I \cite{MMR1}, we have revisited the blow up analysis in light of recent developments related to blow up for the mass critical Schr\"odinger equation \cite{MRgafa, MRinvent, MRannals, MRcmp, MRjams, MRS}  and energy critical geometrical equations \cite{RR2009, MRR, Rstud}.

More precisely, let the set of initial data
$$
\mathcal{A}=
\left\{
u_0=Q+\e_0   \hbox{ with } \|\e_0\|_{H^1}<\alpha_0 \hbox{ and }
\int_{y>0} y^{10}\e_0^2< 1
\right\}, $$
and consider the $L^2$ tube around the family of solitary waves $$\mathcal T_{\alpha^*}=\left\{u\in H^1\ \ \mbox{with}\ \ \inf_{\l_0>0, \, x_0\in \RR}
\|u- \frac{1}{\l_0^{\frac12}}Q\left(\frac{ .-x_0}{\l_0} \right)\|_{L^2} <\alpha^*\right\}.$$ 

In \cite{MMR1}, we have proved the following (see Theorems 1.1 and 1.2 in \cite{MMR1} for more details).

\begin{theorem}[Rigidity of the flow in $\mathcal A$, \cite{MMR1}]
\label{thmmoe}
Let $0<\alpha_0\ll\alpha^* \ll1$ and $u_0\in \mathcal A$. Let $u\in \matchal C([0,T), H^1)$ be the corresponding solution to \fref{kdv}. Then, one of the following three scenarios occurs:\\
{\em (Blow up)}: the solution blows up in finite time $0<T<+\infty$ in the universal regime 
\be\label{loiblowup}
\|u(t)\|_{H^1}=\frac {\ell(u_0)+o(1)}{T-t}\ \ \mbox{as}\ \ t\to T, \ \ \ell(u_0)>0.
\ee
{\em (Soliton)}: the solution is global $T=+\infty$ and converges asymptotically to a solitary wave.\\
{(\em Exit)}: the solution leaves the tube $\matchal T_{\alpha^*}$ at some time $0<t^*<+\infty$.

Moreover, the scenarios {\em (Blow up)} and {\em (Exit)} are {\it stable} by small perturbation of the data in $\mathcal A$.
\end{theorem}

Our aim in this paper is  first to classify the 
{\it minimal mass} dynamics $\|u_0\|_{L^2}=\|Q\|_{L^2}$ and then, from this classification,
to complete the description obtained in Theorem~\ref{thmmoe} in the (Exit) regime.
Indeed, we will show that for $\alpha_0$ small enough, the (Exit) case  is directly connected to the understanding of minimal mass dynamics.

\subsection{Minimal mass dynamics} The question of existence and possibly uniqueness of minimal blow up dynamics for dispersive and parabolic PDE's has motivated several works since the pioneering result by Merle \cite{Mduke} for the mass critical nonlinear Schr\"odinger equation: 
\begin{equation}\label{nls}
{\rm (NLS)}  \quad  {\rm i} \partial_t u+\Delta u +|u|^{\frac 4N} u=0, \ \ (t,x)\in \RR\times \RR^N. \end{equation}
Let us recall that for (NLS), the pseudo conformal symmetry generates an {\it explicit} minimal mass blow up solution \be
\label{vneionveonv}
S_{\rm NLS}(t,x)=\frac 1 {t^{N/2}} e^{- i\frac {|x|^2}{4t}- \frac i t  } Q_{\rm NLS}\left(\frac xt\right)
\ee where $Q_{\rm NLS}$ is the ground state solution to $$\Delta Q_{\rm NLS}-Q_{\rm NLS}+Q_{\rm NLS}^{1+\frac 4N}=0, \ \ Q>0, \ \ Q\in H^1.$$ 

Merle proved in \cite{Mduke} that $S_{\rm NLS}$ is the unique (up to the symmetries of the equation) minimal mass blow up element in the energy space.  The proof heavily relies on the pseudo conformal symmetry. Such minimal blow up dynamics have also been exhibited for the energy critical NLS and wave problems \cite{DM}, \cite{DM2}, using the virial algebra and a fixed point argument.
For the inhomogeneous mass critical (NLS) in dimension 2: $$  {\rm i} \partial_t u+\Delta u +k(x)|u|^2 u=0,
$$
while Merle  \cite{Me} derived  sufficient conditions on $k(x)$ to ensure the {\it nonexistence} of minimal elements,
 Rapha\"el and Szeftel \cite{RS2010}  introduced a more dynamical approach to existence and uniqueness under a necessary and sufficient condition on $k(x)$.   A robust energy method is implemented to completely classify the minimal mass blow up, in regimes such that the inhomogeneity $k$ influences dramatically the bubble of concentration \fref{vneionveonv} -- in contrast with    direct perturbative methods developed in \cite{BW}, \cite{BGT}, \cite{BCD}, see also \cite{KLR} for existence in the  one dimensional half wave problem.

\medskip

Recall that for the mass critical (gKdV) problem \fref{kdv}, Martel and Merle   \cite{MMduke} obtained the following {\it global existence} result for minimal mass solutions with   decay on the right.

\begin{theorem}[Global existence at minimal mass, \cite{MMduke}]\label{thmmduke}
Let $u_0\in H^1$ with $\|u_0\|_{L^2}=\|Q\|_{L^2}$ and 
\be
\label{conditiondecay}
\sup_{x_0>0} x_0^3\int_{x>x_0}u_0^2(x)dx<+\infty.
\ee
Then, the corresponding solution $u(t)$ of \eqref{kdv} is global for $t>0$.
\end{theorem}

In other words, minimal mass blow up is not compatible with the decay \fref{conditiondecay}. This is in agreement with the analysis in \cite{MMR1} where the threshold dynamics for data in $\mathcal A$ between the stable (Blow up) and (Exit) regimes is proved to correspond  to  a solitary wave behavior -- and not to a minimal blow up. We refer to \cite{MRS} for a further discussion of threshold dynamics. 


\subsection{Statement of the result}


The first main result of this paper is the {\it existence and uniqueness} in the energy space of a minimal mass blow up element: 

\begin{theorem}[Existence and uniqueness of the minimal mass blow up element]
\label{th:1}
\quad \\
{\rm (i) Existence.} There exists a   solution $S(t)\in \mathcal C((0,+\infty),H^1)$  to \fref{kdv} with minimal mass $\|S(t)\|_{L^2}=\|Q\|_{L^2}$ which blows up backward at the origin:
$$S(t,x)-\frac{1}{t^{\frac 12}}Q\left(\frac{x+ \frac 1t+\bar{c}t}{t}\right)\to 0\ \ \mbox{in}\ \ L^2\ \ \mbox{as}\ \ t\downarrow 0$$ at the speed 
\be
\label{speedlvowup}
\|S(t)\|_{H^1}\sim \frac {C^*}t \ \ \mbox{as} \ \ t\downarrow 0
\ee
for some universal constants $\bar{c},C^*$. Moreover, $S$ is smooth and well localized to the right in space: \be\label{th:1.6}
\forall x \geq 1, \quad 
S(1,x) \leq e^{-Cx}.
\ee
{\rm  (ii) Uniqueness. } Let $u_0\in H^1$ with $\|u_0\|_{L^2}=\|Q\|_{L^2}$ and assume that the corresponding solution $u(t)$ to \fref{kdv} blows up in finite time. Then $$u\equiv S$$ up to the symmetries of the flow.
  \end{theorem}

Observe that the minimal element blows up with   speed \fref{speedlvowup} which is the same as in the    (Blow up) regime obtained in Theorem \ref{thmmoe}. However, the case of (Blow up) in Theorem \ref{thmmoe} is   shown to be stable by small perturbation in $\matchal A$, while minimal mass blow up is unstable by perturbation of the data $S(0)\to (1-\e)S(0)$, $\e>0$, since the corresponding solution has subcritical mass and is thus global in time. This shows that the decay assumption to the right in Theorem \ref{thmmoe} is essential and that the minimal blow up solution has slow decay to the left\footnote{remember that it blows up backwards in time.}. The nature of the minimal blow up    is different   from the one of stable blow up.\\
 
We now relate the (Exit) case in Theorem \ref{thmmoe} to the minimal mass blow up dynamics. We claim that at the (Exit) time, the solution is $L^2$ close up to renormalization to the unique minimal solution $S(t)$. 

\begin{theorem}[Description of the (Exit) scenario]
\label{PR:1}
Let $u(t,x)$ be a solution of \eqref{kdv} corresponding to the {\em (Exit)} scenario in Theorem \ref{thmmoe}
and let $t^*_u\gg 1$ be the corresponding exit time. Then there exist  $\sigma^*=\sigma^*(\alpha^*)$ (independent of $u$) and $(\lambda_u^*,x_u^*)$ such that 
$$\left\|  (\l_u^*)^{\frac 12} u\left(t_u^*,   \l_u^* x +  x_u^*\right) -S(\sigma^*,x)\right\|_{L^2} \leq \delta_I(\alpha_0),$$
where 
$\delta_I(\alpha_0) \to 0$ as ${\alpha_0\to 0}.$
\end{theorem}

Note that uniqueness in Theorem \ref{th:1} is an essential ingredient of the proof. In view of the universality of $S$ as attractor to all defocused solutions, and in continuation of Theorem \ref{th:1}, it is an important open problem 
  to understand the behavior of $S(t)$ as $t\to +\infty$. For the mass critical (NLS), the explicit formula \fref{vneionveonv} ensures that $S_{\rm NLS}$ {\it scatters} as $t\to\infty$, and hence it is a connection from $+\infty$ to $0$. For (gKdV),   the decay in space \fref{th:1.6} of $S(t,x)$ on the left, combined  with Theorem~\ref{thmmduke}, ensures that $S(t)$ is globally defined for $t>0$, but scattering as $t\to +\infty$ is an open problem\footnote{by scattering for (gKdV), we mean that there exists a solution $v(t,x)$ to the Airy equation $\pa_tv+v_{xxx}=0$ such that $\lim_{t\to +\infty}\|S(t)-v(t)\|_{L^2}=0$.}. We conjecture that $S(t)$ actually scatters, and because scattering is an open in $L^2$ property, \cite{KPV2}, we obtain the corollary:
  
\begin{corollary}
Assume  that $S(t)$ scatters as $t\to+\infty$. Then any solution in the {\em (Exit)} scenario is global for positive time and  scatters as $t\to +\infty$.
\end{corollary}

Related rigidity theorems near the solitary wave were recently obtained by Nakanishi and Schlag \cite{NS1}, \cite{NS2} for super critical wave and Schr\"odinger equations using the invariant set methods of Berestycki, Cazenave \cite{BeresCaze}, the Kenig-Merle concentration compactness approach \cite{KM}, the classification of minimal dynamics \cite{DM2}, \cite{DM}, \cite{DR} and a further ``no return'' lemma in the (Exit) regime. This approach relies on the virial algebra which is  not known for (gKdV).\\

We expect the strategy of the proof of Theorem \ref{PR:1},  reducing the dynamics of defocused solutions to the sole description of the minimal mass solution, to be quite general. 

\medskip

\noindent{\bf Aknowledgements}: 
The authors would like to thank the anonymous referees for their very useful corrections and comments.

P.R is supported by the French ERC/ANR project SWAP.
This work is also supported by the project ERC 291214 BLOWDISOL. This work was completed when P.R was visiting the MIT Mathematics Department, Boston, which he would like to thank for its kind hospitality.\\

\medskip

\noindent{\bf Notation}: We introduce the generator of $L^2$ scaling: $$\Lambda f=\frac12f+yf'.$$ We note the $L^2$ scalar product: $$(f,g)=\int_{\RR}f(x)g(x)dx.$$ Let the linearized operator close to $Q$ be:
\be
\label{deflplus}
Lf=-f''+f-5Q^4f.
\ee  
For a given generic  small constant $0<\alpha^*\ll1 $,  $\delta(\alpha^*)$ denotes a generic positive small constant with $$\delta(\alpha^*)\to 0\ \ \mbox{as}\ \ \alpha^*\to 0.$$ Given $I$ an interval of $\RR$, we let ${\mathbf{1}}_I$ denote the characteristic function of $I$.


\subsection{Strategy of the proof}
\label{strategy}

Let us give a brief insight into the strategy of the proof of Theorem \ref{th:1} and Theorem \ref{PR:1}.\\

{\bf step 1} Modified blow up profiles. We construct the minimal element using a variation of the compactness argument used for the construction of non dispersive objects in \cite{Me0},  \cite{Ma1}, \cite{MMC}, \cite{RS2010}. This solution will admit near blow up time a decomposition $$u(t,x)=\frac{1}{\l^{\frac 12}(t)}(Q_{b(t)}+\e)(s,y)\ \ \mbox{with}\ \ \frac{ds}{dt}=\frac{1}{\l^3}, \ \ y=\frac{x-x(t)}{\l(t)},$$ and $$\e(t)\to 0\ \  \mbox{in}\ \ H^1 \ \ \mbox{as}\ \ t\downarrow 0.$$ Here $Q_b$ is the slow modulated deformation of the ground state constructed in \cite{MMR1} which formally leads to the dynamical system $$\left\{\begin{array}{ll}b_s+2b^2=0\\  -\lsl=b\end{array}\right.  \ \ \mbox{i.e.}\ \ \left\{\begin{array}{ll} \lambda(t) =\ell^*t,\\ \ \ b(t)=-\ell^*\l^2(t)\end{array}\right.$$ and hence the blow up speed \fref{speedlvowup}.\\

{\bf step 2} The formal argument. Following \cite{Me0}, \cite{RS2010}, we could build the minimal element by considering the solution $u_n(t)$ to (gKdV) with data $$u(t_n)=\frac{1}{\l(t_n)^{\frac 12}}Q_{b(t_n)}\left(\frac{x-x(t_n)}{\l(t_n)}\right) \ \ \mbox{with}\ \ \lambda(t_n)=\ell^*t_n, \ \ b(t_n)=-\ell^*\l^2(t_n)$$ and show that there exists a time $t_0>0$ independent of $n$ such that 
$$\|u_n(t_0)\|_{H^1}\lesssim 1\ \ \mbox{as}\ \  t_n\downarrow 0.
$$
Such an estimate is the heart of the proof and would be a consequence of the fine monotonicity properties exhibited in \cite{MMR1}. Passing to the limit $t_n\to 0$ automatically produces the expected blow up element.\\
We will argue slightly differently and propose a scheme adapted to the proof of both Theorem \ref{th:1} and Theorem \ref{PR:1} and which as in \cite{MRS} illustrates the fact the minimal element can be obtained as limiting sequences of {\it defocusing} solutions. Indeed, we pick a sequence of {\it well prepared} initial data $$u_n(0)=Q_{b_n(0)}, \ \ b_n(0)=-\frac{1}{n}$$ which by construction have sub critical mass $$\|u_n(0)\|_{L^2}=\|Q\|_{L^2}-\frac{c}{n}+o\left(\frac{1}{n}\right).$$ Such solutions are automatically in the (Exit) regime of Theorem \ref{thmmoe}. Moreover, we have from \cite{MMR1} a {\it complete description of the flow} for $t\in[0,t^*_n]$ i.e. the solution admits a decomposition 
\be
\label{cneoneoncoe}
u_n(t,x)=\frac{1}{\l_n^{\frac 12}(t)}(Q_{b_n(t)}+\e_n)\left(t,\frac{x-x_n(t)}{\l_n(t)}\right)
\ee where to leading order the modulation equations for $(b_n,\l_n)$ are given by $$\frac{b_n(t)}{\l^2_n(t)}\sim b_n(0)=-\frac{1}{n}, \ \ (\lambda_n)_t\sim- b_n(0)$$ i.e. 
\be
\label{cnekneneo}
\lambda_n(t)\sim 1-b_n(0)t, \ \ b_n(t)\sim b_n(0)\l_n^2(t).
\ee
The (Exit) time $t_n^*$ is the one for which the solution moves strictly away from the solitary wave which in our setting is equivalent to $$b_n(t_n^*)=-\alpha^*$$ independent of $n$. This in particular allows us to compute $t_n^*$ and show using \fref{cnekneneo} that the solution defocuses: $$\l^2_n(t_n^*)\sim \frac{b_n(t_n^*)}{b_n(0)}\sim n\alpha^* \ \  \mbox{as}\ \ n\to +\infty.$$ We therefore renormalize the flow at $t_n^*$ and consider the solution to (gKdV) with data at $t_n^*$ given by the renormalized $u_n$ at $t_n^*$, explicitly: $$v_n(\tau,x)=\l_n^{\frac12}(t_n^*)u_n(t_\tau,\l_n(t_n^*)x+x_n(t_n^*)),\ \ t_{\tau}=t_n^*+\tau \l_n^3(t_n^*).$$ Then $v_n$ admits from direct check a decomposition $$v_n(\tau,x)=\frac{1}{\l_{v_n}(\tau)^{\frac 12}}(Q_{b_{v_n}}+\e_{v_n})\left(\tau,\frac{x-x_{v_n(\tau)}}{\lambda_{v_n(\tau)}}\right)$$ with from the symmetries of the flow 
$$\l_{v_n}(\tau) = \frac {\l_n(t_\tau)}{\l_n(t_n^*)},  \ \ 
x_{v_n}(\tau)= \frac  {x_n(t_\tau)-x_n(t_n)}{\l_n(t_n^*)} ,\ \ 
b_{v_n}(\tau) = b_n(t_\tau),\ \ 
\e_{v_n}(\tau) = \e_n(t_\tau).
$$
The renormalized parameters can be computed approximatively using \fref{cnekneneo}:
\bee
\l_{v_n}(\tau) & \sim & \frac{\l_n(t_n^*+\tau\l_n^3(t_n^*))}{\l_n(t_n^*)}\sim \frac{1}{\l_n(t_n^*)}\left[1-b_n(0)(t_n^*+\tau\l_n^3(t_n^*))\right]\\
& \sim & \frac{1}{\l_n(t_n^*)}\left[\l_n(t_n^*)-\tau b_n(0)\l_n^3(t_n^*)\right]=1-\tau b_n(0)\l^2_n(t_n^*)\\
& = & 1-\tau b_n(t_n^*)= 1+\tau\alpha^*.
\eee
Letting $n\to +\infty$, we therefore expect to extract a weak limit $v_n(0)\rightharpoonup v(0)$ such that the corresponding solution $v(\tau)$ to (gKdV) has minimal mass $\|v(0)\|_{L^2}=\|Q\|_{L^2}$ and blows up backwards at some finite time $\tau^*\sim -\frac{1}{\alpha^*}$ with the blow up speed $\lambda_v(\tau)\sim \tau-\tau^*$ i.e \fref{speedlvowup}.\\
The extraction of the weak limit now requires sharp controls on the remaining radiation $\e_{v_n}$. Here an essential use is made of the fact that the set of data $u_n(0)$ is {\it well prepared} as this induces uniform bounds for $\e_{v_n}(0)=\e_{u_n}(t_n^*)$ in $H^1$ and allow us to use the $H^1$ weak continuity of the flow in the limiting process.\\

{\bf step 3} Solutions in the (Exit) regime. The proof of Theorem \ref{PR:1} follows similarly considering sequences of data $(u_0)_n$ with $\|(u_0)_n\|_{L^2}\to \|Q\|_{L^2}$ such that the corresponding solution to (gKdV) is in the (Exit) regime. We write explicitly the solution at the (Exit) time in the form \fref{cneoneoncoe}, renormalize the flow and now aim at extracting a weak limit as $n\to +\infty$. The architecture of the proof is similar, except that we have lost the fact that the data is well prepared which destroys the uniform $H^1$ bound on $v_n(0)$. We therefore use two new tools: a concentration compactness argument on sequences of solutions in the critical $L^2$ space in the spirit of \cite{KM} using the tools developed in \cite{KKSV}, which allows us to extract a non trivial weak limit with suitable dynamical controls; refined local $H^1$ bounds on $v_n(\tau)$ in order to ensure that the $L^2$ limit is in fact also in $H^1$. Hence the weak limit is a minimal mass $H^1$ blow up element.\\

{\bf step 4} Uniqueness. It remains to prove the {\it uniqueness} in $H^1$ of the minimal element. This is a delicate problem and here we adapt the direct dynamical approach developed in \cite{RS2010}. The first step is to show that any $H^1$ minimal blow up element blows up with the blow up speed \fref{speedlvowup}. Here the proof relies first on exponential decay estimates of minimal elements proved in \cite{MMannals} which allow us in a second time to enter the monotonicity machinery developed in \cite{MMR1}. Once the blow up speed is known, one may integrate the flow backwards from the singularity and show that the blow up element is close in a strong sense to the $S(t)$ minimal element previously constructed. It remains to show that the difference is exactly zero. This requires revisiting the monotonicity properties for the difference of two such solutions, and showing that the previously obtained a priori bounds on the solution implies that the difference is exactly zero\footnote{This equivalently 
means that the integration of the flow from blow up time defining the minimal blow up element is a contraction mapping in a suitable function space.}. Let us insist that as in \cite{RS2010}, \cite{MRS}, we are forced to work with a finite order approximation of the solution\footnote{and not arbitrarily degenerate as in \cite{BW} for example.} and therefore this step is always delicate.


\section{Nonlinear profiles and decomposition close to the soliton}


We collect in this section a number of tools which can be explicitly found in the literature and which we will use in the proof of the main results. We start with recalling the status of scattering theory and profiles decomposition in the critical $L^2$ space for (gKdV). We then recall the nonlinear decomposition of the flow for data near the ground state, and the main monotonicity formula at the heart of the analysis in \cite{MMR1} and which will play again a distinguished role in the analysis.


\subsection{Cauchy problem and scattering from \cite{KPV}}


We use in the paper the terminology strong solution in the sense of Kenig, Ponce, Vega \cite{KPV}.
For $u_0\in L^2$, we denote by $v(t)=W(t)u_0$ the solution of the Airy equation 
$v_t+v_{xxx}=0$ with $v(0)=v_0$. The following space-time Strichartz type estimate proved in \cite{KPV} is essential in the resolution of the Cauchy problem for \eqref{kdv} in $L^2$ and $H^1$:
\be
\label{strichtrx}
\|W(t) v_0\|_{L^5_xL^{10}_t} \lesssim \|v_0\|_{L^2}.
\ee We recall the following classical results.

\begin{theorem}
[Kenig, Ponce, Vega \cite{KPV}]\label{thmkpv}
 \label{KPV}\quad \\
\noindent{\rm (i) $L^2$ theory.} The Cauchy problem \eqref{kdv} is locally well-posed in $L^2$: for all $u_0 \in L^2$, there exists a unique $L^2$ solution of \eqref{kdv} defined on a maximal interval of existence $[0,T)$. There is continuous dependance on the data in $L^2$, and there holds the blow up alternative:$$T<+\infty \ \ \mbox{implies}\ \ \|u\|_{L^5_xL^{10}_T}=+\infty.$$ Moroever, there exists $\delta>0$ such that $\|u_0\|_{L^2} < \delta$ implies that the solution is global with $\|u\|_{L^5_xL^{10}_{\infty}} <+\infty$.

\noindent{\rm (ii) $H^1$ theory.}
The Cauchy problem \eqref{kdv} is locally well-posed in $H^1$: for all $u_0 \in H^1$, there exists a unique $H^1$ solution of \eqref{kdv} defined on a maximal interval of existence $[0,T)$.  There is continuous dependance on the data in $H^1$, and there holds the blow up alternative:$$T<+\infty \ \ \mbox{implies}\ \ \lim_{t\uparrow T}\|\pa_xu\|_{L^2}=+\infty.$$

\noindent{\rm (iii) Scattering and stability of scattering.}
Let $u(t)$ be a global $L^2$ solution of \eqref{kdv}. If $\|u\|_{L^5_xL^{10}_{\infty}}<+\infty$, then the solution $u(t)$ scatters at $+\infty$ i.e. there exists $v_0^+ \in L^2$ such that
$$
\lim_{t\to   \infty} \|u(t) - W(t) v_0^+\|_{L^2}=0. 
$$
The set $\mathcal{S} = \{u_0 \in L^2$ such that $u(t)$ is global and scatters at $+\infty\}$ is open in $L^2$.
\end{theorem}

Point (iii) of Theorem \ref{KPV}   follows from \cite{KPV} and standard arguments (see e.g. \cite{KM} for similar arguments in the case of nonlinear Schr\"odinger equation), and means that scattering is a stable regime without any assumption of size on the solution.

\medskip

We now recall the known results from \cite{Shao} on profile decomposition in the critical space of sequences of solutions to the Airy equation which describes the lack of compactness of the Strichartz estimate \fref{strichtrx}. For any   $x_0\in \RR$ and   $\lambda > 0$,   define  the operator $g_{x_0,\lambda}: L^2(\RR) \to L^2(\RR)$ 
$$
[g_{x_0, \lambda} f](x) :=  \lambda^{-\frac12} f\bigl( \lambda^{-1}(x-x_0) \bigr).
$$
\begin{definition}\label{def:ot}
For $j\neq k$, $\Gamma_n^j=(\lambda_n^j, \xi^j_n,x_n^j, t_n^j)_{n\geq 1}$ and
$\Gamma_n^k=(\lambda_n^k, \xi^k_n,x_n^k, t_n^k)_{n\geq 1}$ in $(0,\infty)\times\RR^3$ are
{\rm orthogonal} if one of the following holds
\begin{itemize}
\item $\displaystyle \lim_{n\rightarrow \infty} \left(\frac{\lambda_n^j}{\lambda_n^k}
+\frac{\lambda_n^k}{\lambda_n^j}+\lambda_n^j |\xi_n^j-\xi_n^k| \right)=\infty$;
\item $(\lambda_n^j,\xi_n^j) = (\lambda_n^k,\xi_n^k)$ and
$$
\lim_{n\rightarrow \infty} \left(\frac{|t_n^k-t_n^j|}{(\lambda_n^j)^3}
+ \frac{3|(t_n^k-t_n^j)\xi_n^j|}{(\lambda_n^j)^2}
+\frac{|x_n^j-x_n^k+3(t_n^j-t_n^k) (\xi_n^j)^2|}{\lambda_n^j} \right)=\infty.
$$
\end{itemize}
\end{definition}

\begin{lemma}[Profile decomposition \cite{Shao}]\label{le:profile}
Let $\{u_n\}_{n\geq 1}$ be a sequence of real-valued functions bounded in $L^2$.  Then, after passing to a subsequence if
necessary, there exist (complex) $L^2$ functions $\{\phi^j\}_{j\geq 1}$, 
and a family of orthogonal sequences $\Gamma_n^j=(\lambda_n^j,\xi_n^j, x_n^j,t_n^j)\in (0,\infty)\times \RR^3$
such that for all $J\geq 1$,
\begin{equation}\label{profile}
u_n= \sum_{1\le j\le J} 
e^{-t_n^j\partial_x^3}\left(g_{x_n^j, \lambda_n^j} [{\rm Re}(e^{ix\xi_n^j\lambda_n^j}\phi^{j})]\right)+ w_n^J,
\end{equation}
where 
$\xi_n^j$ satisfy the following property: for any $1\leq j\leq J$, either $\xi_n^j= 0$ for all $n\geq 1$, or
$\xi_n^j\lambda_n^j\to \infty$ as $n\to \infty$.  Here, $w^J_n \in L^2(\RR)$ is real-valued and 
\begin{equation}\label{errorprofile}
\lim_{J\to \infty}\limsup_{n\to \infty} 
\left\{\bigl\||\partial_x|^{1/6}e^{-t\partial_x^3}w_n^J\bigr\|_{L^6_{t,x}
 (\RR\times\RR)}+
\bigl\| e^{-t\partial_x^3}w_n^J\bigr\|_{L^5_xL^{10}_t(\RR\times\RR)}\right\}
=0.
\end{equation}
Moreover, 
for any $J\geq 1$,
\begin{align} \label{orthoprofile}
\lim_{n\to\infty} \Biggl\{\|u_n\|^2_{L^2}-\sum_{1\le j\le J}\bigl\|{\rm Re}[e^{ix\xi_n^j\lambda_n^j}\phi^j]\bigr\|^2_{L^2} -
\|w_n^J\|^2_{L^2}\Biggr\}=0.
\end{align}\end{lemma}

Using this lemma for the study of the nonlinear flow \fref{kdv} requires a suitable perturbation theory:

\begin{lemma}[$L^2$ perturbation theory  \cite{KKSV}]\label{le:perturbation}
Let $I$ be an interval of $\RR$, $ 0 \in I$, and let   $\tilde u$ be an $L^2$ solution of 
$$
u_t+ (u_{xx}+u^5)_x=e_x
$$
on $I\times \RR$ for some function $e$. Assume that 
$$
\|\tilde u\|_{L^\infty_tL^2_x(I\times \RR)} +
\|\tilde u\|_{L^5_xL^{10}_t(I\times \RR)}\leq M,
$$
for some $M>0$.
Let $u(0)\in L^2$ be such that 
$$
\|u(0)-\tilde u(0)\|_{L^2} \leq M',
$$
$$
\|e^{-t\partial^3} (u(0)-\tilde u(0))\|_{L_x^5L^{10}_t(I\times \RR)}
+ \| e \|_{L^1_xL^2_t(I\times \RR)} \leq \epsilon,
$$ 
for $M'>0$ and for some small $0<\epsilon<\epsilon_0(M,M')$.
Then, the solution $u(t)$ of \eqref{kdv} corresponding to $u(t_0)$ is defined on $I$ and there holds the bound:
\begin{align}
&\|u-\tilde u\|_{L^5_xL^{10}_t(I\times \RR)}+\|u-\tilde u\|_{L^\infty_tL^2_x(I\times \RR)}\leq C(M,M') \epsilon.
\end{align}
\end{lemma}

 
\subsection{Approximate self similar profiles}


We recall the existence of suitable approximate self similar solutions which give the leading order profile of solutions with data near $Q$. The specific sutrcture of these profiles drives both the blow speed in the (Blow up) regime and the speed of defocalization in the (Exit) regime. Let
       ${\mathcal Y}$  be the set of functions $f\in \mathcal C^\infty({{\mathbb R}},{{\mathbb R}})$ such that
\begin{equation}\label{Y}
    \forall k=0,1,2\ldots,\ \exists C_k,\, r_k>0,\ \forall y\in {{\mathbb R}},\quad |f^{(k)}(y)|\leq C_k (1+|y|)^{r_k} {e^{-|y|}}. 
\end{equation}
Let $\chi\in {\cal C}^{\infty}({\mathbb R})$ be such that $0\leq \chi \leq 1$, $\chi'\geq 0$ on $\mathbb{R}$,
$\chi\equiv 1$ on $[-1,+\infty)$, $\chi\equiv 0$ on $(-\infty,-2]$.
Define
\begin{equation}\label{defgamma} 
\chi_b(y)= \chi\left(|b|^{\gamma} {y} \right),\ \ \gamma = \frac 34.
\end{equation}

\begin{lemma}[Approximate self-similar profiles $Q_b$, \cite{MMR1}]
\label{cl:2}
There exists a unique smooth function $P$ such that $P' \in \mathcal{Y}$ and\be
\label{eq:23}
 (L  P)'={\Lambda} Q, \ \  \lim_{y \to -\infty}   P(y) = \frac 12 \int Q,  \ \ \lim_{y \to +\infty}   P(y) =0,
 \ee
 \be
 \label{PQ}
 (P, Q) = \frac 1{16} \left(\int Q\right)^2 > 0,\quad (P,  Q')=0.
 \ee
Moreover, the localized approximate profile: 
\begin{equation}\label{eq:210}
\quad 
Q_b(y) = Q(y) + b \chi_b(y)   P (y),
\end{equation}
 satisfies:\\
{\rm (i) Estimates on $Q_b$:} For all $y \in \mathbb{R}$,
\begin{align}
&	|Q_b(y)|\lesssim  e^{-|y|} + |b| \left(  {\mathbf{1}}_{[-2,0]}(|b|^\gamma y) +  e^{-\frac {|y|}{2}} \right),
\label{eq:001}\\
&	|Q_b^{(k)}(y)|\lesssim  e^{-|y|} +   |b|  e^{-\frac {|y|}{2}} +|b|^{1+k \gamma} {\mathbf{1}}_{[-2,-1]}(|b|^\gamma y),\quad  \text{for $k\geq 1$},
\label{eq:002}
\end{align}
\noindent
{\rm (ii) Equation of $Q_b$:} let
\begin{equation}\label{eq:201}
-\Psi_b=\left(Q_b''- Q_b+ Q_b^5\right)'+b {\Lambda} Q_b,
\end{equation}
then, for all $y\in \mathbb{R}$,
\begin{align}
& |\Psi_b(y)|\lesssim |b|^{1+ \gamma}  {\mathbf{1}}_{[-2 ,-1]}(|b|^{\gamma}y) 
+  b^2  \left(e^{-\frac {|y|} 2}+ {\mathbf{1}}_{[- 2,0]}(|b|^{\gamma}y) \right),\label{eq:202}
\\
& |\Psi_b^{(k)}(y)|\lesssim    |b|^{1+ (k+1) \gamma}  {\mathbf{1}}_{[-2 ,-1]}(|b|^{\gamma}y)  + b^2   e^{-\frac {|y|} 2} ,\quad 
 \text{for $k\geq 1$}. \label{eq:203}
\end{align}
{\rm (iii) Mass and energy properties of $Q_b$:}
\begin{align}
& \left|\int Q_b^2 - \left( \int Q^2 + 2b \int PQ \right)\right| \lesssim |b|^{2-\gamma}, \label{eq:204}\\
& \left| E(Q_b) + b \int PQ\right|\lesssim b^2 .
\label{eq:205}
\end{align}
\end{lemma}

 
\subsection{Geometrical decomposition of the flow}


Let  $u\in \mathcal C^0([0,t_0],H^1)$ be a solution of \fref{kdv}   close in $L^2$ to the manifold  of solitary waves
i.e., we assume that there exist $(\lambda_1(t),x_1(t))\in \RR^*_+\times \RR$ and $\e_1(t)$ such that
\be\label{decompo1}
\forall t\in [0,t_0], \ \ u(t,x)=\frac{1}{\lambda^{\frac 12}_1(t)}(Q+\e_1)\left(t,\frac{x-x_1(t)}{\lambda_1(t)}\right)
\ee
with 
\be
\label{hypeprochien}
\forall t\in [0,t_0],\ \
 \|\e_1(t)\|_{L^2}+\left(\int (\partial_y \e_1)^2 e^{-\frac {|y|}{2}} dy \right)^{\frac 12}\leq \kappa_0
\ee
for some small enough universal constant $\kappa_0>0$. This decomposition is refined using the $Q_b$ profiles and a standard modulation argument.

\begin{lemma}[Decomposition and $H^1$ properties, \cite{MMR1}]
\label{le:2}
Assume \fref{hypeprochien}.

\noindent{\em (i) Decomposition:} There exist $\mathcal C^1$ functions $({\lambda}, x,{b}):[0,t_0]\to (0,+\infty)\times \mathbb{R}^2$ such that
\begin{equation}\label{defofeps}
\forall t\in [0,t_0],\quad
\varepsilon(t,y)={\lambda}^{\frac 12}(t) u(t,{\lambda}(t) y + x(t)) - Q_{{b}(t)}(y)
\end{equation}
satisfies the orthogonality conditions
\be
\label{ortho1}
(\e(t), y {\Lambda}   Q )=(\e(t),\Lambda Q) =(\e(t),Q)=0.
\ee
and 
\be
\label{controle}
\|\varepsilon(t)\|_{L^2}+ |{b}(t)|\lesssim \delta(\kappa_0),\quad
\|\e(t)\|_{H^1}\lesssim \delta(\|\e_1(t)\|_{H^1}).
\ee
{\em (ii) Equation of $\varepsilon$}: 
Let
\be
\label{rescaledtime}
{s}=   \int_{0}^{t} \frac {dt'}{{\lambda}^3(t')}\quad \hbox{and} \quad s_0=s(t_0).
\ee
For all ${s}\in [0,s_0]$,
\begin{align}  \nonumber  
  \varepsilon_{s} - (L \varepsilon)_y + {b}{\Lambda} \varepsilon
&= \left(\frac {{\lambda}_{s}}{{\lambda}}+{b}\right)({\Lambda} Q_b+{\Lambda} \varepsilon)
+ \left(\frac { x_{s}}{\lambda} -1\right) (Q_b + \varepsilon)_y \\ & 
+ \Phi_{b}  
 + \Psi_{{b}} - (R_{{b}}(\varepsilon))_y -(R_{\rm NL}(\varepsilon))_y,
\label{eqofeps}\end{align}
where $\Psi_{b}$ is defined in \eqref{eq:201} and
\begin{align}
&\Phi_{b} =- {b}_{s} \left(\chi_b   + \gamma   y (\chi_b)_y\right) P, \label{rbnl0}
\\ 
&R_{{b}}(\varepsilon)= 5 \left(Q_b^4 - Q^4\right) \varepsilon,
\quad 
R_{\rm NL}(\varepsilon)=(\varepsilon+Q_b)^5
- 5 Q_b^4 \varepsilon - Q_b^5.\label{rbnl}
\end{align}
{\em (iii) Estimates induced by the conservation laws}:  
on $[0,s_0]$, there holds
\be
\label{twobound}
\|\e(s)\|^2_{L^2}\lesssim |b(s)|+\left|\int u_0^2-\int Q^2\right|,
\ee
\be
\label{energbound}
\left|2 \lambda^2(s)E_0+\frac {b(s)}{8}\|Q\|_{L^1}^2-   \| \e_y(s)\|_{L^2}^2  \right|\lesssim b^2(s)+  \|\e(s)\|_{L^2}^2   + \delta(\|\e\|_{L^2}) \|\e_y(s)\|_{L^2}^2.
\ee
{\em (iv) Rough modulation equations}: on $[0,s_0]$,
\begin{align}
& \left|\frac {{\lambda}_{s}}{\lambda} + {b}\right| + \left| \frac { x_{s}}{\lambda} - 1 \right| \lesssim
\left(\int \varepsilon^2 {e^{-\frac{|y|}{10}}} \right)^{\frac 12} +  b^2;\label{eq:2002}\\
& |{b}_{s}+2b^2|  \lesssim
 |b|\left( \int   \varepsilon^2  {e^{-\frac{|y|}{10}}} \right)^{\frac12}+|b|^3+\int  \varepsilon^2   {e^{-\frac{|y|}{10}}} .\label{eq:2003}
\end{align}
{\rm (v) Minimal mass}: if in addition, $\|u(t)\|_{L^2}=\|Q\|_{L^2}$ then $E_0=E(u_0)\geq 0$ and on $[0,s_0]$,
\be
\label{twoboundmini}
b(s)\leq 0, \quad E_0\lambda^2(s)\lesssim |b(s)|+\|\e(s)\|^2_{H^1}\lesssim E_0\lambda^2(s).
\ee
\end{lemma}

The proof of Lemma \ref{le:2} is given in \cite{MMR1}, except 
\eqref{twoboundmini} which we prove now.
 
\begin{proof}[Proof of \eqref{twoboundmini}]
Using the decomposition \eqref{defofeps}, one has
$$
\int u^2=\int Q_b^2+\int\varepsilon^2+2\int \varepsilon Q_b.
$$
Since $(\varepsilon,Q)=0$, and $\chi_b(y)=\chi(|b|^{\frac34}y)$,
\begin{equation}\label{ePb} 
 |(\varepsilon,Q_b)|=|b||(\varepsilon,P\chi_b)|\lesssim
 |b|^{\frac58}\left(\int\varepsilon^2\right)^{\frac12}\lesssim
 |b|^{\frac18}\int \varepsilon^2+|b|^{\frac98}.
\end{equation}
Moreover, by \fref{eq:204}, 
$$\int Q_b^2=\int Q^2+2\int PQ +O(|b|^{\frac54}).$$
Thus, we obtain in general
\begin{equation}\label{NewL2}
 \int u^2=\int Q^2+2b\int PQ+\int \varepsilon^2+O(|b|^{\frac18}) \left(|b|+\int \varepsilon^2\right).
\end{equation}
In particular, using the minimal mass assumption $\int u^2=\int Q^2$,
\be
\label{estunbnormeldeux}
2b(P,Q)+\int\e^2=\delta(\kappa_0)\left(|b|+\int \varepsilon^2\right),
\ee
which implies $b\leq 0$. Now we write the conservation of energy using $(\e,Q)=0$ and \fref{eq:205}:
\bee
2{\lambda}^2 E_0 &  = &  2E(Q_b) - 2\int  \varepsilon  (Q_b)_{yy}  + \int \varepsilon_y^2 - \frac 13 \int \left((Q_b+\varepsilon)^6  -Q_b^6\right)\nonumber\\
	&  = & - 2{b} \int PQ + O(b^2)+\int\e_y^2-5\int Q^4\e^2-\frac13\int\e^6\\
	& - &  2\int \varepsilon\left[(Q_b - Q)_{yy}+(Q_b^5-Q^5)\right]+5 \int  (Q^4-Q_b^4)\e^2\\
	& - &   \frac 13 \int \left[(Q_b+\varepsilon)^6   - Q_b^6-6Q_b^5\e-15Q_b^4\e^2- \varepsilon^6 \right]
.
\eee
We estimate the nonlinear terms using the Sobolev bound
$$\|\e\|_{L^{\infty}}\lesssim \|\e_y\|_{L^2}^{\frac 12}\|\e\|_{L^2}^{\frac 12},$$
and thus
\begin{equation}\label{NewE}
 2{\lambda}^2 E_0=-2b\int PQ+\int\varepsilon_y^2-5\int Q^4\varepsilon^2
 +O(|b|^{\frac18}+\|\varepsilon\|_{H^1})\left(|b|+\int \varepsilon^2\right).
\end{equation}
Combining with \fref{estunbnormeldeux}, we obtain:
$$2\l^2E_0=(L\e,\e)+\delta(\kappa_0)(|b|+\|\e\|^2_{H^1}).$$
The choice of orthogonality conditions on $\varepsilon$ ensures (see Lemma  2.1 in \cite{MMR1}) the coercivity of the linearized energy, i.e. $(L\varepsilon,\varepsilon)\gtrsim \|\varepsilon\|_{H^1}^2$ and thus
$$\|\e\|_{H^1}^2\lesssim \l^2E_0+\delta(\kappa_0)|b|,$$ which combined with \fref{estunbnormeldeux}  implies \fref{twoboundmini}.
\end{proof}

The modulation equations can be sharpened under an additional $L^1$ control of the solution.

\begin{lemma}[Refined laws for $H^1$ solution with decay, \cite{MMR1}]
\label{lemmarefinedmod}
Under the assumptions of Lemma \ref{le:2}, assume moreover the uniform $L^1$ control on the right:
\be
\label{cnbooeoe}
\forall t\in [0,t_0], \ \ \int_{y>0}|\e(t)|\lesssim \delta(\kappa_0),
\ee
then the quantities $J_1$ and $J_2$   below are well-defined and satisfy on $[0,t_0]$:
\begin{itemize}
\item{Law of ${\lambda}$:} let
\begin{equation}\label{rho1}
\rho_1(y)=  \frac 4{\left(\int Q\right)^2}  \int_{-\infty}^y {\Lambda} Q, \ \ {J_1}({s})=  (\varepsilon(s),\rho_1), 
\end{equation}
then for some universal constant $c_1$,
\begin{equation}\label{eq:2004} 
\left|  {\frac{{\lambda}_s}{{\lambda}}}   + {b} +c_1 b^2   - 2 \left(({J_1})_{s} + \frac 12 {\frac{{\lambda}_s}{{\lambda}}}  {J_1} \right) \right| \lesssim  \int \varepsilon^2 {e^{-\frac{|y|}{10}}}   +  |b|\left( \int \varepsilon^2 {e^{-\frac{|y|}{10}}} \right)^{\frac 12}. 
\end{equation}
\item{Law of $b$:} let
\be
\label{rho2} 
\rho_2 = \frac {16}{\left( \int  Q\right)^2} \left(\frac {({\Lambda} P, Q)}{\|{\Lambda} Q\|_{L^2}^2} {\Lambda} Q + P-\frac 12 \int Q\right)
- 8 \rho_1,\ \  
{J_2}({s})=    (\varepsilon({s}),  \rho_2), 
\ee
then for some universal constant $c_2$.
\begin{equation}\label{eq:2006}
\left|b_s + 2b^2 +c_2 b^3  + b \left(({J_2})_s + \frac 12 {\frac{{\lambda}_s}{{\lambda}}} {J_2}\right) \right| \lesssim  \int \varepsilon^2 {e^{-\frac{|y|}{10}}}   +  |b|^4.
\end{equation}

\item{Law of $\frac b{{\lambda}^2}$:} let $$\rho = 4\rho_1 + \rho_2, \ \ J = (\e,\rho),$$ then $\rho\in \mathcal{Y}$ 
so that $|J| \lesssim (\int \e^2 e^{-\frac {|y|}{10}})^{\frac 12}$ and 
  for $c_0=c_2-2c_1$,
\begin{equation}\label{eq:bl2}
\left|\frac d{ds}\left( \frac b{{\lambda}^2}\right)  + \frac b{{\lambda}^{2}} \left(J_s+ \frac 12 {\frac{{\lambda}_s}{{\lambda}}} J\right) +c_0 \frac {b^3}{\l^2}  \right|
\lesssim \frac 1{{\lambda}^2}  \left( \int \varepsilon^2 {e^{-\frac{|y|}{10}}} + | b|^4\right).
\end{equation}
\end{itemize}
\end{lemma}


\subsection{Weak $H^1$ stability of the decomposition}


The geometrical decomposition of Lemma \ref{le:2} is stable by weak $H^1$ limits.

\begin{lemma}[$H^1$-weak stability and  convergence of the parameters \cite{MMjmpa}]\label{le:2.7}
 Let $u_n(0)$ be a sequence of $H^1$ initial data such that
 $$
u_n(0)\rightharpoonup u(0)\in H^1\ \ \hbox{ as $n\to +\infty$.}
$$
Assume that for some $T_1>0$,   for all $n$, the corresponding solution $u_n$ of \eqref{kdv} exists and satisfies \eqref{hypeprochien}  on $[0,T_1]$. Assume further that the  decomposition 
$(\lambda_n,x_n,b_n)$ of $u_n$ given by Lemma \ref{le:2} satisfies
\be\label{hypoweak}
\forall t\in [0,T_1],\quad 
0<c\leq \l_n(t) <C ,\quad \l_n(0)=1, \quad x_n(0)=0 .
\ee
Then, the $H^1$ solution $u(t)$ of \eqref{kdv} corresponding to $u(0)$ exists on $[0,T_1]$, satisfies \eqref{hypeprochien} and its decomposition satisfies
\be\label{convweak}
\forall t\in [0,T_1],\quad 
\e_n(t)\rightharpoonup \e(t),\quad 
\l_n(t)\to \l(t),\quad x_n(t)\to x(t), \quad b_n(t)\to b(t).
\ee
\end{lemma}

This lemma is   similar to a result proved in Lemma 17 and Appendix D of \cite{MMjmpa}, and   therefore we omit its proof. 
  

\subsection{Main monotonicity functionals from \cite{MMR1}}


We now recall the monotonicity formula at the heart of the analysis in \cite{MMR1} and on which we shall heavily rely again. We refer to \cite{MMR1} for a further introduction to the natures of these functionals and the associated ridigity of the flow implied by \fref{conrolbintegre}.\\ 
Let  $\varphi,\psi\in C^{\infty}({\mathbb R})$  be such that:
\bea
\label{defphi1bis}
\varphi(y) =\left\{\begin{array}{lll}e^{y}\ \ \mbox{for} \ \   y<-1,\\
 1+y  \ \ \mbox{for}\ \ -\frac 12<y<\frac 12\\ 
 y^2\ \ \mbox{for}\ \ \mbox{for}\ \ y>2
 \end{array}\right., \ \ \varphi'(y) >0, \ \ \forall y\in \RR,
\eea
\bea
\label{defphi2}
\psi(y) =\left\{\begin{array}{ll} e^{2y}\ \ \mbox{for}   \ \ y<-1,\\
 1  \ \ \mbox{for}\ \ y>-\tfrac 12\end{array}\right., \ \ \psi'(y) \geq 0 \ \ \forall y\in \RR.
\eea

For  $B\geq 100$  to be fixed, let
$$\psi_B(y)=\psi\left(\frac yB\right), \ \ \varphi_{B}=\varphi \left(\frac yB\right), $$ 
and define
 \be\label{eq:no}
{\mathcal{N}}  (s)  = \int  \varepsilon_y^2(s,y) \psi_B(y)dy 
+ \int  \varepsilon^2(s,y) \varphi_B(y)   dy.
\ee

\begin{proposition}[Monotonicity formula, \cite{MMR1}]
\label{propasymtp}
There exist $\mu>0$ and  $0< \kappa^*< \kappa_0$ such that the following holds for $B>100$ large enough. Assume that     $u(t)$ is a solution of \eqref{kdv} which satisfies \eqref{hypeprochien} on $[0,t_0]$ and thus
admits on $[0,t_0]$ a decomposition \fref{defofeps} as in Lemma \ref{le:2}. Let $s_0=s(t_0)$, and assume the  following a priori bounds: $\forall s\in [0,s_0]$,\\
{\em (H1)   smallness:} 
\be
\label{boundnwe}
\|\e(s)\|_{L^2}+|b(s)|+\mathcal N(s)\leq \kappa^*;
\ee
{\em (H2) comparison between $b$ and $\l$:}
\be
\label{bootassumption}
 \frac{|b(s)|+\mathcal N(s)}{\lambda^2(s)}\leq \kappa^*;
\ee
{\em (H3) $L^2$ weighted bound on the right:}
\be
\label{uniformcontrol}
\int_{y>0}y^{10}\e^2(s,x)dx\leq 10\left(1+\frac{1}{\l^{10}(s)}\right).
\ee
Let    for $j\in\{1,2\}$:
\bea
\label{feps}
 {\cal F}_{j}  =  \int \left[\varepsilon_y^2 \psi_B+ \varepsilon^2(1+\mathcal J_{j})\varphi_{B}- \frac  13 \left((\varepsilon + Q_b)^6  -  {Q_b^6}   - 6 \varepsilon Q_b^5\right)\psi_B\right] ,
\eea
with 
\be
\label{defjj}
\mathcal J_{j}=(1-J_1)^{-4j}-1.
\ee
Then   the following bounds hold on $[0,s_0]$:\\
{\em (i) Scaling invariant Lyapounov control:}  
\be
\label{lyapounovconrol}
\frac{d\mathcal F_{1}}{ds}+  \mu    \int \left(\e_y^2 + \e^2\right) \varphi_B'\lesssim |b|^{4}.
\ee
{\em (ii) Scaling weighted $H^1$ Lyapounov control:} 
\be
\label{lyapounovconrolbis}
\frac{d}{ds}\left\{\frac{\mathcal F_{2}}{\l^2}\right\}+  \frac \mu {\l^2}   \int  \left(\e_y^2 + \e^2\right) \varphi_B' \lesssim \frac{|b|^{4}}{\l^2}.
\ee
{\em (iii) Pointwise bounds:}
\be
\label{controlj}
|J_1|+|J_2|\lesssim \mathcal N^{\frac{1}{2}} ,
\ee
\be
\label{lowerbound}
\mathcal N  \lesssim \mathcal F_{j}\lesssim \mathcal N  , \ \ j=1, 2.
\ee
\end{proposition}

The integration of the modulation equations of Lemma \ref{lemmarefinedmod} with the dispersive bounds of Proposition \ref{propasymtp} implies the control of the flow by the sole parameter $b$:

\begin{lemma}[Control of the flow by $b$, \cite{MMR1}]\label{le:oubli}
Under the assumptions of Proposition \ref{propasymtp}, the following hold

\noindent{\rm (i) Control of the dynamics for $b$.}
For all $  0\leq s_1\leq s_2< s_0$,
\be
\label{contorlbonehoeh}
\int_{s_1}^{s_2}b^2(s)ds\lesssim   \int \left( \e_y^2(s_1) \psi_B + \e^2(s_1)\varphi_B'\right) +|b(s_2)|+|b(s_1)|,
\ee
\be
\label{conrolbintegre}
\left|\frac{b(s_2)}{\lambda^2(s_2)}-\frac{b(s_1)}{\lambda^2(s_1)}\right|\leq
\frac {C^*}{10}\left[\frac{b^2(s_1)}{\lambda^2(s_1)}+\frac{b^2(s_2)}{\l^2(s_2)}+ \frac{1}{\lambda^2(s_1)} \int \left( \e_y^2(s_1) \psi_B + \e^2(s_1)\varphi_B'\right)\right],
\ee
for some universal constant $C^*>0$.
\\
\noindent{\rm (ii) Control of the scaling dynamics.}
Let $\lambda_0(s) = \lambda(s) (1-J_1(s))^2$. Then on $[0,s_0)$,
\be\label{dvnkoenneoneor}
 \left|\frac{(\l_0)_s}{\l_0}+b+c_1b^2 \right| \lesssim  \int \e^2e^{-\frac {|y|}{10}} +|b| \mathcal{N}^{\frac 12} +|b|^3 .
\ee\\
\noindent{\rm (iii) Dispersive bounds.}
For all $  0\leq s_1\leq s_2< s_0$,
\be
\label{estfondamentale}
\mathcal N (s_2)+\int_{s_1}^{s_2}\left[\int \left( \e_y^2+ \e^2\right)(s)\varphi_B' + |b|^4(s)\right] ds \lesssim  \mathcal N (s_1)+ (|b^3(s_2)|+|b^3(s_1)|).
\ee
\be
\label{estfondamentalebis}
\frac{\mathcal N (s_2)}{\lambda^2(s_2)}+\int_{s_1}^{s_2}\left[\int \left( \e_y^2+ \e^2\right)(s)\varphi_B'+|b|^{4}(s)\right] \frac {ds}{\lambda^2(s)}\lesssim \frac{\mathcal N(s_1)}{\lambda^2(s_1)}+ \left[\frac{|b^3(s_1)|}{\lambda^2(s_1)}+\frac{|b^3(s_2)|}{\lambda^2(s_2)}\right].
\ee
\end{lemma}

 
\subsection{Localization in space and decay properties of minimal mass solutions}


Minimal mass blow up solutions have been studied in some details in \cite{MMduke} using tools developed in \cite{MMannals} and \cite{MMjams}.
Recall that the main result of \cite{MMduke} is the nonexistence of minimal mass blow up solutions, assuming initial decay in space. In proving this result, several general properties of minimal mass blow up solutions were derived. We gather in the next lemma all useful information which can be deduced from \cite{MMduke} on general minimal mass blow up solutions.
Note that at this stage, we do not know whether a minimal mass blow up  solution should blow up in finite or infinite time. See Proposition \ref{sharpbounds} in Section~4 for refined information.

  \begin{lemma}[First properties of minimal mass blow up solutions \cite{MMduke}]
  \label{lemmadecay}
  Let $u(t)$ be a solution of \eqref{kdv} defined on $(T,0]$, which blows up backwards in finite or infinite time $-\infty\leq T<0$.
  Assume
  \be\label{encoreL2}
  \int u^2 (0) = \int Q^2.
  \ee
There exists $t_1>T$ close to $T$ such that for all $t\in(T,t_1]$,
$u(t)$ or $-u(t)$ admits a decomposition $(\l(t),x(t),b(t),\e(t))$ as in Lemma \ref{le:2},
with 
\be
\label{aausmptoblowup}
\lim_{t\to T}\lambda(s)=0,
\ee
and 
 for all $t\in (T,t_1)$,
 \be\label{eq:l2.11} 
  \int_T^t \frac{b^2(t')+\int \varepsilon^2(t',y) e^{-\frac{|y|}{10}}dy}{\lambda^3(t')} dt'
  +\mathcal N(t) + \|\varepsilon(t)\|_{H^1}^2 +|b(t)| \lesssim \lambda^2(t) E_0.
\ee
Moreover,
\be
\label{monotolambda}
\hbox{for all $t,t'\in (T,t_1)$,} \quad t<t' \hbox{ implies } \lambda(t')\geq \frac1{24} \lambda(t),
\ee  
\be
  \label{controlright} 
\hbox{for all $t\in (T,t_1)$, for all  $y >0$,}  \quad |\e(t,y)| \lesssim  e^{-\frac y{2000} }.
\ee
\end{lemma}
 
\begin{proof}
 Let $u(t)\in \mathcal C((T,0],H^1)$ be  a general backwards minimal mass blow up 
solution
  defined on $(T,0]$  and blowing  up   in finite or infinite time
  \footnote{Note that the uniqueness statement in Theorem 1.3 (ii) concerns only finite time blow up solutions. Actually, we also prove in this paper the nonexistence of minimal mass solutions blowing up at infinity (see Proposition 4.1 and Section 5). However, we do not treat the case of global minimal mass solutions blowing up only on a subsequence of time.} 
  :  $-\infty\leq T<0$: 
  \be
  \label{assumptionblowup}
  \lim_{t\to T}\|u_x(t)\|_{L^2}= +\infty.
  \ee
From standard concentration compactness arguments\footnote{See for example the lecture notes \cite{Rzurich}.} and using the  mass and energy conservations, 
either $u(t)$ or $-u(t)$ satisfies \eqref{decompo1} and \eqref{hypeprochien} for $t$ close to $T$, with in addition $$\|\e_1(t)\|_{H^1} \to 0 \ \ \mbox{as}\ \ 
t\to T$$ thanks to the minimal mass assumption.
Therefore,  possibly considering $-u(t)$ instead of $u(t)$, there exists $t_0>T$ such that the solution $u(t)$ admits on    $(T,t_0]$    a decomposition given by  Lemma \ref{le:2}:
\begin{equation}\label{tiret}
  u(t,x)=\frac{1}{\lambda^{\frac 12}(t)}(Q_{b(t)}+\e)\left(t,\frac{x-x(t)}{\lambda(t)}\right)
\end{equation}
   with 
 \be\label{encoreinit}
 \forall t\in (T,t_0],\quad 
 |b(t)|+\|\e(t)\|_{H^1} \leq \alpha^*,
 \ee
 where $\alpha^*>0$ is any small given constant.
With this decomposition, the (finite or infinite time) blow up assumption \fref{assumptionblowup} is equivalent to:
$
\lim_{t\to T}\lambda(s)=0,
$
and by \fref{twoboundmini},
 \be
 b(t) <0, \quad E_0>0,\quad
 \label{cenononeo}
 |b(t)|+\|\e(t)\|^2_{H^1}\lesssim \lambda^2(t)E_0.
 \ee
 
Now, we recall results from \cite{MMduke}. First, recall that the solution $u(t,x)$ is decomposed in a different way in \cite{MMduke}, Lemma 1. Indeed, there exist  $C^1$ functions $\tilde \lambda$
and $\tilde x$ such that
$$
\tilde \varepsilon(t,y)=\tilde \lambda^{\frac 12}(t) u(t,\tilde \lambda(t) y+\tilde x(t)) -Q(y)
$$
satisfies the orthogonality conditions
$$
\int \Lambda Q(y) \tilde \varepsilon(t,y)dy=\int y \Lambda Q(y) \tilde \varepsilon(t,y)dy=0.
$$
Note that one easily compares this   decomposition with \eqref{tiret},
in particular, combining the orthogonality conditions   of
$\varepsilon$ and $\tilde{\varepsilon}$, one obtains
\begin{equation}\label{eq:deccomp3}
 \left|1-\frac{\lambda(t)}{\tilde \lambda(t)}\right|+ |b(t)|\lesssim
 \left(\int \tilde \e^2(t) e^{-\frac{|y|}{10}}dy\right)^{\frac 12}.
\end{equation}
Under the general assumptions of Lemma \ref{lemmadecay}, we now claim  that
for some $T<t_1<t_0$,
\begin{equation}\label{nouveautruc}
\hbox{for all $t',t\in (T,t_1)$, if $t\leq t'$ then $\tilde \lambda(t)\leq 4 \tilde \lambda(t')$
}\end{equation}
and
\be
  \label{controlrighttilde} 
\hbox{for all $t\in (T,t_1)$, for all  $y >0$,}  \quad |\tilde \e(t,y)| \lesssim  e^{-\frac y{1000} }.
\ee

To prove \eqref{nouveautruc} and \eqref{controlrighttilde}, we invoke the arguments of Section 4 in \cite{MMduke}. 
Recall that the main result of \cite{MMduke}, stated in Theorem \ref{thmmduke} of the present paper, claims forward global existence for minimal mass solutions under the decay assumption \eqref{conditiondecay}.
Unlike Section~3, based on the decay assumption on the initial data, Section 4 of \cite{MMduke} does not make use of this assumption, except when asserting that blow up occurs in finite time. 
At this point, it is important to note that here time is reversed with respect to \cite{MMduke}, thus left and right in space are also reversed 
(recall that if $u(t,x)$ is solution of \eqref{kdv}, then $u(-t,-x)$ is also solution of \eqref{kdv}).
First, using Lemma 4 in \cite{MMduke}, one obtains   uniform exponential decay on the right in space, on a special sequence of time $t_n\to T$,
\begin{equation}\label{zero40}
  \hbox{ for all  $y >0$,}  \quad |\tilde \e(t_n,y)| \lesssim  e^{-\frac y{1000} }.
\end{equation}
Then,  combining Step 2 of the proof of Proposition 2 (page 401) of \cite{MMduke},
we obtain \eqref{nouveautruc}, i.e. the almost monotonicity property of $\tilde \lambda$ and in return, using Lemma 4 again, the decay property on the right \eqref{controlrighttilde}  for all time.\\

Note that from \eqref{cenononeo} and \eqref{controlrighttilde}, we have
\begin{equation}\label{zero10}
 \int_{y>0} y^{10}\tilde \varepsilon^2(t,y) dy \rightarrow 0 \quad \hbox{as $t\rightarrow T$}.
\end{equation}

Now, using further algebra developed page 405 of \cite{MMduke}, we claim that,
 for all $T<t_2<t_1$,
 \begin{equation}\label{duk2bis}
\int_{t_2}^{t_1}\frac{\int \tilde{\varepsilon}^2 e^{-\frac{|y|}{10^4}}dy}{\tilde \lambda^3} dt
\lesssim \tilde\lambda^2(t_1) E_0.
\end{equation}
Indeed, it is proved there that for all $T<t_2<t_1$,
\begin{equation}\label{duk2}
\int_{t_2}^{t_1}\frac{\int \tilde{\varepsilon}^2 e^{-\frac{|y|}{10^4}}dy}{\tilde \lambda^3} dt
\lesssim  \int \varepsilon^2(t_1)+\int \varepsilon^2(t_2)
+  \int_{t_2}^{t_1} \frac{\left(\int \tilde\varepsilon Q\right)^2}{\tilde \lambda^3} dt,
\end{equation}
and that there exists $\bar\lambda(t)$ (related to yet another decomposition of $u(t,x)$ which requires the decay \eqref{controlrighttilde}),  with $\bar \lambda(t)\approx \tilde \lambda(t)$ such that the following holds, for a universal constant $c_0>0$,
\begin{equation}\label{duk1}
-E_0 \lambda^2 \lesssim \int \tilde \varepsilon Q<0,\quad
\left|2 \int \tilde \varepsilon Q +c_0 \bar\lambda^2 {\bar\lambda}_t
\right| \lesssim \left(\int \tilde \varepsilon^2 e^{-\frac {|y|}{4}}\right)^{\frac 34}.
\end{equation}
Since $\int_{t_2}^{t_1} \bar \lambda \bar \lambda_t dt\lesssim\bar \lambda^2(t_1)\lesssim\tilde \lambda^2(t_1)$, we obtain
\eqref{duk2bis} by integration.\\

Passing to the limit as $t_2\to T$ in \eqref{duk2bis}, we obtain
\begin{equation}\label{duk3}
\int_{T}^{t_1}\frac{\int \tilde{\varepsilon}^2 e^{-\frac{|y|}{10^4}}dy}{\tilde \lambda^3} dt
\lesssim  \tilde \lambda^2(t_1) E_0\lesssim    \lambda^2(t_1) E_0,
\end{equation}
and thus, using \eqref{eq:deccomp3},
\begin{equation}\label{duk4}
\int_{T}^{t_1}\frac{\int {\varepsilon}^2 e^{-\frac{2|y|}{10^4}}dy}{\lambda^3} dt
\lesssim      \lambda^2(t_1) E_0.
\end{equation}
By \eqref{eq:2003}, we have $\frac {b^2}{\lambda^3} \leq -b_t+\frac C{\lambda^3}  \int {\varepsilon}^2e^{-\frac{|y|}{10}}dy$
and  thus  $$\int_{T}^{t_1}\frac{ b^2(t)}{\lambda^3(t)} dt \lesssim  \lambda^2(t_1)E_0
+C\int_{T}^{t_1}  \frac{\int {\varepsilon}^2e^{-\frac{|y|}{10}}}{\lambda^3(t)}dy
\lesssim  \lambda^2(t_1) E_0.$$ 
Now,   we claim
\begin{equation}\label{duk5}
  \int_{y>1} y^{10} \varepsilon^2(t,y)dy \lesssim \lambda^2(t) E_0.
\end{equation}
Indeed, consider a smooth function 
$$\varphi_{10}(y)=\left\{\begin{array}{ll}
                          0&\hbox{for $y\leq 0$},\\
                          y^{10}&\hbox{for $y\geq1$},
                         \end{array}\right.
\quad
\hbox{with $\varphi_{10}'\geq 0$.}
$$
Using the computations of the proof of  Lemma 3.7 in \cite{MMR1} on $\tilde \varepsilon$ (the computations for the decomposition $(\tilde \varepsilon,\tilde \lambda,\tilde x)$ are actually simpler, since they correspond to the choice $b=0$), we obtain
$$
\frac{1}{\tilde\lambda^{10}}\frac{d}{dt}\left(\tilde\lambda^{10}\int \varphi_{10} \tilde\varepsilon^2\right)
\lesssim \frac{\int {\tilde\varepsilon}^2 e^{-\frac{|y|}{10}}dy}{\tilde\lambda^3}
$$
and thus, using \eqref{monotolambda}, \eqref{duk3} 
$$\int \varphi_{10} \tilde\varepsilon^2(t_1,y)dy
\lesssim \int \varphi_{10} \tilde\varepsilon^2(t_2,y)dy+\lambda^2(t_1) E_0.$$
Passing to the limit as $t_2\to T$ and using
\eqref{zero10}, we obtain for all $T<t<t_0$,
$$
\int_{y>1} y^{10} \tilde \varepsilon^2(t,y)dy \lesssim \lambda^2(t) E_0.
$$
Finally, by \eqref{eq:deccomp3} and \eqref{cenononeo},
we obtain \eqref{duk5}.
\end{proof}

\section{Construction of a minimal element}
\label{sectionconstruction}


This section is devoted to the proof of the existence of a minimal blow up element. We propose a strategy of proof slightly different from the recent approach developed for the construction of non dispersive solutions in \cite{Me0,Ma1,MMnls,Cotekdv,MMC,CoteZaag,KMR, RS2010, KLR}, mainly to  prepare the analysis of the (Exit) regime in Theorem \ref{PR:1}, see also section \ref{estrescAAbis} and Remark \ref{rkk} below.\\
The strategy of the proof goes as follows. We consider a {\it well prepared} sequence of initial data $(u_n)$ with $$\|u_n\|_{L^2} < \|Q\|_{L^2} \hbox{ and 
 $u_n(0)\to Q$ in $H^1$.}
 $$
By Theorem \ref{thmmoe}, such solutions are in the (Exit) scenario and we denote by $t_n^*>0$ the corresponding exit time. The estimates extracted from \cite{MMR1} allow for a complete dynamical description of the (Exit) regime and in particular the defocusing structure of the solution at $t^*_n$. This explicit detailed knowledge allows us to renormalize the flow and extract in the limit $n\to +\infty$ a solution $v\in \matchal C((t^*,0],H^1)$ which blows up at time $t^*<0$ and has subcritical mass $\|v\|_{L^2}\leq \|Q\|_{L^2}$. But then the global wellposedness below the ground state mass implies $\|v\|_{L^2}=\|Q\|_{L^2}$ and $v$ is an $H^1$ minimal mass blow up element. \\

{\bf step 1} Well prepared data. Let $u_n(0)=Q_{b_n(0)}$, where $b_n(0)=-\frac 1 n$ so that
$$
u_n(0)\in \mathcal A\subset H^1, \quad
u_n(0)\to Q \quad \hbox{in $H^1$ as $n\to +\infty$}.
$$
By \eqref{eq:204}, we have $\int u_n^2(0) <\int Q^2$. In particular, from energy and mass conservation, and the Gagliardo-Nirenberg's inequality \eqref{gn}, the solution $u_n(t)$ is global. We take $n>0$ large enough and we apply  Theorem \ref{thmmoe}. The solution $u_n$ being global, the (Blow up) scenario is ruled out. The solution cannot converge locally to a solitary wave because of mass conservation and the strictly subcritical mass assumption, hence (Soliton) is also ruled out.  Hence (Exit) holds and we define the \emph{exit time} (related to the constant $\alpha^*$ of Theorem \ref{thmmoe}) by 
$$
t^*_n = \sup\{t>0, \hbox{ such that } \forall t'\in [0,t], \ u_n(t')\in \mathcal{T}_{\alpha^*}\}.
$$
Note that $t_n^*\to +\infty$ as $n\to +\infty$  from the continuous dependence of the solution of \eqref{kdv} with respect to the initial data, and the fact that $Q(x-t)$ is solution of \eqref{kdv}.
\\

Now, we use refined information given in the (Exit) case by Proposition 4.1 in \cite{MMR1}.
In particular, we know that $u_n(t)$ satisfies \eqref{hypeprochien} and has a  decomposition $(\l_n,x_n,b_n,\e_n)$ as in Lemma \ref{le:2} on $[0,t^*_n]$. Moreover,  (H1)--(H3) are satisfied on $[0,t_n^*]$, and by definition of $u_n(0)$,
\be\label{initun}
\l_n(0)=1, \ \ x_n(0)=0, \ \ b_n(0)=-\frac 1n, \ \ \e_n(0)=0.
\ee
In addition, from the proof of Proposition 4.1 (see (4.41) in \cite{MMR1}), we also have 
\be\label{lambdacontrol}
\hbox{for all } 0\leq t_1\leq t_2\leq t_n^*,\quad
\l_n(t_2) \geq \frac 12 \l_n(t_1).
\ee
Note also that  by continuity in time and the definition of $t^*_n$,
\be\label{exi1b}\inf_{\l_0>0,\ x_0\in \RR}\|u_n(t_n^*)-\l_0^{-1/2} Q(\l_0^{-1}(.+x_0))\|_{L^2} = \alpha^*,
\ee
and
\be\label{exi1}
\alpha^* \leq \|u_n(t_n^*)-\l_n^{-1/2}(t_n^*) Q(\l_n^{-1}(t_n^*)(.+x_n(t_n^*))\|_{L^2} \leq \delta(\alpha^*).
\ee
\\

{\bf step 2} Structure of the defocalized bubble.
From Lemma \ref{le:oubli}, \eqref{initun} and \eqref{exi1}, we claim: 

\begin{lemma}\label{le:2.4}\quad
\\
{\rm (i) Estimates on $[0,t_n^*]$.}
\bea
\forall t\in [0,t_n^*],  &&
\frac {1-\delta(\alpha^*)}{n} \leq  - \frac { b_n(t)}{\l_n^2(t)} \leq \frac {1+\delta(\alpha^*)}n,\label{TRT2}\\
&& \|\e_n(t)\|_{H^1}^2 \lesssim \frac {\l^2_n(t)}{n}   \lesssim  \delta(\alpha^*).\label{TRT3}
\eea
{\rm (ii) Estimates at $t_n^*$.} For all $n$,
\be\label{TRTbis}
 (\alpha^*)^2 \lesssim  - b_n(t_n^*)\lesssim \delta(\alpha^*),
\ee
\be\label{TRT}
(\alpha^*)^2 \lesssim\int \e_n^2(t_n^*) \lesssim \delta(\alpha^*), \quad  {(\alpha^*)^2}   \lesssim \frac {\l_n^2(t_n^*)}n  \lesssim  {\delta(\alpha^*)}  ,
\ee
\be\label{borne} 
0<c(\alpha^*) \leq \frac {t_n^*} {\l_n^3(t_n^*)} \leq C(\alpha^*).
\ee
{\rm (iii)} Control of the dynamics on $[0,t_n^*]$.
\be\label{eqln}
  -(1-\delta(\alpha^*))\frac {b_n(t_n^*)}{\l_n^2(t_n^*)}  \leq  {(\l_{0n})_t(t)} \leq - (1+\delta(\alpha^*))\frac {b_n(t_n^*)}{\l_n^2(t_n^*)} .
\ee
\end{lemma}
\begin{proof}
Using \eqref{NewL2} and \eqref{NewE} at $t=0$, we obtain
$$
\int u_n^2(0)=\int Q^2-\frac{2}{n}\int PQ+O(n^{-9/8}),
$$
$$
2E(u_n(0))=\frac{2}{n}\int PQ+O(n^{-9/8}).
$$
Combining the conservation of the $L^2$ norm, the conservation of energy and 
\eqref{NewL2}, \eqref{NewE}, we obtain at  any $t\in(0,t_n^*)$,
\begin{equation}\label{bepsn}
|b_n|\lesssim \int \varepsilon^2 +\frac 1n. 
\end{equation}
and
$$
(L\varepsilon_n(t),\varepsilon_n(t))
=2\lambda_n^2 E(u_n(0))+\int u_n^2(0)-\int Q^2+O\left(n^{-\frac98}+\|\varepsilon_n(t)\|_{H^1}^{\frac94}\right),
$$
Thus, by $\|\varepsilon_n(t)\|_{H^1}^2\lesssim
(L\varepsilon_n(t),\varepsilon_n(t))$, we obtain
\be
\|\e_n(t)\|_{H^1}^2 \lesssim \frac{\l_n^2(t)}{n}.
\ee

Next, recall \eqref{exi1}
$$
\alpha^* \leq \|b_n(t_n^*) \chi_{b_n(t_n^*)} P + \e_n(t_n^*)\|_{L^2}\leq \delta(\alpha^*),
$$
and use \eqref{ePb} and \eqref{bepsn} to obtain 
\be\label{tyt}
 (\alpha^*)^2  \lesssim \int \e_n^2(t_n^*)   \lesssim \delta(\alpha^*),
\quad
 (\alpha^*)^2  \lesssim - b_n(t_n^*)   \lesssim \delta(\alpha^*),
\ee

Now, we use the dynamical information given by the rigidity property \eqref{conrolbintegre} and the initialization \eqref{initun}: for all $t\in [0,t_n^*]$,
\be\label{3.40}
(1+\delta(\alpha^*)) b_n(0) \leq \frac {b_n(t)}{\l_n^2(t)} \leq (1-\delta(\alpha^*))    {b_n(0)}  
\ee
 and thus by \eqref{tyt},
 \be
 \label{ebsiebevbib}
  {(\alpha^*)^2} \lesssim \frac {\l_n^2(t_n^*)} {n} \lesssim  {\delta(\alpha^*)} .
\ee

Finally, let us prove \eqref{eqln} and \eqref{borne}. By \eqref{dvnkoenneoneor}, we have
$$
\left| {(\l_{0n})_t}\frac{\l_n} {\l_{0n}} + \frac {b_n}{\l_n ^2} \right| 
\lesssim \frac {\int \e_n^2 e^{-\frac {|y|}{10}} }{\l_n^2} + \frac {|b_n|}{\l_n^2} (\mathcal{N}_n^{\frac 12} + |b_n|)
\lesssim  \frac {1}{\l_n^2} (\mathcal{N}_n  + |b_n|^2).
$$
By  
 \eqref{estfondamentalebis}, \eqref{initun}, and then \eqref{3.40}, we have
$$
\frac {\mathcal{N}_n(t)}{\l_n^2(t)}\lesssim \frac {|b_n(t)|^3}{\l_n^2(t)} +
\frac {|b_n(0)|^3}{\l_n^2(0)} \lesssim \delta(\alpha^*) \frac {|b_n(t_n^*)|}{\l_n^2(t_n^*)}.
$$
Thus, again by \eqref{3.40}, 
$$  -(1-\delta(\alpha^*))\frac {b_n(t_n^*)}{\l_n^2(t_n^*)}  \leq  {(\l_{0n})_t(t)} \leq - (1+\delta(\alpha^*))\frac {b_n(t_n^*)}{\l_n^2(t_n^*)} ,
$$
which is \eqref{eqln}.
We integrate on $[0,t_n^*]$ and then divide by $\l_{0,n}(t_n^*)$ to obtain:
\be\label{ggto20}
-t_n^* \frac {b_n(t_n^*)}{\l_n^3(t_n^*)} (1-\delta(\alpha^*)) \leq   1   - \frac 1 {\l_{0,n}(t_n^*)}
\leq  -t_n^* \frac {b_n(t_n^*)}{\l_n^3(t_n^*)}(1+\delta(\alpha^*)).
\ee
 Hence using \fref{ebsiebevbib} and $\frac{\l_{n}}{\l_{0,n}}\lesssim \mathcal N_n^{\frac12}\lesssim \delta(\alpha^*)$, we obtain
 \be\label{OKbash} 
- \frac {1-\delta(\alpha^*)}{b_n(t_n^*)}\leq \frac {t_n^*}  {\l_n^3(t_n^*)} 
\leq- \frac {1+\delta(\alpha^*)}{b_n(t_n^*)},
\ee
which together with \fref{tyt} implies \eqref{borne}.
\end{proof}
 
{\bf step 3} Renormalization and extraction of the limit. Let:
$$\forall \tau\in [- \frac {t_n^*}{\l_n^3(t_n^*)},0], \quad t_\tau= t_n^*+\tau \l_n^3(t_n^*),$$
\bea\label{eq:3.4}
v_n(\tau,x)&&
= \l_n^{\frac 12}(t_n^*) u_n\left(t_\tau,{\l_n(t_n^*)} x+x(t_n^*)\right)
\\ &&= \frac {\l_n^{\frac 12}(t_n^*)}{\l_n^{\frac 12}(t_\tau)}
(Q_{b_n(t_\tau)}+\e_n)\left(t_\tau, 
\frac {{\l_n(t_n^*)} }{\l_n(t_\tau)} x + \frac { x(t_n^*)-x(t_\tau)}{\l_n(t_\tau)} \right),
\eea
so that $v_n$ is solution of \eqref{kdv} and  belongs to the tube $\mathcal T_{\alpha^*}$ for $\tau\in [- \frac {t_n^*}{\l_n^3(t_n^*)},0]$. Moreover, its  decomposition $(\lambda_{v_n} ,x_{v_n},\e_{v_n})$   satisfies
 on $[- \frac {t_n^*}{\l_n^3(t_n^*)},0]$ 
\begin{align}
\label{neoieohfe}\l_{v_n}(\tau) = \frac {\l_n(t_\tau)}{\l_n(t_n^*)},  \
x_{v_n}(\tau)= \frac  {x_n(t_\tau)-x_n(t_n^*)}{\l_n(t_n^*)} ,\
b_{v_n}(\tau) = b_n(t_\tau),\
\e_{v_n}(\tau) = \e_n(t_\tau).
\end{align}
By \eqref{TRTbis}, \eqref{TRT} and \eqref{TRT2}, we have
$$ \forall \tau \in [- \frac {t_n^* }{\l_n^3(t_n^*)},0], \quad \|\e_{v_n}(\tau)\|_{H^1}^2
  \lesssim \delta(\alpha^*). 
$$
$$
\l_{v_n}(0)=1, \ \ x_{v_n}(0)=0, \ \ (\alpha^*)^2 \lesssim -b_{v_n}(0) \leq \delta(\alpha^*).
$$
Therefore, there exists a subsequence of $(v_n )$, which we will still denote by $(v_n)$, and $v(0)\in H^1$ such that
$$
v_n(0)\rightharpoonup v(0) \hbox{ in $H^1$ weak, and $\|v(0)-Q\|_{H^1} \lesssim \delta(\alpha^*)$}.
$$
and, by \eqref{TRT} and \eqref{borne}, 
 $$
 \tau_n^*=- \frac {t_n^*}{\l_n^3(t_n^*)} \to  -\tau^{*}, \quad \tau^*>0,
 \quad -b_n(t_n^*)\to b^*>0. $$
 Moreover, by \eqref{OKbash},
\be\label{oubli}
 \frac {1-\delta(\alpha^*)}{b^*}\leq 
\tau^*  \leq  \frac {1+\delta(\alpha^*)}{b^*}.
\ee
We let $v(\tau)$ be the backward $H^1$ solution of \eqref{kdv} with initial data $v(0)$ at $\tau=0$.\\

{\bf step 4} Minimal mass blow up. We claim that $v$ is a minimal mass blow up element $ \|v\|_{L^2}=\|Q\|_{L^2}$ which blows up in finite negative time $-\tau^*$ with for $\tau$ close enough to $-\tau^*$:
\be\label{approxrate}
 \frac { (1- \delta(\alpha^*))} {1 +\frac \tau{\tau^*}} \leq \|v_x(\tau)\|_{L^2} 
 \leq  \frac { (1+ \delta(\alpha^*))}{1 +\frac \tau{\tau^*}}.
\ee
Indeed, we integrate \eqref{eqln} and obtain for $t\in [0,t_n^*]$, $n$ large enough,
$$
  \frac {b^* t}{\l_n^2(t_n^*)} (1-\delta(\alpha^*))\leq
\l_{0,n}(t) - \l_{0,n}(0)\leq\frac {b^* t}{\l_n^2(t_n^*)} (1+\delta(\alpha^*)) .
$$
We conclude from \fref{neoieohfe} and the definition of $\tau_n^*$: for all $\tau\in [\tau_n^*,0]$,
\be\label{estlvn}
    {b^*} (\tau_n^*+\tau) (1-\delta(\alpha^*))\leq
\l_{0,v_n}(\tau) - \l_{0,v_n}(\tau_n^*)   \leq    {b^*} (\tau_n^*+\tau) (1+\delta(\alpha^*)) .
\ee
Let   $\tau_0 \in (-\tau^*,0)$.
From \fref{TRT}, \fref{neoieohfe}
and $\varepsilon_{v_n}(\tau_n^*)=\varepsilon_n(0)$, we have
$$\l_{0,v_n}(\tau_n^*)=\frac{\lambda_n(0)}{\lambda_n(t_n^*)}\to 0\ \ \mbox{as}\ \ n\to+\infty$$ and $(1-\delta(\alpha^*))\leq \frac {\l_{0,v_n}}{\l_{v_n}}\leq (1+\delta(\alpha^*))$.
Thus, we conclude from \fref{estlvn} that for $n$ large enough depending on $\tau_0$,
 $$\forall \tau \in [\tau_0,0],\quad
 {b^*} (\tau^*+\tau) (1-\delta(\alpha^*))\leq \l_{v_n}(\tau) \leq b^*  (\tau^*+\tau) (1+\delta(\alpha^*)) ,$$ 
 and
$$\frac 12 {b^*}(\tau^*+\tau_0) \leq  \l_{v_n}(\tau).$$ 
 It follows from Lemma \ref{le:2.7} that $v(\tau)$ is well-defined 
 and $\l_{v_n}(\tau)\to \l_v(\tau)$  on $[\tau_0,0]$.
In particular, $v$ exists on $(-\tau^*,0]$ and for all $\tau \in (-\tau^*,0]$,
$$
{b^*} (\tau^*+\tau) (1-\delta(\alpha^*))\leq \l_{v}(\tau) \leq b^*  (\tau^*+\tau) (1+\delta(\alpha^*)),
$$
which together with \eqref{oubli} implies \eqref{approxrate}. Finally, we have by weak $H^1$ convergence  $\int v^2(0) \leq \lim_{n\to \infty} \int v_n^2(0)=\int Q^2$, and since $v$ blows up in finite time, $\int v^2(0)=\int Q^2$.\\
This concludes the proof of the existence of the minimal element.

\begin{remark}
\label{rkk} We may rewrite this proof by saying that understanding the minimal mass blow up scenario is in some sense equivalent to understanding how subcritical solutions initially near the ground state move away from the ground state and start {\it defocusing}, and here the sharp knowledge of the speed of defocusing is fundamental for the proof. Another approach for the construction of the minimal blow up element in the continuation of \cite{Me0,Ma1,MMnls,Cotekdv,MMC,CoteZaag,KMR,RS2010} would have been to take the initial data $Q_{b(t_n)}$ at some time $t_n\downarrow 0$ with $b(t_n)=t_n$ and to obtain uniform $H^1$ bounds on the corresponding forward solution $u_n(t)$ to \fref{kdv} at a time $t_0>0$ independent of $n$ using the monotonicity machinery of Proposition \ref{propasymtp} and Lemma \ref{le:oubli}. It is not clear to us whether a direct fixed point approach as in \cite{BW,KST,KST2,DM2,DM} is applicable here due to the poor localization in space of the minimal element.
\end{remark}


\section{Sharp description of minimal mass blow up}

  
 We now turn to the proof of uniqueness in $H^1$ of the minimal element. Let us stress the fact that uniqueness is always a delicate problem, in particular in the absence of suitable symmetries as in \cite{MMnls}. As in \cite{RS2010}, the first crucial information is to derive the blow up speed for all minimal elements, and here we shall use the a priori localization in space of minimal elements given by Lemma \ref{lemmadecay} which allows us to use the  monotonicity tools Proposition \ref{propasymtp} and Lemma \ref{le:oubli}. Once the minimal mass blow up regime is sufficiently well described, we may rerun the analysis of Proposition \ref{propasymtp} for the difference of two such bubbles and conclude that they are equal, this is done in section \ref{secunique}.
 

\subsection{Finite time blow up and blow up speed for minimal mass blow up solutions}


Our aim in this section is to derive sharp qualitative bounds on minimal mass blow up solutions, improving general results stated in Lemma~\ref{lemmadecay}.  In particular, we prove that the blow up time is finite, $T>-\infty$, and we specify  the blow up speed and the behavior of the concentration point which are essential preliminary information on the singularity formation. Note that the additional information below requires the sharper analysis of \cite{MMR1} and cannot be derived from \cite{MMduke}.
We consider $u(t)$ a minimal mass blow up solution and in  the setting of Lemma \ref{lemmadecay}, we introduce the rescaled time 
\be\label{st}s(t)=- \int_{t}^{t_0}\frac{ds}{\l^3}.
\ee
Recall that $s(T)=-\infty$ from a standard argument (see e.g. \cite{MMannals}).
 
\begin{proposition}[Sharp bounds]
\label{sharpbounds}
Let $u(t)$ be a solution of \eqref{kdv} defined on $(T,0]$, which blows up backwards in finite or infinite time $-\infty\leq T<0$. Assume
$$
\int u^2(0) = \int Q^2.
$$
\noindent {\em (i) Finite time blow up}: There holds $$T>-\infty.$$
{\em (ii) Sharp controls near blow up time}:  there exist universal constants
$c_\l$, $c_x$, $c_b$ and 
 $\ell^*=\ell^*(u)<0$, $x^*=x^*(u)\in  \RR$ such that, for $t$ close to $T$,
\be
\label{loilambda}
\lambda(t)=|\ell^*|(t-T)+c_\l|\ell^*|^4(t-T)^3+ O\left[(t-T)^4\right],
\ee
\be
\label{loix}
x(t)=-\frac{1}{(\ell^*)^2(t-T)}+x^*+ c_x\ell^* (t-T)+O\left[(t-T)^2\right],
\ee
\be
\label{loib}
\frac{b(t)}{\l^2(t)}=\ell^*+c_b(\ell^*)^4 (t-T)^2+O\left[(t-T)^3\right],
\ee
\be
\label{controlnorme}
\mathcal N(t)\lesssim (t-T)^6.
\ee
{\em (iii)  Estimates  in rescaled time}: for $-s$ large,
\be\label{estrescbis}
\|\e(s)\|_{L^\infty}^2 \lesssim \|\e(s)\|_{H^1}^2 \lesssim \l^2(s) \lesssim \frac 1{|s|},
\ee
\be\label{estresctri}
  \mathcal N(s)+   \int_{-\infty}^s \int (\e_y^2 + \e^2)(s') \varphi_B' ds' \lesssim \frac 1{|s|^3},\ee
  \be\label{estresc}
   \left|\lsl+b\right| + \left| \xsl-1\right| \lesssim
\frac 1 {|s|^{\frac 32}},\quad 
|b_s|\lesssim \frac 1{|s|^2},
\ee
\be
\label{devpb}
b(s)=\frac{1}{2s}+\frac{c_1^*\log |s|}{s^2}+\frac{c_2^*}{s^2}+o\left(\frac1{s^2}\right)\  \mbox{as} \ s\to -\infty
\ee
for some universal constants $(c_1^*,c_2^*)\in \RR\times \RR$.\\
{\em (iv) Global forward behavior}: the solution is globally defined for $t>T$,
$u\in \mathcal C((0,+\infty)\times \RR)$ and for some $C(t)$, $\gamma(t)>0$,
\begin{equation}\label{decayu}
\forall t>T,\ \forall x>0,\quad |u(t,x)|\leq C(t) e^{-\gamma(t) x}.
\end{equation}
{\rm (v) Time decay of weighted Sobolev  norms:} 
For $|s|$ large,
\be\label{decayscaling}
\int \e^2(s,y) e^{\l(s) y} dy \lesssim \frac 1{|s|^2}.
\ee
For all $\frac 9B\leq \omega\leq \frac 1{10}$, for $|s|$ large,
\begin{equation}\label{sobolev}
  \sum_{k=0}^3\int \left(\partial_y^k \e\right)^2(s,y) e^{\omega  y  } dy +
\int_{-\infty}^s \sum_{k=0}^4\int \left(\partial_y^k \e\right)^2(s',y) e^{\omega  y  } dyds'
\lesssim \frac 1{|s|},
\end{equation}
\be\label{sobolevbis}
\| \left( (\e_i)_{yy}^2+(\e_i)_y^2  \right)(s ) e^{\omega  y  }\|_{L^\infty}
\lesssim \frac 1{|s|}.
\ee
\end{proposition}
\begin{remark}\label{Rk4.1}
The constant $\ell^*$ in \fref{loilambda} depends on the solution and the scaling $u(t,x)\mapsto u_{\lambda_0}(t,x)= \l_0^{\frac 12} u(\l_0^3 t, \l_0 x)$ leads to 
\be\label{scal4.1}
\ell^*(u_{\l_0}) = \l_0^2 \ell^*(u).
\ee
\end{remark}

\begin{proof}[Proof of Proposition \ref{sharpbounds}]
From Lemma~\ref{lemmadecay}, $E(u(t))>0$. 
Using the scaling invariance of the (gKdV) equation, we consider
the solution 
$$  u_{\lambda_0}(t,x)  =\l_0^{\frac 12} u(\l_0^3 t, \l_0 x) ,$$
where $ \lambda_0>0$ is chosen so that $E( u_{\lambda_0}(t))\ll\kappa^*$, $\kappa^*$ being the small constant in Proposition \ref{propasymtp}.
We work on $u_{\lambda_0}$ instead of working on $u$, all statement being scaling invariant.
Hereafter, we denote $u_{\lambda_0}$ simply by $u(t)$.

\medskip

{\bf step 1} Entering the monotonicity regime. 
Note first that from Lemma~\ref{lemmadecay} and $E_0 \ll \kappa^*$,  (H1), (H2), (H3)  hold on $(- \infty,s_0]$, for $-s_0$ large enough.
The solution is therefore in the monotonicity regime  of Proposition \ref{propasymtp} and Lemma \ref{le:oubli} on $(-\infty,s_0]$.

\medskip

{\bf step 2} Rigidity and blow up speed.  We claim the key non degeneracy relation:
\be
\label{approxbl}
\mathcal N \lesssim O(\l^6),\quad 
\left|\frac{b}{\l^2}-\ell^* - \frac {c_0} 2 \frac {b^2}{\l^2}\right| \lesssim O(\l^3)
\ee
for some constant $\ell^*>0$.\\
Let $C^*>0$ be the universal constant in \fref{conrolbintegre}. Let us first remark that there exists a sequence $s_n\to -\infty$ such that \be\label{defsn} \forall n\geq 1,  \ \ b(s_n)\leq -C^* \int \left( \e_y^2+ \e^2\right)(s_n)\varphi_B' .\ee

Indeed, assume for the sake of contradiction that there exists a time $s^*\leq s_0$ such that 
(recall that $b<0$)
\be
\label{cneneoen}\forall s<s^*, \ \ |b(s)|\leq C^* \int \left( \e_y^2+ \e^2\right)(s)\varphi_B'.
\ee
Thus, \eqref{dvnkoenneoneor} implies
$$
\left|\frac{(\lambda_0)_s}{\lambda_0}\right|\lesssim
\int \varepsilon^2e^{-\frac{|y|}{10}}+|b|\lesssim\int \left( \e_y^2+ \e^2\right)(s)\varphi_B'
,
$$
where $\lambda_0(s)=\lambda(s)(1-J_1(s))^2$.
Using \fref{estfondamentale}, we obtain $$\forall s <s_2^*, \ \ \left|\log\left(\frac{\lambda_0(s)}{\lambda_0(s_2^*)}\right)\right|\lesssim \int_{s}^{s^*}\int \left( \e_y^2+ \e^2\right)(s)\varphi_B' ds\lesssim 1$$ 
but together with 
\begin{equation}\label{eq:JJJ}
 \left|\frac{\lambda}{\lambda_0}-1\right|\lesssim |J_1|
 \lesssim \mathcal N^{\frac 12}
\end{equation}
and \eqref{eq:l2.11},
this  contradicts the blow up assumption: $$\lambda(s)\to 0\ \ \mbox{as}\ \ s\to -\infty.$$ The sign $b<0$ concludes the proof of \eqref{defsn}.\\

Inserting \eqref{defsn} in   \fref{conrolbintegre} yields the rigidity:
$$\hbox{for all $n\geq 1$, for all $s$ such that  
$s_n\leq s\leq s_0$, } 2\frac{b(s_n)}{\l^2(s_n)}\leq \frac{b(s)}{\l^2(s)}\leq \frac{b(s_n)}{2\l^2(s_n)}.$$
We conclude using $b<0$ and $\l(s_0)=1$ that for all $s<s_0$,
\be
\label{cneioenoe}
4 b(s_0)\leq\frac{b(s)}{\l^2(s)}\leq  \frac {b(s_0)} 4<0.
\ee 
By \eqref{aausmptoblowup} and \eqref{eq:l2.11}, we have $\lim_{-\infty} \mathcal N =0$.
Using  \fref{estfondamentale}, we have, for $s_1<s_2$ 
$$
\mathcal N(s_2)+\int_{s_1}^{s_2} \int (\e_y^2 +\e^2) \varphi_B' ds \lesssim \mathcal N(s_1) + |b^3(s_1)| + |b^3(s_2)|.
$$
Thus, passing to the limit $s_1\to -\infty$, and using \eqref{cenononeo},
\be
\label{contoroel}
\mathcal N(s_2) +\int_{-\infty}^{s_2} \int (\e_y^2 +\e^2) \varphi_B' ds\lesssim   |b^3(s_2)|\lesssim \l^6(s_2).
\ee
From \eqref{eq:bl2}, we have
\be\label{bbbb}
\left| \frac d{ds}\left(\frac b{{\lambda}^2} e^{J}\right)
+ c_0 \frac {b^3}{\l^2} \right| \leq \frac{1}{\l^2}\left( \int \e^2 e^{-\frac {|y|}{10}} +  |b|^4\right).
\ee
Letting $s_1\to -\infty$ in \fref{estfondamentalebis} ensures:
\be
\label{eobeoeo}
\int_{-\infty}^{s_2}  \frac{1}{\l^2}\left( \int \e^2 e^{-\frac {|y|}{10}} +  |b|^4\right)ds  \lesssim \frac{b^3 }{\l^2 }\lesssim \l^4 .
\ee
Next, by \eqref{eq:2003}, \fref{eobeoeo}:
$$
  \frac {b^3}{\l^2}
=- \frac 12   \frac {b_sb}{\l^2} 
+ O\left(  \frac{1}{\l^2}\left( \int \e^2 e^{-\frac {|y|}{10}} +  |b|^4\right)\right),
$$
so that by integration by parts and \eqref{eq:2002},
$$
\int_{-\infty}^{s}  \frac {b^3}{\l^2}
=  -\frac 14 \frac {b^2(s)}{\l^2(s)}- \frac 12  \int_{-\infty}^s \frac {b^2\l_s}{\l^3} + O(\l^4)
=  -\frac 14 \frac {b^2(s)}{\l^2(s)}+ \frac 12  \int_{-\infty}^s \frac {b^3}{\l^2} + O(\l^4)
$$
and thus
$$
\int_{-\infty}^{s}  \frac {b^3}{\l^2}
=  -\frac 12 \frac {b^2(s)}{\l^2(s)}  + O(\l^4).
$$

It follows   by integrating \eqref{bbbb} and using  \eqref{cneioenoe} that
\be\label{limitebl2}
\lim_{s\to -\infty} \frac{b }{\l^2}(s)=\ell^*<0 ,
\ee
 and    more precisely, using $|J|\lesssim \mathcal N^{\frac 12} \lesssim \l^3$,
\be
\label{approxblbis}
\left|\frac{b}{\l^2}-\ell^* - \frac {c_0} 2 \frac {b^2}{\l^2}\right|\lesssim O(\l^4)+\frac{|b|}{\l^2}(1-e^J)\lesssim O(\l^3).
\ee

\medskip

{\bf step 3} Finite time blow-up.

From  \fref{contoroel} and  \eqref{controlj}, we have
\be
\label{cneonoenoe}
\lambda_0=\lambda+O(\l J_1)=\lambda+O(\l \mathcal N^{\frac 12})=\lambda+O(\l^4)
\ee 
and then by \fref{dvnkoenneoneor}, \eqref{approxblbis},
\begin{align*}
-\l^3\frac {(\lambda_0)_t}{\l_0}&=-\frac{(\l_0)_s}{\l_0}=b+c_1b^2  + 
O(\mathcal N+|b| \mathcal{N} ^{\frac 12}+|b|^3)
\\ &=\ell^*\lambda^2+c (\ell^*)^2 \lambda^4+ O(\l^5),
\end{align*} 
where $c$ denotes here and thereafter various universal constants.
Hence   
\be\label{lot}-(\lambda_0)_t=\ell^*+c (\ell^*)^2 \lambda_0^2+O(\l_0^3).\ee 
For $t$ close to $T$, we obtain 
$(\l_0)_t > \frac {\ell^*}2>0$ and thus 
$\lambda_0$ vanishes backwards at some finite time $$T>-\infty;$$ 
in particular, the solution blows up in finite time. Moreover,  integrating  \eqref{lot} on $(T,t]$  
  for $t>T$ close to $T$, using \eqref{cneonoenoe}, yields
$$\lambda (t)=\lambda_0(t)+O\left[(t-T)^4\right]=|\ell^*|(t-T)+c |\ell^*|^4(t-T)^3+O\left[(t-T)^4\right].$$ 
Together with \fref{contoroel}, \fref{approxbl},  this concludes the proof of \fref{loib}, \fref{loilambda}, \fref{controlnorme}. 

We now integrate the modulation equation \fref{eq:2002} for the blow up point: 
$$x_t=\frac{1}{\l^2}\frac{x_s}{\l}=\frac{1}{\l^2}\left[1+O(b^2+\mathcal N^{\frac 12})\right]=\frac{1}{\l^2}\left[1+O(\l^{3})\right]$$
 and thus using \fref{loilambda}: $$x_t(t)=\frac{1}{(\ell^*)^2(T-t)^2} - 2 c \ell^*+ O(t-T)$$ which implies \fref{loix} by integration in time.

\medbreak

{\bf step 4} Sharp estimates in   rescaled time.
From \fref{loilambda}: $$s(t)= - \int_{t}^{t_0}\frac{dt}{|\ell^*|^3(t-T)^3(1+O((T-t)^2))}=-\frac{1}{2|\ell^*|^3(t-T)^2}(1+O(t-T)).$$
From step 3 and   \eqref{eq:2002}-\eqref{eq:2003}, we thus get  the following  estimates in terms of the variable $s$:
\be
\label{boundaryn}
\lambda(s)=\frac{1+o(1)}{\sqrt{2|\ell^*s|}}, \  \ \mathcal N(s) +\int_{-\infty}^{s} \int (\e_y^2+ \e^2) \varphi_B' ds' \lesssim \frac{1}{|s|^3}, \ \ b(s)=\frac{1+o(1)}{2s}
\ee
\be \left|\lsl(s)\right|\lesssim |b(s)|+\mathcal N^{\frac 12}(s)\lesssim \frac1s,
\quad \left|\lsl+b\right| + \left| \xsl-1\right|  \lesssim
\frac 1 {|s|^{\frac 32}},\ee
and $|b_s|\lesssim \frac 1{|s|^2}$,
so that \eqref{estrescbis} and \eqref{estresc} are proved.

Now, we prove \eqref{devpb}.
We  rewrite the sharp modulation equation \fref{eq:2006} for $b$ as: 
$$\left|(b(1+J_2))_s+2b^2+c_2b^3\right|\lesssim  \int (\e_y^2+ \e^2) \varphi_B' +\frac1{s^4}+\frac1{s^2}\mathcal N^{\frac 12}\lesssim \int (\e_y^2+ \e^2) \varphi_B'+\frac{1}{|s|^{\frac 72}}.$$  Let \be\label{cnekocneone} \bt=b(1+J_2)=b+O\left(\frac{1}{|s|^{\frac 52}}\right),\ee then equivalently:
$$\left|\bt_s+2\bt^2+c_2\bt^3\right|\lesssim \int (\e_y^2+ \e^2) \varphi_B'+\frac{1}{|s|^{\frac 72}}.$$ 
If $c_2\leq 1$, let $b_0=-1$, otherwise let $b_0=-1/c_2$.
In order to integrate this differential inequation, we let
\be
\label{nkononeo}
F(b)=\int_{b_0}^b\frac{d\beta}{2\beta^2+c_2\beta^3}=-\frac{1}{2b}-\frac{c_2 \log |b|}{4}+c_0+O(b)\ \ \mbox{as} \ \ b\to 0,
\ee 
for some universal constant $c_0\in \RR$. Then,
\be
\label{cnekonoen}
\frac{d}{ds}F(\bt)=-1+O\left(s^2  \int (\e_y^2+ \e^2) \varphi_B' +\frac{1}{|s|^{\frac 32}}\right).
\ee 
 By \eqref{estfondamentalebis}
 with $s_1\to -\infty$, we have
 $$s^2\mathcal N(s)+\int_{-\infty}^s (s')^2 \int (\e_y^2+ \e^2)(s') \varphi_B'ds'\lesssim \frac1{|s|}.$$
Therefore, integrating \fref{cnekonoen} on $[s,s_0]$ and using \fref{nkononeo}:
$$F(\tilde{b}(s))=-\frac{1}{2\tb(s)}-\frac{c_2 \log  |\tb(s)|}{4}+c_0'+O\left(\frac1s\right)=-s+O\left(\frac{1}{\sqrt{|s|}}\right)$$
which is easily inverted to get:
$$\bt(s)=\frac{1}{2s}+\frac{c_1^*\log | s|}{s^2}+\frac{c_2^*}{s^2}+O\left(\frac{1}{|s|^{\frac 52}}\right)$$ for some universal constants $c_1^*,c_2^*$. The estimate \fref{cnekocneone} now implies \fref{devpb}.

\medskip
 
{\bf step 5} Global existence for $t>t_0$. 
Recall that for all $y>0$,
$|\e(s_0,y)|\lesssim e^{-\frac y{20}}$. Thus,
 $u(t)$ has  exponential decay in space on the right ($x>0$), in particular,
 $\int_{x>0} x^{10} u^2(t_0) <\infty.$
From this fact and since $u(t)$ has critical mass, we conclude from Theorem \ref{thmmduke}  that $u$ is globally defined for $t>t_0$. Since $Q$ has exponential decay at $\infty$, the exponential decay \eqref{controlright} obtained on $\e$ translates into   exponential decay on $u$ \eqref{decayu}. Finally, it is proved in \cite{Kato} that a solution of (gKdV) equation with such exponential decay on the right is smooth, i.e. $u\in C^\infty((0,+\infty)\times \RR)$.

\medskip

The proofs of \eqref{decayscaling} and \eqref{sobolev}-\eqref{sobolevbis} are given in Appendix A.
\end{proof}


\subsection{Sharp description of $S(t)$}


We conclude from Proposition \ref{sharpbounds} that the minimal element constructed in section \ref{sectionconstruction} satisfies the following sharp bounds  which conclude the proof of statements (i) and (iii) of Theorem \ref{th:1}.

\begin{corollary}\label{cor:1}
There exists a solution $S  \in \mathcal C((0,+\infty),H^1)\cap \mathcal C^{\infty}((0,+\infty)\times \RR)$  to \fref{kdv} with critical mass $\|S(t)\|_{L^2}=\|Q\|_{L^2}$ such that:
 \be\label{cor2}\|  \pa_xS(t)\|_{L^2}\sim \frac{\|\pa_xQ\|_{L^2}}{ t}\ \ \mbox{as} \  t\downarrow 0,
 \ee 
\begin{equation}\label{cor3}
S(t,x)-\frac{1}{t^{\frac 12}}Q\left(\frac{x+ \frac 1t+\bar{c} t}{ t}\right)\to 0\ \ \mbox{in}\ \ L^2 \ \ \mbox{as}\ \ t\downarrow 0,
\end{equation} 
\begin{equation}\label{cor4}
\forall x>0,\quad |S(1,x)|\lesssim e^{-\gamma x}
\end{equation}
for some universal constants $(\bar{c},\gamma)\in \mathbb R\times \mathbb R^*_+$. Moreover,
\be\label{derive}
   \frac d{dt} \left( \inf_{\l_1>0,   x_1\in \RR} \|S(t) - Q_{\l_1}(.-x_1)\|_{L^2}^2  \right)= 4 t (P,Q) +  O(t^2).
  \ee
\end{corollary}

\begin{proof}[Proof of Corollary \ref{cor:1}]
Let $v(t,x)$ be the minimal mass blow up solution constructed in section \ref{sectionconstruction} with finite backward blow up time $T<0$.
Let $\ell^*=\ell^*(v)$ and $x^*=x^*(v)$ be the constants corresponding  to $v$ in Proposition \ref{sharpbounds}. From the invariances of the equation and Remark \ref{Rk4.1},   $S(t)$ defined by
$$
S(t,x)= (\ell^*)^{- \frac 14} v\left((\ell^*)^{-\frac 32}t+T, (\ell^*)^{-\frac 12} x+x^*\right) 
$$
satisfies equation \eqref{kdv}, and the estimates of Proposition \ref{sharpbounds} with
$\ell^*(S)=1$, $x^*(S)=0$ and $S$ blows up backward at the origin in time. In particular, there exist $\e(t)$, $b(t)$, $\l(t)$ and $x(t)$ such that
\be\label{S1}
S(t,x)=\frac 1{\l^{\frac 12}(t)} \left(Q_{b(t)} + \e\right) \left( t, \frac {x-x(t)} {\l(t)}\right),
\ee
\be\label{S2}
b(t)=-t^2 + O(t^4),\quad \l(t) = t + O(t^3),\quad x(t) = -\frac 1t + \bar c t + O(t^2), 
\ee
\be\label{S3}
\|\e(t)\|_{L^2} \lesssim t,\quad \int  \left(e^{-\frac {|y|}{10}}
+\mathbf{1}_{y>0}(y)\right) \left( \e_y^2 + \e^2\right)(t,y)dy
\lesssim t^6. 
\ee
We now prove \fref{cor3}. Since 
$$\big\|\frac 1{\l^{\frac 12}(t)}  \e  \big( t, \frac {.-x(t)} {\l(t)}\big)\big\|_{L^2}
= \|   \e  \left( t  \right)\|_{L^2}\lesssim t,$$ we are reduced to estimate 
\begin{align*}
&
\left\| \frac 1{\l^{\frac 12}(t)}  Q_{b(t)} \left(  \frac {.-x(t)} {\l(t)}\right)-\frac{1}{t^{\frac 12}}Q\left(\frac{.+ \frac 1t +\bar c t }{ t}\right)\right\|_{L^2}\\
& = \left\|Q_{b(t)} - \frac {\l^{\frac 12}(t)}{t^{\frac 12}} Q\left( \frac {\l(t)}{t} x + \frac 1t \left(\frac 1t +\bar ct+ x(t)\right)\right) \right\|_{L^2}\\
&\lesssim |b(t)|^{\frac 58} + \left| 1 - \frac {\l(t)}{t}\right| 
+ \left|\frac 1t\left(\frac 1t +\bar ct+ x(t)\right)\right|\lesssim t,
\end{align*}
and  \fref{cor3} is proved. 

Let (see \eqref{S1})
$$
A(t) = \inf_{\l_0>0,   x_0\in \RR} \|S(t) - Q_{\l_0}(.-x_0)\|_{L^2}^2 =
\inf_{\l_1>0,   x_1\in \RR} \left\| Q_{b(t)}  +\e(t) - Q_{\l_1}(.-x_1)\right\|_{L^2}^2. 
$$
Let $\l_1(t)$ and $x_1(t)$ realizing the infimum in the definition of $A(t)$.
(The existence, uniqueness and regularity of $\l_1(t)$ and $x_1(t)$ follow by standard arguments.)

Note that by extremality of $\l_1$ and $x_1$,
\be\label{extremum}
\int \left( Q_b {+} \e {-} Q_{\l_1}(.{-}x_1)\right)   
    \frac {\pa Q_{\l_1}}{\pa \l_1}(.{-}x_1) =0,\ \
    \int \left( Q_b {+} \e {-} Q_{\l_1}(.{-}x_1)\right)   \frac {\pa Q_{\l_1}}{\pa x_1}  (.{-}x_1) =0,
\ee
and by $\|S(t)\|_{L^2} = \|Q_b+\e \|_{L^2} =\|Q\|_{L^2},$
$$
\int \left( Q_b + \e \right)  \frac {\pa}{\pa t} (Q_b+\e)=0
$$
so that 
\begin{align*}
& \frac 12 \frac d{dt} A(t) \\&= 
\int \left( Q_b + \e - Q_{\l_1}(.-x_1)\right) \left( \frac {\pa}{\pa t} (Q_b+\e) 
- \l_1'  \frac {\pa Q_{\l_1}}{\pa \l_1} (.-x_1)
+x_1' \frac {\pa Q_{\l_1}}{\pa x_1}  (.-x_1)\right)\\
& =
- \int   Q_{\l_1}(.-x_1)  \frac {\pa}{\pa t} (Q_b+\e)
= -\int  Q_{\l_1}(.-x_1) \left( b_t \frac {\pa Q_b}{\pa b}   + \e_t\right)
\\
&  
= - b_t \int  Q_{\l_1}(.-x_1)   \frac {\pa Q_b}{\pa b} 
- \int  \e_t \left(  Q_{\l_1}(.-x_1) - Q - (\l_1-1) \Lambda Q    \right) .
\end{align*}
where we have used at last $\int \e Q=\int \e \Lambda Q =0$.

To estimate this term, we now claim that from \eqref{extremum} and \eqref{S2}, 
\eqref{S3},
$$
|\l_1-1|\lesssim \left| \int (Q_b+\e - Q) \Lambda Q\right| \lesssim t^2,
$$
$$
|x_1|\lesssim \left| \int (Q_b+\e - Q)  Q'\right| \lesssim t^3,
$$
(the extra smallness of $|x_1|$ is due to $(P,Q')=0$).
Using $b_t \sim  -2t$ and the equation of $\e_t$ (after integration by parts,
and using \eqref{S2} and \eqref{S3}),
we obtain
\begin{align*}
  \frac 12 \frac d{dt} A(t) = 2 t \int PQ +  O(t^2).
  \end{align*}
\end{proof}


 \section{Uniqueness}
\label{secunique}


We prove in this section the uniqueness statement, i.e. part (ii) of Theorem \ref{th:1}. The stategy is to rerun the monotonicity machinery of Proposition \ref{propasymtp} for the difference of two solutions. The reintegration of the   Lyapounov functional backwards from blow up time {\it using the sharp  a priori bounds of Proposition \ref{sharpbounds}} will yield that this difference is zero. The proof is delicate because like in \cite{RS2010}, we only have a finite order expansion of the approximate solution and of the error. Therefore,  reintegrating the difference  of the modulation equations   requires sharp dispersive controls on the difference of two solutions to close the estimates.


\subsection{Reduction of the proof}


We consider $S(t,x)=u_{1}(t,x)$ the minimal mass blow up solution constructed in Corollary \ref{cor:1}. Let $u_{2}(t)$ be another minimal mass solution of \eqref{kdv} which blows up in finite time.
From Proposition \ref{sharpbounds}, $u_2(t)$ is defined on a maximal interval of time of the form $(-\infty,T)$ or $(T,+\infty)$ for a finite time $T$. By time translation invariance,
we may assume that $u_2(t)$ is defined on $(0,+\infty)$ and blows backwards as $t\downarrow 0$. 
Let $t_0>0$ small such that $u_1$ and $u_2$ admit   the decomposition of Lemma \ref{le:2} on $(0,t_0]$ (see also Lemma \ref{lemmadecay})
$$
 \e_i (s,y)=
 \l_i ^{\frac 12}(s) u_i(t_i(s),\l_i(s)y + x_i(s))-Q_{b_i(s)}(y),
$$
where $t_i(s)$ satisfies
 $\frac {d{t_i}}{ds}= {\l_i^3}$, $t_i(-1)=t_0$. 
Applying  Proposition \ref{sharpbounds} to $u_2$,   estimates \eqref{loilambda}--\eqref{sobolev} hold for $u_2(t)$, for some $\ell^*(u_2)$ and $x^*(u_2)$.

Using scaling and translation invariances  (see Remark \ref{Rk4.1}),  we assume further that 
the limits as defined in Proposition \ref{sharpbounds} are equal: $$\ell^*(u_2)=\ell^*(S)=1, \ \ x^*(u_2)=x^*(S)=0.$$ The uniqueness statement reduces to proving that
\be
\label{equaltiyutwp}
u_1\equiv u_2.
\ee

 
Note that for $i=1,2$,
$\e_i$ satisfies on $(-\infty,-1]\times \RR$,
$$  (\e_i)_{s} - (L \e_i)_y + {b_i}{\Lambda} \varepsilon_i= \Gamma_i({\Lambda} Q_{b_i}+{\Lambda} \varepsilon_i)+X_i(Q_{b_i} + \varepsilon_i)_y +\Psi_{i}  -( R_{i}(\varepsilon_i))_y,
$$ with 
 $$\Gamma_i=\frac{(\lambda_i)_s}{\l_i}+b_i, \ \ X_i=\frac{(x_i)_s}{\l_i}-1,$$
$$\Psi_i=\Psi_{b_i}- (b_i)_{s} \left(\chi_{b_i}   + \gamma   y (\chi_{b_i})_y\right) P,\quad
\hbox{$\Psi_b$ being defined in \eqref{eq:201}},$$
$$ R_i(\varepsilon_i)= 5 \left(Q_{b_i}^4 - Q^4\right) \varepsilon_i+(\varepsilon_i+Q_{b_i})^5
- 5 Q_{b_i}^4 \varepsilon_i - Q_{b_i}^5.
$$
 
We form the difference $$\e(s,y)=\e_2(s,y)-\e_1(s,y)$$ which satisfies the orthogonality conditions \fref{ortho1} and the equation: 
\be
\label{eqe}
\e_{s} - (L \e)_y = \Gamma\Lambda Q_{b_2}+X(Q_{b_2})_y +\frac{(\l_2)_s}{\l_2}\Lambda\e+E+F_y
\ee
with:
$$\Gamma=\Gamma_2-\Gamma_1,\ \ b=b_2-b_1,\ \ X=X_2-X_1,$$
\be
\label{defe}
E=(\Gamma-b)\Lambda \e_1+\Gamma_1\Lambda (Q_{b_2}-Q_{b_1})+(\Psi_2-\Psi_1),
\ee
\be\label{deff}
F=X_1(Q_{b_2}-Q_{b_1})+X_2 \e+ X\e_1-R_2(\e_2)+R_1(\e_1).
\ee

For $B$ as in Proposition \ref{propasymtp}, we consider 
$100 < \Bb < \frac 1{50}B$, large enough (in the next lemma,  we need $\bar B$ large, so we take a possibly larger universal $B$ in Proposition~\ref{propasymtp}).
We define the norms:
$$
\overline{\mathcal{N}}(s)= \int  \varepsilon_y^2(s,y) \psi_\Bb(y)dy 
+ \int  \varepsilon^2(s,y) \varphi_\Bb(y)   dy,
$$
$$  
\overline {\mathcal N}_{\rm loc}(s)=  \int   \varepsilon^2(s,y) \varphi_\Bb'(y)   dy.
$$

The key to the proof of uniqueness is the following Proposition which revisits Proposition \ref{propasymtp}  for $\e$:

\begin{proposition}[Bounds on the difference]
\label{lemmadiff}
For $|s|$ large, there holds the bounds:\\
{\it (i) Refined control of $b$}: Let  $$J_2=(\e,\rho_2)$$ with $\rho_2$ given by \fref{rho2}. Then,
\begin{equation}\label{eq:2006bis}
|J_2|\lesssim \NNb^{\frac 12},\quad 
\left|\frac d{ds} \left\{s^2 \left(b +  \frac {J_2} {2s} \right) \right\}  \right|\lesssim s^2 \int \e^2 e^{-\frac {|y|}{10}}+   |s|^{\frac 12} |b|   + |s|^{\frac 12} \NNb^{\frac 12}   .
\end{equation}
{\it (ii) Refined bounds}: let
\bea
\label{fepsbis}
&&{\cal F}(s)\\
\nonumber & =&  \int\left[\psi_\Bb\left(\varepsilon_y^2-5Q^4\e^2-\frac{\e^6}{3}\right) + \varphi_\Bb\varepsilon^2\right](s,y)dy+ \frac 1 {\sqrt{|s|}} \int e^{\l_2(s) y} \e^2(s,y)dy,
\eea
then: 
\be
\label{coercivityf}
 \NNb+ \frac 1 {\sqrt{|s|}}\int e^{\l_2  y} \e^2\lesssim \mathcal F\lesssim  
 \NNb + \frac 1 {\sqrt{|s|}}\int e^{\l_2  y} \e^2.
\ee
Moreover, there exists $\mu>0$ such that, for $|s|$ large,

\be
\label{pointwise}
\frac{d}{ds} \left( s^2\mathcal F\right) +\mu s^2 \int \left(\e_y^2+ \e^2\right) \varphi_\Bb'  \lesssim    |s|^{\frac {11}{10}}  b^2. 
\ee
\end{proposition}
\begin{remark}
The first term in the definition of $\mathcal F$ in \eqref{fepsbis} corresponds to a refined combination of viriel estimates and monotonicity properties which was used  in 
\cite{MMR1} (see also Proposition \ref{propasymtp} of the present paper).
Unfortunatly,
the scaling term of the equation of $\e$, i.e. the term $\frac {(\l_2)_s}{\l_2} \Lambda \e$, produces  bad apriori lower order terms which prevent us from  closing the estimates as in \cite{MMR1}. To control these terms we have to add to the definition of $\mathcal F$ the second term $\frac 1 {\sqrt{|s|}} \int e^{\l_2  y} \e^2$ which is a lower order corrective term. Note that this term is scaling invariant and thus it does not produce such bad terms.
\end{remark}

The next two sections are devoted to the proof of Proposition \ref{lemmadiff}.


\subsection{Proof of (i)}


We start with the control of the modulation parameters and the proof of the improved bound \fref{eq:2006bis}.\\

{\bf step 1} Modulation equations. We start with computing the modulation equations and claim the bounds:
\be
\label{boundone}
|\Gamma|+|X|\lesssim \left(\int \e^2 e^{-\frac {|y|}{10}} \right)^{\frac 12}+\frac  {|b|}{|s|},
\ee
\be\label{eqbtwo}
\left| \Gamma - \frac {(\e,L(\Lambda Q)')}{\|\Lambda Q\|_{L^2}^2}\right| 
\lesssim    \int \e^2 e^{-\frac {|y|}{10}} 
+\frac{1}{|s|}\left( \int \e^2 e^{-\frac {|y|}{10}}\right)^{\frac 12}+ \frac {|b|}{|s|},
\ee
\be
\label{boundtwo}
|b_s|\lesssim  \int \e^2 e^{-\frac {|y|}{10}} +\frac1{|s|} {\left(\int \e^2 e^{-\frac {|y|}{10}} \right)^{\frac 12}}+\frac  {|b|}{|s|}.
\ee
Indeed, we compute the modulation parameters $\Gamma,X$ using \fref{eqe} and the orthogonality conditions \fref{ortho1}. We argue like for the proof of \fref{eq:2002}, \fref{eq:2003} (see \cite{MMR1}) taking the scalar product of the equation of $\e$ by $\Lambda Q$, $y\Lambda Q$
and then by $Q$.
We obtain
$$
|\Gamma|+|X|\lesssim \left(\int \e^2 e^{-\frac {|y|}{10}} \right)^{\frac 12}+ |b|\left(|b_1|+|b|+|\Gamma_1|+|X_1|+\left(\int \e_1^2 e^{-\frac {|y|}{10}} \right)^{\frac 12}\right)+|b_s|;
$$
\begin{align*}
|b_s| &\lesssim \int \e^2 e^{-\frac {|y|}{10}}   +\frac {\left(\int \e^2 e^{-\frac {|y|}{10}} \right)^{\frac 12}}{|s|} \\
&+ |b|\left(|b_1|+|b|+|\Gamma_1|+|X_1|+\left(\int \e_1^2 e^{-\frac {|y|}{10}} \right)^{\frac 12}\right)+|b_1|(|\Gamma|+|X|).
\end{align*}
Next, using estimates \eqref{estrescbis}--\eqref{devpb} for $\e_1$, we find
\eqref{boundone} and \eqref{boundtwo}.

Note that estimate \eqref{boundone}  can be improved into 
\begin{align*}
\left| \Gamma - \frac {(\e,L(\Lambda Q)')}{\|\Lambda Q\|_{L^2}^2}\right| 
& \lesssim  \int \e^2 e^{-\frac {|y|}{10}} 
+ \left(|b|+\frac{1}{|s|}\right) \left(\int \e^2 e^{-\frac {|y|}{10}} \right)^{\frac 12} + \frac {|b|}{|s|}\\
& \lesssim  \int \e^2 e^{-\frac {|y|}{10}} +\frac{1}{|s|}\left( \int \e^2 e^{-\frac {|y|}{10}}\right)^{\frac 12}  + \frac {|b|}{|s|},
\end{align*}
which is \eqref{eqbtwo}.\\

{\bf step 2} Proof of (i). The estimate $|J_2|\lesssim \NNb^{\frac 12}$ follows from the properties of
$\rho_2$: $$|\rho_2|\lesssim {\bf 1}_{y>0}+ e^{-\frac {|y|}{10}}{\bf 1}_{y<0}.$$ We now turn to the proof of the refined equation of $b$. We claim the bound:
\begin{equation}\label{eq:2007}
\left|b_s + 4b_2b + b_2  ({J_2})_s  \right| \lesssim \int \e^2 e^{-\frac {|y|}{10}} +\frac {\NNb^{\frac 12}}{|s|^{\frac 32}}+\frac {|b|}{|s|^{\frac 32}}
\end{equation}
which follows from combining the following two estimates:
\be\label{gouh}
  \left|  {b_s} +4 b_2 b   -   b_2 (\e,L(\rho_2'))\right| 
 \lesssim \int \e^2 e^{-\frac {|y|}{10}}+\frac {\left(\int \e^2 e^{-\frac {|y|}{10}} \right)^{\frac 12} }{|s|^{\frac 32}}+ \frac {|b|}{|s|^{\frac 32}},
\ee
\be\label{gouh2}
\left| (J_2)_s   + (\e,L(\rho_2')) \right|\lesssim 
\int \e^2 e^{-\frac {|y|}{10}}+ \frac {\NNb^{\frac 12}}{|s|} + \frac {|b|}{|s|} .
\ee
Assume \fref{gouh}, \fref{gouh2}, then from \eqref{gouh2}:
$$
|(J_2)_s| \lesssim \int \e^2 e^{-\frac {|y|}{10}} +  \NNb^{\frac 12} +\frac { |b| }{|s|},
$$
and expanding $b_2$ in \eqref{eq:2007}  according to  \eqref{devpb} yields \fref{eq:2006bis}.\\

{\it Proof of \eqref{gouh}}. Taking the scalar product of the equation of $\e$ by $Q$, we obtain
(using $LQ'=0$)
\begin{align}
\nonumber
0 = \frac d{ds} (\e ,Q ) & =   \Gamma  ( \Lambda Q_{b_2}, Q)
+X   ((Q_{b_2})_y, Q) + \frac {(\l_2)_s}{\l_2} ( \Lambda \e, Q) \\ &+ ( E, Q )
-  ( F-\e^5 ,Q').
\label{scalarQ}
\end{align}

Using the definition of $Q_b$ in \eqref{eq:210}, we have $|(\Lambda Q_{b_2},Q) - (\Lambda P, Q) b_2|\lesssim  {|s|^{-10}}$, and thus using 
\eqref{eqbtwo}, we obtain
$$
\left|\Gamma (\Lambda Q_{b_2},Q) - b_2 (\Lambda P, Q)  \frac {(\e,L(\Lambda Q)')}{\|\Lambda Q\|_{L^2}^2} \right|  \lesssim \frac { \int \e^2 e^{-\frac {|y|}{10}}}{|s|} +\frac {|b|}{|s|^{2}} +\frac{\left(\int\e^2e^{-\frac{|y|}{10}}\right)^{\frac 12}}{s^2}.
$$
Similarly, since $|((Q_{b_2})_y,Q)|\lesssim {|s|^{-10}}$, using \eqref{boundone},
$$
\left| X ((Q_{b_2})_y,Q) \right| \leq \frac {\left(\int \e^2 e^{-\frac {|y|}{10}}\right)^{\frac 12}+|b| }{|s|^{10}}. 
$$

Now, we compute $(E,Q)$.
By the expression of $\Psi_b$ (see equation (2.17) in \cite{MMR1}), and the formula
$((10 P^2 Q^3)'+\Lambda P,Q)= \frac 18 \|Q\|_{L^1}^2$, 
we have
$$
(\Psi_{b_2}-\Psi_{b_1},Q) = - \frac 14 b_2 b \|Q\|_{L^1}^2 + O\left( \frac b{s^2}\right).
$$
Next, using the expression of $\Phi_b$ and $(P,Q)=\frac 1{16} \|Q\|_{L^1}^2 $,
$$
(\Phi_{b_2}-\Phi_{b_1},Q) = - \frac {b_s} {16} \|Q\|_{L^1}^2 + O\left(\frac b{s^{10}}\right).
$$
Thus,
\be\label{eqbone}
(\Psi_2-\Psi_1,Q) = - \frac 1{16} \|Q\|_{L^1}^2\left({b_s} +4 b_2 b \right) + O\left( \frac b{s^2}\right).
\ee
Since $(\Lambda \e_1,Q)=-(\e_1,\Lambda Q)=0$ and 
(using \eqref{estresc} on $\e_1$) $$|\Gamma_1 (\Lambda(Q_{b_2}-Q_{b_1}),Q)|\lesssim |s|^{-\frac 32} |b|,$$ 
we obtain 
$$
(E,Q) = - \frac 1{16} \|Q\|_{L^1}^2\left({b_s} +4 b_2 b \right) + O\left( \frac b{|s|^{\frac 32}}\right).
$$

Now, we compute $
(F-\e^5,Q')
$. First, since $|X_1|+|X_2|+\int |\e_1|e^{-\frac {|y|}{10}} \lesssim |s|^{-\frac 32}$, we have
$$|X_1 (Q_{b_2}-Q_{b_1},Q')|+
|X_2 (\e,Q')|+ |X(\e_1,Q')|\lesssim  \frac {\left(\int \e^2 e^{-\frac {|y|}{10}}\right)^{\frac 12}}{|s|^{\frac 32}} +\frac{|b|}{|s|^{\frac 32}}.$$
Second, we estimate $(R_2(\e_2)-R_1(\e_1),Q')$. From the expression of $R_i(\e_i)$, we observe that
\begin{align*}
& |R_2(\e_2) - R_1(\e_1)-20 b_2 P Q^3\e |
\\ &  \leq |R_2(\e_2)-R_2(\e_1) - 20 b_2 PQ^3 \e | + |R_2(\e_1)-R_1(\e_1)|
\\
&\lesssim |\e| \left(|s|^{-2}+|\e_1|+|\e_2|\right)+ |b| |\e_1|,\end{align*}
and so by \eqref{estresctri},
$$
\left|(R_2(\e_2)-R_1(\e_1),Q') - 20 b_2 (\e,P Q^3Q')\right| 
\lesssim \frac {\left(\int \e^2 e^{-\frac {|y|}{10}}\right)^{\frac 12}}{|s|^{\frac 32}} +\frac{|b|}{|s|^{\frac 32}}.
$$
We have thus obtained for this term:
$$
\left| ((F-\e^5)_y,Q)+20 b_2 (\e,P Q^3Q')\right| 
\lesssim \frac {\left(\int \e^2 e^{-\frac {|y|}{10}}\right)^{\frac 12}}{|s|^{\frac 32}} +\frac{|b|}{|s|^{\frac 32}}.
$$
Inserting the above computations into \eqref{scalarQ}, we obtain:    
\begin{align*}
& \left|  {b_s} +4 b_2 b   - \frac {16}{\|Q\|_{L^1}^2} b_2\left[ \frac {(\Lambda P, Q)}{\|\Lambda Q\|_{L^2}^2}
    {(\e,L(\Lambda Q)')}+ 20   (\e,P Q^3Q')\right]\right|\\
&\lesssim  \int \e^2 e^{-\frac {|y|}{10}} +\frac {\left(\int \e^2 e^{-\frac {|y|}{10}}\right)^{\frac 12} }{|s|^{\frac 32}}+ \frac {|b|}{|s|^{\frac 32}}
\end{align*}

Using the following computation\footnote{which of course motivates the definition of $\rho_2$ in \eqref{rho2}.} from \cite{MMR1}, proof of Lemma 2.7:
$$
(\e,L(\rho_2')) = \frac {16}{\|Q\|_{L^1}^2} \left[ \frac {(\Lambda P, Q)}{\|\Lambda Q\|_{L^2}^2}
    {(\e,L(\Lambda Q)')}+ 20   (\e,P Q^3Q')\right],
$$
we obtain \eqref{gouh}.

\medskip

{\it Proof of \eqref{gouh2}}: 
To complete the proof of \eqref{eq:2007}, we take the scalar product of the equation of $\e$ by
$\rho_2$. We obtain first:
\begin{align*}
\frac d{ds} J_2 + \frac {(\l_2)_s}{\l_2} (\e, \Lambda \rho_2) &= - (\e,(L\rho_2)') + \Gamma (\Lambda Q_{b_2},\rho_2) + X ((Q_{b_2})_y,\rho_2)
\\& + (E,\rho_2) + (-F+\e^5,\rho_2').
\end{align*}
Note that
$$
\left |\frac {(\l_2)_s}{\l_2}(\e, \Lambda \rho_2)\right|\leq\left|  \frac {(\l_2)_s}{\l_2} \right|( |J_2| +| (\e,y\rho_2')|)\lesssim \frac {\NNb^{\frac 12}}{|s|}.
$$
Using the orthogonality 
$(\Lambda Q, \rho_2)=0$ (see  \cite{MMR1}), and \eqref{boundone}, we have
$$
|\Gamma (\Lambda Q_{b_2},\rho_2) |
=|\Gamma b_2 (\Lambda P\chi_{b_2}, \rho_2)| 
\lesssim 
\left[\left(\int \e^2 e^{-\frac {|y|}{10}}\right)^{\frac 12}+ \frac {|b|}{|s|}\right] \frac 1{|s|} 
$$
and similarly using $(Q,\rho_2')=0$:
$$
| X ((Q_{b_2})_y,\rho_2) | +
|(E,\rho_2) |+
| (-F+\e^5,\rho_2')|\lesssim \frac {\left(\int \e^2 e^{-\frac {|y|}{10}}\right)^{\frac 12}+|b|+\NNb^{\frac 12}}{|s|}+ \int \e^2 e^{-\frac {|y|}{10}},
$$
and \eqref{gouh2} is proved.


\subsection{Proof of (ii)}


The   functional $\mathcal F$ in \fref{fepsbis} is defined
similarly as in Proposition \ref{propasymtp}  for a   parameter $\Bb$     large enough but smaller than  $B/10$ where $B$ is   used in Proposition \ref{propasymtp}.\\

{\bf step 1} Coercivity of $\matchal F$. The upper and lower bounds \fref{coercivityf} on $\mathcal{F}$ follow from the coercivity of the linearized energy $\int \e_y^2 + \e^2 - 5Q^4 \e^2$ under the   orthogonality conditions \fref{ortho1} together with standard localization arguments. We refer to the proof of Proposition~3.1 (iii) in \cite{MMR1} for example for more details.\\

{\bf step 2} Proof of \eqref{pointwise}. We now turn to the proof of the monotonicity \fref{pointwise}. We decompose
$\mathcal F = \overline {\mathcal F_1} + \frac 1{\sqrt{|s|}} \overline {\mathcal F_2}$ with
$$
\overline{\mathcal F_1} = \int\left[\psi_\Bb\left(\varepsilon_y^2-5Q^4\e^2\right) + \varphi_\Bb\varepsilon^2\right],\quad
\overline {\mathcal F_2} = \int e^{\l_2 y} \e^2$$
and claim the monotonicity formulas for $|s|$ large enough:
\be\label{eq:surF1}
\frac{d\overline{\mathcal F_1}}{ds}+\mu_1 \int \left(\e_y^2+\e^2\right) \varphi_\Bb' \lesssim   \frac  {b^2}{|s|}  + \frac 1{|s|} \int_{y<0} |y| e^{-\frac {|y|}{\Bb}} \e^2,
\ee
\be\label{eq:surF2}
\frac{d\overline {\mathcal F_2}}{ds}+
\frac {\mu_2} {\sqrt{|s|}}\int e^{\l_2 y} (\e_y^2+ \e^2) \lesssim \int \e^2 e^{-\frac {|y|}{10}} +\frac {|b|^2}{|s|^{\frac 25}}.
\ee

Assume \fref{eq:surF1}, \fref{eq:surF2}, then for $|s|$ large:
\begin{align*}
& \frac d{ds} \left(s^2 \overline{\mathcal F_1} +  |s|^{\frac 32} \overline {\mathcal F_2}\right) 
 = 2 s \overline{\mathcal F_1} -\frac 32 |s|^{\frac 12} \overline {\mathcal F_2} 
+s^2   \frac d{ds} \overline{\mathcal F_1} + |s|^{\frac 32}\frac d{ds}\overline {\mathcal F_2}\\
&
\leq  - {\mu_1} s^2   \int \left(\e_y^2+\e^2\right) \varphi_\Bb'  -   {\mu_2}{|s|} \int e^{\l_2 y}\e^2  + C |s|^{\frac {11}{10}} b^2  
  +    C {|s|}   \int_{y<0} |y|e^{-\frac {|y|}{\Bb}} \e^2 \\&+   C{  |s|^{\frac 32}} \int \e^2 e^{-\frac {|y|}{10}}
  \\
&
\leq  - \frac {\mu_1}2 s^2 \int \left(\e_y^2+\e^2\right) \varphi_\Bb' -   {\mu_2}{|s|} \int e^{\l_2 y}\e^2  + C |s|^{\frac {11}{10}}  b^2  
  +    C_1 {|s|}   \int_{y<0} |y|e^{-\frac {|y|}{B}} \e^2. 
\end{align*}
Then, for $0<\kappa<\mu_1/(4\Bb C_1) $   and $|s|$ large enough (depending on $\Bb$),
\bee
  C_1 {|s|}   \int_{y<0} |y|e^{-\frac {|y|}{\Bb}} \e^2
&=&  C_1 {|s|}   \int_{\kappa s<y<0} |y|e^{-\frac {|y|}{\Bb}} \e^2+  C_1 {|s|}   \int_{y<\kappa s} |y|e^{-\frac {|y|}{\Bb}} \e^2\\
&\leq& C_1 \Bb \kappa s^2 \NNbl +   C_1 {|s|}  \sup_{y<\kappa s} \left(|y|e^{-\frac {|y|}{2\Bb}}\right)     \int e^{-\frac {|y|}{2\Bb}} \e^2
\\ &\leq& \frac {\mu_1}  4 s^2 \NNbl  + \frac {\mu_2} 2 |s| \int e^{\l_2 y} \e^2,
\eee
and thus
$$\frac d{ds} \left(s^2 \overline{\mathcal F_1} +  |s|^{\frac 32} \overline {\mathcal F_2}\right) 
+ \frac {\mu_1} 4 s^2 \int (\e_y^2 + \e^2)  \varphi_\Bb'  +   \frac {\mu_2} 2 |s| \int e^{\l_2 y} \e^2
\lesssim |s|^{\frac {11}{10}} b^2 $$
which implies \eqref{pointwise}.\\

{\bf step 3} Proof of \eqref{eq:surF1}.
We   compute the time derivative of $\overline{\mathcal F_1}$ using \fref{eqe}:
\bee
 \frac12\frac{d\overline{\matchal F_1}}{ds}& = & \int\pa_s\e\left[-\psi'_{\Bb} \e_y+\psi_{\Bb}(L\e)+(\varphi_\Bb-\psi_{\Bb})\e\right]\\
&  =&\int\left[(L \e)_y+\Gamma\Lambda Q_{b_2}+X(Q_{b_2})_y  + \frac {(\l_2)_s}{\l_2} {\Lambda} \varepsilon+E+F_y\right]\\ 
&&\times \left[-\psi'_{\Bb}\e_y+\psi_{\Bb}(L\e)+(\varphi_\Bb-\psi_{\Bb})\e\right].
\eee
We now  estimate all these terms similarly as in  the proof of Proposition 3.1 (i) in \cite{MMR1}.

- First, we claim that for $\Bb$ large enough, for some $\mu_1>0$,
\be\label{f5.17}
\int (L \e)_y\left[-\psi'_{\Bb}\e_y+\psi_{\Bb}(L\e)+(\varphi_\Bb-\psi_{\Bb})\e\right]
\leq 
- \mu_1  \int \varphi_\Bb'(\e_y^2+\e^2).
\ee
The proof is mainly based on local virial estimates for $\e$ and explicit computations similar to the ones for the term $f_{1,1}^{(i)}$ of the proof of Proposition 3.1 in \cite{MMR1}.
Here, computations are similar and easier than in \cite{MMR1}. 
We sketch these computations and estimates for the sake of completeness.

By explicit computations (mainly integrations by parts, see \cite{MMR1} for more details), one gets
\bee
&&2 \int (L \e)_y\left[-\psi'_{\Bb}\e_y+\psi_{\Bb}(L\e)+(\varphi_\Bb-\psi_{\Bb})\e\right]
\\
&&= -\int\left[3\psi_\Bb '\e_{yy}^2+(3\varphi_\Bb '+\psi_\Bb '-\psi_\Bb ''')\e_y^2+(\varphi_\Bb '-\varphi_\Bb ''')\e^2\right]\\
\nonumber &&+  \int 5Q^4 \e^2  (\varphi_\Bb '-\psi_\Bb ')
+    \int    20 Q^3Q' \e^2  (\psi_\Bb -\varphi_\Bb )\\
&&+ 10\int\psi_\Bb '\e_y\left\{4Q'Q^3 \e + Q^4  \e_y\right\}\\
&&- \int\psi_\Bb '\left\{\left(  5 Q^4 \e   \right)^2- 10Q^4 \e  \left(-\e_{yy}+\e\right) \right\}\\
&&=  I^<+ I^{\sim}+ I^>
\eee
where $ I^{<,\sim,>}$ respectively corresponds to integration on $y<-\frac \Bb2$, $|y|\leq \frac \Bb2$, $y>\frac \Bb2$.

In the region $y>\Bb/2$, we have $\psi_\Bb'(y)=0$, and thus,
\begin{align*}
I^>
&=  -\int_{y>\Bb/2}\left[ 3\varphi_\Bb ' \e_y^2 +(\varphi_\Bb '-\varphi_\Bb ''')\e^2\right]
+  \int_{y>\Bb/2}  5Q^4 \e^2  \varphi_\Bb ' \\ \nonumber &
+    \int_{y>\Bb/2}    20 Q^3Q' \e^2  (1 -\varphi_\Bb ).
\end{align*}
Using $\varphi_\Bb''' \lesssim \frac 1{\Bb^2} \varphi_\Bb'$ and  the exponential decay of $Q$, we obtain
for $\Bb$ large enough,
$$
I^> \leq -\int_{y>\Bb/2} \varphi_\Bb ' (\e_y^2  +\e^2)   +  \frac C{\Bb^2}  \int_{y>\Bb/2} \varphi_\Bb ' \e^2\leq -\frac 12 \int_{y>\Bb/2} \varphi_\Bb ' (\e_y^2  +\e^2).
$$

In the region $|y|<\Bb/2$, we have $\varphi_\Bb (y)= 1 + y / \Bb$, $\psi_\Bb (y)=1$ and $\psi_\Bb''' =\psi_\Bb'=0$.  Thus,
\begin{align*}
I^{\sim} & = - \frac 1 \Bb \int_{|y| < \frac \Bb2} \left[ 3 \e_y^2+ \e^2 -5   Q^4 \e^2   + 20y Q^3 Q'\e^2 \right]    + \frac 1 \Bb \frac 53 \int_{|y| < \frac \Bb2} \e^6.
\end{align*} 
From Lemma~3.4 in \cite{MMR1} (local virial estimate), for some $\mu>0$ and for $\Bb$ large,
$$
\int_{|y| < \frac \Bb2} \left[ 3 \e_y^2+ \e^2 + 15   Q^4 \e^2   - 20y Q^3 Q'\e^2 \right] 
\geq \mu \int_{|y| < \frac \Bb2}  \left( \e_y^2+ \e^2  \right) - \frac 1{\Bb} \int \e^2 e^{-\frac {|y|}{2}},
$$
and by $\|\e\|_{L^\infty}^4 \lesssim |s|^{-2}$, $\int_{|y| < \frac \Bb2} \e^6\lesssim |s|^{-2}
\int_{|y| < \frac \Bb2} \e^2$, so that for $\Bb$ large and $|s|$ large,
$$
I^{\sim} \leq -  \frac \mu 2 \frac 1 \Bb \int_{|y| < \frac \Bb2} \left( \e_y^2+ \e^2 \right)
+ \frac 1{\Bb^2} \int \e^2 e^{-\frac {|y|}{2}}.
$$

In the region $y<-\Bb/2$, we   use $\psi_\Bb''' \lesssim \frac 1{\Bb^2} \psi_\Bb'$,
$\psi_\Bb'\lesssim \varphi_\Bb'$, $\varphi_\Bb''' \lesssim \frac 1{\Bb^2} \varphi_\Bb'$,
the exponential decay of $Q$ to obtain as before,
for $\Bb$ large enough,
$$
I^<  \leq -\frac 12 \int_{y<-\Bb/2} \varphi_\Bb ' (\e_y^2  +\e^2).
$$

Gathering the  estimates for $I^{>}$, $I^\sim$ and $I^<$, we  get \eqref{f5.17}.

\medskip

- Next, arguing as for estimating $f_{1,2}^{(i)}$ and $f_{1,3}^{(i)}$  in the proof of Proposition 3.1 of \cite{MMR1}, we find
\be\label{f5.18}
\left|\Gamma \int \Lambda Q_{b_2}  \left[-\psi'_{\Bb}\e_y+\psi_{\Bb}(L\e)+(\varphi_{\Bb}-\psi_{\Bb})\e\right]\right| \leq 
\frac   {\mu_1}{100}  \NNbl+ C \frac {b^2}{s^2},
\ee
\be\label{f5.19}
\left| X\int   (Q_{b_2})_y    \left[-\psi'_{\Bb}\e_y+\psi_{\Bb}(L\e)+(\varphi_\Bb-\psi_{\Bb})\e\right]\right| \leq \frac   {\mu_1}{100}  \NNbl+ C \frac {b^2}{s^2}.
\ee
Indeed, using the following algebraic facts  $$
(\Lambda Q,L \e)=-2(Q,\e)=0,\quad (\e,y\Lambda Q)=(\e,\Lambda Q)=(\e,yQ')=0,\quad LQ'=0,
$$
the exponential decay of $Q$, $|b_2| \leq \frac 1{\sqrt{|s|}}$ and integrating by parts to remove all derivative from $\e$, we obtain
$$
\left|  \int \Lambda Q_{b_2}  \left[-\psi'_{\Bb}\e_y+\psi_{\Bb}(L\e)+(\varphi_{\Bb}-\psi_{\Bb})\e\right]\right| \lesssim  \left( \frac{1}{\Bb}   + \frac {\sqrt{\Bb}} {\sqrt{|s|}} \right)\NNbl^{\frac 12},
$$
$$
\left| \int   (Q_{b_2})_y    \left[-\psi'_{\Bb}\e_y+\psi_{\Bb}(L\e)+(\varphi_\Bb-\psi_{\Bb})\e\right]\right|\lesssim    \left( \frac{1}{\Bb}  + \frac {\sqrt{\Bb}}{\sqrt{|s|}} \right)\NNbl^{\frac 12}.
$$
Thus, using  \fref{boundone}, \eqref{f5.18} and \eqref{f5.19} follow for $\Bb$ large and $|s|$ large.

\medskip

At this point, $\Bb$ is fixed and thus $B$ in Proposition \ref{propasymtp} is also fixed. $B$ and $\Bb$ are  universal constants.

\medskip

- The next term is similar to $f_{3}^{(i,j)}$ in \cite{MMR1}. We have using the properties of 
$\psi_{\Bb}$ and  $\varphi_{\Bb}$ and \fref{eq:2002},
\bee && \frac{(\l_2)_s}{\l_2}\int\Lambda \e\left[-(\psi_{\Bb}\e_y)_y+\varphi_{\Bb} \e - \psi_{\Bb}(5Q^4 \e + \e^5) \right]\\
&&= - \frac 12  \frac{(\l_2)_s}{\l_2}\int y\varphi_{\Bb}'\e^2+O\left (\frac 1{|s|}\int (\e_y^2+ \e^2) \varphi_\Bb' \right).\eee
From \fref{controlnorme}, \fref{estresc}  and \fref{devpb}, we have for $s$ large
$$\frac 1{s}\leq  -\frac{(\l_2)_s}{\l_2}\leq \frac 1{4s} <0$$
and thus
$$
- \frac 12  \frac{(\l_2)_s}{\l_2}\int y\varphi_{\Bb}'\e^2
\lesssim \frac 1 {|s|} \int_{y<0}  |y| e^{- \frac {|y|}{B}} \e^2.
$$
Eventually, we have proved for $s$ large:
\begin{align*} 
&  \frac{(\l_2)_s}{\l_2}\int\Lambda \e\left[-(\psi_{\Bb}\e_y)_y+\varphi_{\Bb} \e - \psi_{\Bb}(5Q^4 \e + \e^5) \right] 
\\
& \leq \frac {\mu_1}{100}\int (\e_y^2+ \e^2) \varphi_\Bb'+\frac C {|s|} \int_{y<0}  |y| e^{- \frac {|y|}{B}} \e^2 .
\end{align*}
 
  \medskip
     
We now estimate terms coming from $E$ and $F$. For this, we will need higher order Sobolev estimates on $\e_1$ and $\e_2$ coming from Proposition \ref{sharpbounds}.
Since $\Bb < \frac B{50}$, from  \eqref{sobolev} and \eqref{sobolevbis}, we have, for all $\frac {9}{50} \frac 1 {\Bb} < \omega < \frac 1{10}$, for $i=1,2$,
  \be\label{sobolev2}
  \sum_{k=0}^3\int \left(\partial_y^k \e_i\right)^2(s,y) e^{\omega  y  } dy +
\int_{-\infty}^s \sum_{k=0}^4\int \left(\partial_y^k \e_i\right)^2(s',y) e^{\omega  y  } dyds'
\lesssim \frac 1{|s|},
\ee
\be\label{sobolev3}
\| \left( (\e_i)_{yy}^2+(\e_i)_y^2  \right)(s ) e^{\omega  y  }\|_{L^\infty}
 \lesssim \frac 1{|s|}.
\ee
  
- Estimate for $E$. In view of the expression of $E$ in \eqref{defe}, the first term to estimate is
\begin{align*}
&\left|\left(\Gamma-b\right) \int\Lambda \e_1\left[-\psi'_{\Bb}\e_y+\psi_{\Bb}(L\e)+(\varphi_{\Bb}-\psi_{\Bb})\e\right]\right|\\
& \lesssim   \left(|b|+\left(\int \e^2 e^{-\frac {|y|}{10}}\right)^{\frac 12}\right)
\left(\int (\e_y^2+ \e^2) \varphi_\Bb'\right)^{\frac 12} 
\left(\int \left((\e_1)_{yy}^2 + (\e_1)_y^2 + \e_1^2\right) (1+|y|^3) \varphi_\Bb\right)^{\frac 12}
\\
&
\leq  \frac{\mu_1}{100  } \int (\e_y^2+ \e^2) \varphi_\Bb' +C   {b^2}\int \left((\e_1)_{yy}^2 + (\e_1)_y^2 + \e_1^2\right) e^{\frac 2{\Bb}y} dy\\
& \leq \frac{\mu_1}{100  } \int (\e_y^2+ \e^2) \varphi_\Bb'+  C \frac {b^2}{|s|} ,
\end{align*}
using \eqref{boundone}, integration by parts, and then \eqref{sobolev2}. For the next term, we need to estimate $\Lambda(Q_{b_2}-Q_{b_1})$. From Lemma \ref{cl:2}, we have
$$
Q_{b_2}-Q_{b_1} = b P \chi_{b_2} + b_1 P (\chi_{b_2}-\chi_{b_1}),
$$
and using \fref{devpb}:
$$
|\chi_{b_2}-\chi_{b_1}|=\left|\int_{b_1}^{b_2}\frac{\pa \chi_b}{\pa b}db\right| \lesssim \sup |y\chi'| \frac{|b_1-b_2|}{ |b_1|}.
$$
 
Thus, $|Q_{b_2}-Q_{b_1}|\lesssim |b| |P|$. Arguing similarly, we obtain:
\be\label{diffQ}
|(Q_{b_2}-Q_{b_1})_y|+
| (\Lambda Q_{b_2}-\Lambda Q_{b_1})_y|+|\Lambda (Q_{b_2}-    Q_{b_1})|\lesssim |b| \left(\mathbf{1}_{y<0} + e^{-\frac {|y|}{10}}\right).
\ee 
Using \eqref{diffQ} and   estimates for $\Gamma_1$ from Proposition \ref{sharpbounds},
\bee
&&\left|\int \Gamma_1\Lambda(Q_{b_2}-Q_{b_1})\left[-\psi'_\Bb\e_y+\psi_\Bb(L\e-\e^5)+(\varphi_\Bb-\psi_\Bb)\e\right]\right|\\
&& \lesssim   \frac {|b|}{{|s|}^{\frac 32}}\NNbl^{\frac 12}\leq \frac{\mu_1}{100 }\NNbl+C \frac {b^2}{s^2}.
\eee
Next, from the definition of $\Psi_{b_i}$ and $\Phi_{b_i}$ (see Lemma \ref{cl:2} and equation (2.17) in \cite{MMR1}), \eqref{boundtwo}, \fref{eq:2003}, \fref{estresctri}:
we have 
 \begin{align*}
|\Psi_{b_2}-\Psi_{b_1}| & \lesssim \frac {|b|}{\sqrt{|s|}}  \left(\mathbf{1}_{y<0} + e^{-\frac {|y|}{10}}\right),\\
 |\Phi_{b_2}-\Phi_{b_1}| & \lesssim  \left(|b_s| +  \frac{|(b_1)_s||b|}{|b_1|}\right)  \left(\mathbf{1}_{y<0} + e^{-\frac {|y|}{10}}\right)\\ &\lesssim  \left(\int \e^2 e^{-\frac {|y|}{10}}+\frac {\left(\int \e^2 e^{-\frac {|y|}{10}}\right)^{\frac 12}}{|s|}+ \frac {|b|}{|s|}  \right)\left(\mathbf{1}_{y<0} + e^{-\frac {|y|}{10}}\right)
\end{align*}
and similar estimates for the derivatives of these terms. In particular, we obtain
\begin{align}
&|\Psi_2-\Psi_1|+|(\Psi_2-\Psi_1)_y|+|(\Psi_2-\Psi_1)_{yy}|\nonumber
\\ &\lesssim  \left(\int \e^2 e^{-\frac {|y|}{10}}+\frac {\left(\int \e^2 e^{-\frac {|y|}{10}}\right)^{\frac 12}}{|s|}+ \frac {|b|}{\sqrt{|s|}}  \right)\left(\mathbf{1}_{y<0} + e^{-\frac {|y|}{10}}\right)  \label{diffPsiBIS}
\\ &\lesssim  \left(\NNbl+\frac {\NNbl^{\frac 12}}{|s|}+ \frac {|b|}{\sqrt{|s|}}  \right)\left(\mathbf{1}_{y<0} + e^{-\frac {|y|}{10}}\right) \label{diffPsi}
\end{align}
Thus,
\bee
&& \left|\int (\Psi_2-\Psi_1)\left[-\psi'_\Bb\e_y+\psi_\Bb(L\e)+(\varphi_\Bb-\psi_\Bb)\e\right]\right|\\
&& \lesssim  \left(\NNbl+\frac {\NNbl^{\frac 12}}{|s|}+ \frac {|b|}{\sqrt{|s|}}  \right)\NNbl^{\frac 12}\leq  \frac{\mu_1}{100 }\NNbl +C \frac {b^2}{|s|}.
\eee
 
 In conclusion for $E$, we have obtained
 $$
\left|  \int E\left[-\psi'_\Bb\e_y+\psi_\Bb(L\e)+(\varphi_\Bb-\psi_\Bb)\e\right]\right| 
\leq \frac {3\mu_1} {100} \int (\e_y^2 + \e^2) \varphi_\Bb' + C \frac {b^2}{|s|}.
 $$

- Estimate for $F$. Similarly, we easily get the following three estimates
 \bee
 && \left|X_1\int (Q_{b_2}-Q_{b_1})_y\left[-\psi'_\Bb\e_y+\psi_\Bb(L\e)+(\varphi_\Bb-\psi_\Bb)\e\right]\right|\\
 && \lesssim   \frac {|b|}{|s|^{\frac 32}} \NNbl^{\frac 12}\lesssim \frac{\mu_1}{100}\NNbl+C \frac {b^2}{s^2} ,
\eee
\bee 
&& \left|X_2 \int \e_y\left[-\psi'_\Bb\e_y+\psi_\Bb(L\e)+(\varphi_\Bb-\psi_\Bb)\e\right]\right|\\
& & \lesssim    \frac 1{|s|^{\frac 32}}\int (\e_y^2+ \e^2) \varphi_\Bb'
 \leq \frac{\mu_1}{100}\int (\e_y^2+ \e^2) \varphi_\Bb',
\eee
\begin{align*}
& \left|X \int (\e_1)_y\left[-\psi'_\Bb\e_y+\psi_\Bb(L\e)+(\varphi_\Bb-\psi_\Bb)\e\right]\right|\\
& \lesssim   \left( \frac {|b|}{|s|} + \left(\int \e^2 e^{-\frac {|y|}{10}}\right)^{\frac 12}\right)  \left(\int (\e_y^2+ \e^2) \varphi_\Bb' \right)^{\frac 12}\left(\int \left((\e_1)_{yy}^2 + (\e_1)_y^2 + \e_1^2\right) e^{\frac 2{\Bb}y} dy\right)^{\frac 12}
\\
&\leq \frac{\mu_1}{100}\int (\e_y^2+ \e^2) \varphi_\Bb'+C \frac {b^2}{|s|}.
\end{align*}
The remaining nonlinear term for $F$ is estimated using the Sobolev bound \fref{sobolev2}.
We decompose the nonlinear term as follows
$$R_2(\e_2) - R_1(\e_1) =F_1+F_2$$
where 
$$F_1=   R_2(\e_1)-R_1(\e_1),\quad F_2=R_2(\e_2)-R_2(\e_1) .$$

Using the expression of $R_i$ and $\|\e_1\|_{L^\infty} \lesssim |s|^{-\frac 12}$,
we have
$$
|(F_1)_y|+ |(F_1)_{yy}|\lesssim |b| \left(|\e_1|+|(\e_1)_y| + |(\e_1)_y^2| + |(\e_1)_{yy}|\right),
$$ 
and then by \eqref{sobolev3},
$$
\left(|(F_1)_y|+ |(F_1)_{yy}| \right)e^{\frac   y{4\Bb}}\lesssim 
 |b| \left(|\e_1|+|(\e_1)_y| +   |(\e_1)_{yy}|\right) e^{  \frac y{4\Bb}}
 + \frac {|b|}{\sqrt{|s|}} |(\e_1)_y|.
$$
Thus, integrating by parts, and using Cauchy-Schwarz inequality and \eqref{sobolev2}:
\begin{align*}
& \left| \int (F_1)_y \left[-\psi'_{\Bb}\e_y+\psi_{\Bb}(L\e)+(\varphi_\Bb-\psi_{\Bb})\e\right]
\right| \\
&  \lesssim\int \left(|(F_1)_y|+ |(F_1)_{yy}| \right)e^{  \frac   y{2\Bb}}
(|\e|+|\e_y|) \sqrt{\varphi_\Bb'}
\\
&  \lesssim |b| \int  \left[  \left(|\e_1|+|(\e_1)_y| +   |(\e_1)_{yy}|\right) e^{\frac y{4B}}
 +   |(\e_1)_y|\right] e^{\frac y{4\Bb}}
(|\e|+|\e_y|) \sqrt{\varphi_\Bb'}\\
& \lesssim \frac {|b|}{\sqrt{|s|}}\left(\int (\e_y^2+ \e^2) \varphi_\Bb'\right)^{\frac 12}\leq \frac{\mu_1}{100}\int (\e_y^2+ \e^2) \varphi_\Bb'+C \frac {b^2}{s}.
\end{align*}

We decompose $F_2$ as follows
\begin{align*}
F_2 &= 5(Q_{b_2}^4 - Q^4) \e + 10 Q_{b_2}^3 (\e_2^2-\e_1^2) 
+ 10 Q_{b_2}^2 (\e_2^3 - \e_1^3) + 5 Q_{b_2} (\e_2^4 - \e_1^4)
\\ &+ \e_2^5 - \e_1^5  \\
& =\e \left[ 5(Q_{b_2}^4 - Q^4) + 10 Q_{b_2}^3 (\e_2+ \e_1)
+ 10 Q_{b_2}^2 (\e_2^2+\e_1\e_2 + \e_1^2) \right.\\
& \left.+ 5 Q_{b_2} (\e_2^3+\e_2^2\e_1+\e_2 \e_1^2+\e_1^3)
+  (\e_2^4+\e_2^3\e_1+\e_2^2 \e_1^2+\e_2\e_1^3+\e_1^4)\right]
\end{align*}
Therefore, by suitable integration by parts, we have
\begin{align*}
& \left|\int  (F_2)_y  \left[-\psi'_{\Bb}\e_y+\psi_{\Bb}(L\e-\e^5)+(\varphi_\Bb-\psi_{\Bb})\e\right]
\right|\\
& \lesssim |b_2| \int (\e_y^2 + \e^2) \varphi_\Bb' + \int (\e_y^2 + \e^2) \varphi_\Bb \left(|\e_1|+|\e_2|\right)\\
&+\int (\e_y^2 + \e^2) \psi_{\Bb}    \sum_{i=1,2} \left(|\e_i|+|(\e_i)_y|   
+|(\e_i)_y^2| +|(\e_i)_{yy}|\right).
\end{align*}
These terms are next treated as follows:
$$
C |b_2| \int (\e_y^2 + \e^2) \varphi_\Bb'
\leq \frac{\mu_1}{100}\int (\e_y^2 + \e^2) \varphi_\Bb',
$$
and using \eqref{sobolev3}, $\|\e_i\|_{L^\infty}\lesssim \frac 1{\sqrt{|s|}}$ and the notation $y_+ = \max (0,y)$:
\begin{align*}
& \int (\e_y^2 + \e^2) \varphi_\Bb \left(|\e_1|+|\e_2|\right)
\\
& \lesssim  \left( \|\e_1 (1+|y_+|)\|_{L^\infty} + \|\e_2 (1+|y_+|)\|_{L^\infty}\right)
 \int (\e_y^2 + \e^2) \varphi_\Bb'
\\ &\leq \frac{\mu_1}{100}\int (\e_y^2 + \e^2) \varphi_\Bb'.
\end{align*}
Finally, using $\psi_{\Bb}\lesssim \varphi_\Bb' e^{\frac y\Bb}$ and \eqref{sobolev3}:
\begin{align*}
& C\int (\e_y^2 + \e^2) \psi_{\Bb}    \sum_{i=1,2} \left(|\e_i|+|(\e_i)_y|   
+|(\e_i)_y^2| +|(\e_i)_{yy}|\right)\\
& \lesssim   \left \|\left(|\e_i|+|(\e_i)_y|   
+|(\e_i)_y^2| +|(\e_i)_{yy}|\right)   e^{\frac y\Bb}\right\|_{L^\infty} 
\int (\e_y^2 + \e^2) \varphi_\Bb'
\\ & \leq \frac{\mu_1}{100}\int (\e_y^2 + \e^2) \varphi_\Bb'.
\end{align*}
 The collection of above estimates yields the bound:
 $$
\left|  \int F_y \left[-\psi'_\Bb\e_y+\psi_\Bb(L\e-\e^5)+(\varphi_\Bb-\psi_\Bb)\e\right]\right| 
\leq \frac {7\mu_1} {100} \int (\e_y^2 + \e^2) \varphi_B' + C \frac {b^2}{|s|}.
 $$
 \medskip

{\bf step 4} Proof of \eqref{eq:surF2}.
We   compute the time derivative of $\overline {\mathcal F_2}$ using \fref{eqe}:
\bee
 \frac12\frac{d\overline {\mathcal F_2}}{ds}& = &  \frac {(\l_2)_s}2 \int y e^{\l_2 y} \e^2 + 
 \int  e^{\l_2 y} (\pa_s \e) \e\\
&  =& \frac {(\l_2)_s}2 \int y e^{\l_2 y} \e^2\\ & + &\int\left[(L \e)_y+\Gamma\Lambda Q_{b_2}+X(Q_{b_2})_y  + \frac {(\l_2)_s}{\l_2} {\Lambda} \varepsilon+E+F_y\right]e^{\l_2 y} \e .
\eee
Since $\int \Lambda \e e^{\l_2 y} \e = -\frac 12 \l_2 \int \e^2 y e^{\l_2 y}$, the scaling term cancels. By usual integrations by parts, we get
\bee
 \frac12\frac{d\overline {\mathcal F_2}}{ds}
&  =&  - \frac 32 \l_2 \int \e_y^2 e^{\l_2 y} - \frac 12 \l_2 (1- \l_2^2) \int  \e^2 e^{\l_2 y}
\\
&+&\int (-10Q^3Q_y+\tfrac 52 \l_2 Q^4) \e^2 e^{\l_2 y}
+ \int\left[\Gamma\Lambda Q_{b_2}+X(Q_{b_2})_y   +E+F_y\right]e^{\l_2 y} \e .
\eee
Using $\l_2\sim \frac 1{\sqrt{2|s|}}$ and the decay properties of $Q$,  we get
for $|s|$ large,
\begin{align*}
 \frac12\frac{d\overline {\mathcal F_2}}{ds}
& \leq - \frac 1{4 \sqrt{|s|}} \int (\e_y^2  + \e^2) e^{\l_2 y} 
+ \int \e^2 e^{-\frac {|y|}{10}}  \\
&+ \left| \int\left[\Gamma\Lambda Q_{b_2}+X(Q_{b_2})_y   +E+F_y\right]e^{\l_2 y} \e\right|.
\end{align*}
Now, we estimate the remaining terms.
First, from \eqref{boundone}, and the definition of $Q_b$,
\begin{align*}&\left| \Gamma\int \Lambda Q_{b_2} e^{\l_2 y} \e\right|+ \left|X  \int (Q_{b_2})_y   e^{\l_2 y} \e \right|\\
&\leq
C \left( \left( \int \e^2 e^{-\frac {|y|}{10}}\right)^{\frac 12}+ \frac {|b|}{|s|}\right) 
\left(\left( \int \e^2 e^{-\frac {|y|}{10}}\right)^{\frac 12}+ \frac {|b_2|}{\sqrt{\l_2}}  \left(\int \e^2 e^{\l_2 y} \right)^{\frac 12} \right)
\\
& \leq \frac {\l_2}{100} \int \e^2 e^{\l_2 y}+
C  \int \e^2 e^{-\frac {|y|}{10}}  + C \frac {b^2}{s^2}.
\end{align*}

Second, we estimate terms coming from $E$.
Since $$\int \Lambda \e_1 \e e^{\l_2 y} =
-\int \e_1 \Lambda \e e^{\l_2y} - \l_2 \int y \e_1\e e^{\l_2y},$$ we get
\bee
&&\left|(\Gamma-b) \int \Lambda \e_1 e^{\l_2 y} \e\right|\\
&&\lesssim
\left( \left( \int \e^2 e^{-\frac {|y|}{10}} \right)^{\frac 12}+ |b|\right) \left(\int (\e_y^2+\e^2) e^{\l_2y}\right)^{\frac 12}
\left(\int (1+y^2) \e_1^2 e^{\l_2 y} \right)^{\frac 12}.
\eee
We estimate the term in $\e_1$ using Proposition \ref{sharpbounds}, for some $0<\omega<\omega'<\frac 1{10}$,
using $\l_2 \geq \frac {19}{20} \l_1$ for $|s|$ large,
\begin{align*}
\int  y^2 \e_1^2 e^{\l_2 y} &\leq \int_{y<0} y^2 \e_1^2e^{\l_2 y} + \int_{y>0} y^2 \e_1^2 e^{\omega y}\\ 
& \lesssim  \sup_{y<0} \left[y^2 e^{\frac 1{20}\l_1 y}\right]    \int_{y<0}  \e_1^2 e^{\frac 9{10} \l_1 y}   + \int_{y>0}   \e_1^2 e^{\omega' y}\\
& \lesssim  \l_1^{-2} \left(\int_{y<0}   \e_1^2\right)^{\frac 1{10}}\left(\int_{y<0}  \e_1^2 e^{\l_1 y} \right)^{\frac 9{10}}
+ \int_{y>0}   \e_1^2 e^{\omega' y}\lesssim |s|^{-\frac 9{10}},
\end{align*}
we obtain
\begin{align*}
\left|(\Gamma-b) \int \Lambda \e_1 e^{\l_2 y} \e\right|&\lesssim
|s|^{-\frac 9{20}}\left(  \left(\int \e^2 e^{-\frac {|y|}{10}} \right)^{\frac 12}+ |b|\right) \left(\int (\e_y^2+\e^2) e^{\l_2y}\right)^{\frac 12}\\
& \leq C \int \e^2 e^{-\frac {|y|}{10}} + C\frac {|b|^2}{|s|^{\frac 25}} + \frac {\l_2}{100 } \int (\e_y^2+\e^2) e^{\l_2y}.
\end{align*}
  
  For the second term coming from $E$, we use \eqref{diffQ} and $|\Gamma_1|\lesssim \frac 1{|s|}$, so that
  \begin{align*}
  \left|\Gamma_1 \int \Lambda (Q_{b_2}-Q_{b_1}) \e e^{\l_2 y}\right|
& \lesssim \frac {|b|}{|s|} \left( \int_{y<0} |\e| e^{\l_2 y} +\left( \int \e^2 e^{-\frac {|y|}{10}}\right)^{\frac 12}\right)\\
& \lesssim \frac {|b|}{|s|} \left( |s|^{\frac 14}\left(\int_{y<0} \e^2 e^{\l_2 y}\right)^{\frac 12} + \left( \int \e^2 e^{-\frac {|y|}{10}}\right)^{\frac 12}\right)\\
 & \leq \frac {\l_2}{100} \int \e^2 e^{\l_2 y}+ C \int \e^2 e^{-\frac {|y|}{10}}  +C \frac {|b|^2}{|s|}.
  \end{align*}
  
For the last term in $E$, we use \eqref{diffPsiBIS} and argue similarly:
\begin{align*}
&\left| 
\int (\Psi_2-\Psi_1) \e e^{\l_2 y}
\right|   \\& \lesssim \left(\int \e^2 e^{-\frac {|y|}{10}}+\frac {\left(\int \e^2 e^{-\frac {|y|}{10}}\right)^{\frac 12}}{|s|}+ \frac {|b|}{\sqrt{|s|}}  \right)
\left( \int_{y<0} |\e| e^{\l_2 y} + \left(\int \e^2 e^{-\frac {|y|}{10}}\right)^{\frac 12}\right)\\
 & \leq \frac {\l_2}{100} \int \e^2 e^{\l_2 y}+ C \int \e^2 e^{-\frac {|y|}{10}} +C \frac {|b|^2}{|s|}.
\end{align*}

Now, we estimate terms coming from $F$. Arguing as before, since $|X_1|\lesssim \frac 1{|s|^{\frac 32}}$ and using \eqref{diffQ}, we obtain
$$
\left|
X_1 \int (Q_{b_2}-Q_{b_1})_y \e e^{\l_2 y}
\right|\lesssim
 \frac {\l_2}{100} \int \e^2 e^{\l_2 y}+\int \e^2 e^{-\frac {|y|}{10}}+ \frac {|b|^2}{|s|}.
$$
Next, since $|X_2| \lesssim \frac 1{|s|^{\frac 32}}$, $\l_2\lesssim |s|^{-\frac 12}$, by integration by parts
$$
\left| X_2 \int \e_y \e e^{\l_2 y}\right|
\lesssim \frac 1{s^2} \int \e^2 e^{\l_2 y} \leq \frac {\l_2}{100} \int \e^2 e^{\l_2 y}. 
$$
Then, by integration by parts,  \eqref{boundone} and \eqref{decayscaling},
\begin{align*}
\left| X \int (\partial_y \e_1) \e e^{\l_2 y} \right|& \lesssim \left(\left(\int \e^2 e^{-\frac {|y|}{10}}\right)^{\frac 12} + \frac {|b|}{|s|}\right)
\left(\int \e_1^2 e^{\l_2 y} \right)^{\frac 12} \left(\int (\e_y^2+\e^2) e^{\l_2 y} \right)^{\frac 12} 
\\ &\lesssim \frac {\l_2}{100} \int (\e_y^2+\e^2) e^{\l_2 y}+ \int \e^2 e^{-\frac {|y|}{10}} + \frac {|b|^2}{|s|}.
\end{align*}
Finally, we decompose  $R_2(\e_2) - R_1(\e_1)$ as follows, using $\|\e_i\|_{L^\infty} \lesssim |s|^{-\frac 12}$ and $|b_i|\sim |s|^{-1}$,
\begin{align*}
|R_2(\e_2) - R_1(\e_1)|
& \leq |R_2(\e_2)-R_2(\e_1)| + |R_2(\e_1)-R_1(\e_1)|
\\
&\lesssim |\e| \left(|b_2|+e^{-\frac {|y|}{10}}\right)+ |b| |\e_1|
\lesssim |s|^{-1} |\e|+ |\e| e^{-\frac {|y|}{10}} + |b| |\e_1|,
\end{align*}
and thus, using also \eqref{decayscaling},
we estimate the last term coming from $F$ as follows:
\begin{align*}
& \left| \int \left( R_2(\e_2) - R_1(\e_1)\right)  \left(\e_y + \l_2 \e\right) e^{\l_2 y} 
\right|\\
& \lesssim \frac 1{|s|} \int (\e_y^2 + \e^2) e^{\l_2 y} + \int \e^2 e^{-\frac {|y|}{10}} + 
|b| \left(\int \e_1^2 e^{\l_2 y}\right)^{\frac 12}\left(\int \e^2 e^{\l_2 y}\right)^{\frac 12}\\
& \lesssim \frac {\l_2}{100} \int (\e_y^2 + \e^2) e^{\l_2 y} + \int \e^2 e^{-\frac {|y|}{10}} + 
\frac {|b|^2}{|s|}.
\end{align*}

The collection of above bounds yields \fref{eq:surF2}. 

\subsection{Conclusion}

  
  We are now in position to conclude the proof of uniqueness \fref{equaltiyutwp}.

  First, recall that   the following estimates from Proposition \ref{sharpbounds} as $s\to -\infty$:
 \be\label{prems}
 \NNb\lesssim \frac{1}{|s|^3},\ \
 |J_2|\lesssim \NNb^{\frac 12} \lesssim \frac 1{|s|^{\frac 32}}, \ \
 b=o\left(\frac{1}{|s|^2}\right).
 \ee
  Moreover, we estimate from \fref{decayscaling}, \fref{sobolev}:
 \bee
  \int e^{\l_2 y} \e^2& \lesssim&  \int e^{\l_2 y} \e_1^2+ \int e^{\l_2 y} \e_2^2
  \lesssim  \int e^{\l_2 y} \e_1^2+\frac{1}{s^2}\\
  & \lesssim & \int_{y<0} e^{-\l_2 |y|} \e_1^2+\int_{y>0} e^{\l_2 y} \e_1^2+\frac{1}{s^2}\\
  & \lesssim & \int_{y<0} e^{-\frac{9}{10}\l_1 |y|} \e_1^2
  +\int_{y>0} e^{\frac y{10}} \e_1^2+\frac{1}{s^2}
  \lesssim \left(\int e^{-\l_1 |y|} \e_1^2\right)^{\frac9{10}}+\frac{1}{s^2}\\
  & \lesssim &  \frac{1}{s^{\frac 95}}.
  \eee
 This yields in particular from \fref{coercivityf} the bound: 
 \be
 \label{aprioriboudn}
 \matchal F\lesssim \frac{1}{s^{2+\frac3{10}}}.
 \ee
Recall from  \fref{eq:2006bis} and \fref{pointwise} (using $\int \e^2 e^{-\frac {|y|}{10}} \lesssim
\NNbl$):
$$\frac{d}{ds}\left\{s^2\left(b+\frac {J_2}{2s}\right)\right\}\lesssim s^2\NNbl+ |s|^{\frac 12} |b|+ |s|^{\frac 12}\NNb^{\frac12},$$ 
  \be\label{oepjope}
  \frac{d}{ds}(s^2\mathcal F)+\mu s^2\NNbl\lesssim |s|^{\frac {11}{10}}b^2 .
\ee
It follows that for $K_0>0$ large enough, using also \eqref{prems},
\be\label{couple}
\frac{d}{ds}\left\{s^2\left(b+\frac {J_2}{2s} +K_0  \mathcal F\right) \right\}
\lesssim  |s|^{\frac 12} |b|+ |s|^{\frac 12}\NNb^{\frac12} + s^{\frac{11}{10}} b^2\lesssim  |s|^{\frac 12} |b|+ |s|^{\frac 12}\NNb^{\frac12} . 
\ee
 We now observe from \eqref{prems}, \fref{aprioriboudn} the a priori bound:
 $$s^2\left(|b| +\frac{|J_2|}{|s|}+ \mathcal F\right)\to 0\ \ \mbox{as}\ \ s\to -\infty.$$ 
We then integrate \fref{oepjope}  on $(-\infty,s]$:
$$
  s^2 \mathcal F(s) \lesssim  \int_{-\infty}^s |s'|^{\frac {11}{10}} b^2(s') ds' \lesssim \int_{-\infty}^s (s')^4 b^2(s') \frac {ds'}{|s'|^{2+\frac9 {10}}} 
 \lesssim \frac 1 {|s|^{\frac {19}{10}}} \left(\sup_{(-\infty,s]} ((s')^2 |b(s')|\right)^2,
   $$
  so that 
  \be
  \label{cenkoenneo}
  |s| \mathcal F^{\frac 12}(s) \lesssim  \frac 1 {|s|^{\frac {19}{20}}}  \sup_{(-\infty,s]} (s')^2 |b(s')| .
  \ee 

 Next, by integration of \eqref{couple} on $(-\infty,s]$ and   \eqref{coercivityf} and \eqref{prems},
  \bea   s^2|b(s) | 
& \lesssim&|s||J_2|+   s^2 \mathcal F +\int_{-\infty}^s  \left(|s'|^{\frac 12} |b|+ |s'|^{\frac 12}\NNb^{\frac12}\right) 
 ds'\nonumber \\
& \lesssim&   |s|\mathcal \NNb^{\frac12}+   s^2 \mathcal F +\int_{-\infty}^s   |b| |s'|^2\frac {ds'}{|s'|^{\frac 32}} +\int_{-\infty}^s |s'|^{\frac 74}   \NNb^{\frac 12}  \frac {d s' }{|s'|^{\frac 54}}   \nonumber\\
&\lesssim& \frac{1}{ |s|^{\frac 14}}\sup_{(-\infty,s]}\left\{(s')^2|b (s')|+|s'|^{\frac 74}\mathcal F^{\frac 12}(s')\right\}. \label{undeplus}
  \eea 

Putting together \eqref{cenkoenneo} and \eqref{undeplus}, we get  \be 
  s^2 |b(s) |+ |s|^{\frac 74}\mathcal F^{\frac 12}(s)    \lesssim  \frac{1}{ |s|^{\frac 15}}\sup_{(-\infty,s]}\left\{(s')^2|b (s')|+(s')^{\frac 74}\mathcal F^{\frac 12}(s')\right\}.
  \ee 
This give immediately for $|s|$ large,  by     \eqref{coercivityf},
$$|b(s) |+\NNb(s)=0\  \ \mbox{and thus}\ \ \e(s,y)\equiv 0.$$
Therefore, for some $t>0$, $u_2(t)$ is a rescaled and translated of $S(t)$ and 
thus for all time by uniqueness of the Cauchy problem in $H^1$.
This concludes the proof of \fref{equaltiyutwp} and Theorem \ref{th:1}.


\section{Description of the (Exit) scenario}
\label{sec:6}


This section is devoted to the proof of Theorem \ref{PR:1}. The argument relies first on an extension of the compactness argument of section \ref{sectionconstruction} and  second on the uniqueness up to symmetries of the minimal mass blow up solution.


\subsection{Reduction of the proof}


Theorem \ref{PR:1} is a direct consequence of the following proposition which describes the defocusing bubble in the (Exit) regime at the exit time.

\begin{proposition}[Compactness of sequences of solutions at the (Exit) time]
\label{PR:6.1}
There exists a small universal constant $ \alpha^*>0$ such that the following holds. Let $(u_n(0))$
be a sequence in $H^1$  satisfying:
\begin{enumerate}
\item $u_n(0) \in \mathcal{A}$;
\item $\|u_n(0) - Q\|_{H^1} \leq \frac 1n$;
\item the solution $u_n\in \mathcal C([0,T_n),H^1)$ of \eqref{kdv} corresponding to $\left(u_n(0)\right)_{n\ge 1}$ satisfies the {\rm (Exit)} scenario, i.e. 
for all $n$ large enough,  
\be\label{exittime}
t_n^* = \sup\left\{t>0  \hbox{ such that $\forall t'\in [0,t],$ $ u_n(t')\in \mathcal{T}_{\alpha^*}$} \right\}<T_n.
\ee
\end{enumerate}
Then, there exists $\sigma^*=\sigma^*(\alpha^*)$ (independent of the sequence $u_n$) such that
\be\label{cl:6.1}
  \l_{n}  ^{\frac 12}(t_{n}^*) u_{n}\left(t_{n}^*,   \l_{n}(t_{n}^*)  \cdot+   x_{n}(t_{n}^*) \right) \to \l_S^{\frac 12} (\sigma^*)   S\left(\sigma^*, \l_S(\sigma^*) \cdot+x_S(\sigma^*)\right) \quad \hbox{in $L^2$} 
\ee
as $n\to +\infty$.
\end{proposition}

\subsection{Proof of Proposition \ref{PR:6.1}}
The strategy of the proof is similar to the proof of existence 
of Theorem \ref{th:1} in section \ref{sectionconstruction}. However the initial data in section \ref{sectionconstruction} are {\it well prepared} and in particular generate $H^1$ bounded sequences after renormalization, see \fref{TRT3}. Here the $H^1$ bound is lost, and one needs to invoke a concentration compactness argument in the critical $L^2$ space for sequences of solutions to \fref{kdv} and uniform local estimates to recover a non trivial weak limit.\\
 
{\bf step 1} Renormalization. Let $C^*>0$ be the universal constant in \fref{conrolbintegre} of Lemma \ref{le:oubli}. Let $t_n^*$ be the exit time \eqref{exittime}, and consider the   decomposition of $u_n(t)$ on $[0,t_n^*]$ given by Lemma \ref{le:2}.  It follows from the proof of
Theorem 1.2 in \cite{MMR1}, Section~4.3, that there exists  a time $0\leq t_{1,n}^*< t_n^*$ such that
\be\label{deftun}
b_n(t_{1,n}^*)\leq -C^* \int \left((\partial_y \e_n)^2(t_{1,n}^*)\psi_B + \e_n^2(t_{1,n}^*) \varphi_B'\right),
\ee
and
\be\label{6.0}
\l_n(t_{1,n}^*)\approx 1, \quad
|b_n(t_{1,n}^*)|+ \mathcal N_n(t_{1,n}^*)\to 0 \quad \hbox{as $n\to \infty$},
\ee
where $\mathcal N_n$ denotes the quantity $\mathcal N$ defined in \eqref{eq:no}
for $u_n$.
This time corresponds to when the (Exit) regime is decided ($b_n$ is negative and becomes predominant   in the sense \eqref{deftun}), and it is proved in \cite{MMR1} that such a time $t^*_{1,n}$ can be chosen so that  the solution has moved only $\delta(\|u_n(0)-Q\|_{H^1})$ far from the initial data (see equation (4.37) in \cite{MMR1}), which implies \eqref{6.0} in the present situation 
(since $\|u_n(0)-Q\|\to 0$ as $n\to \infty$).

Recall also from \cite{MMR1} that $u_n(t)$ satisfies (H1), (H2) and (H3) on $[0,t_n^*]$.

Define
$$\forall \tau\in \left [\tau_n^*,0\right], \quad t_\tau= t_n^*+\tau \l_n^3(t_n^*),\ \ \tau_n^*=- \frac {t_n^*-t_{1,n}^*}{\l_n^3(t_n^*)}$$
 and consider on
 $[\tau_n^*,0]$ the renormalized solution $v_n(\tau)$ at the exit time $t_n^*$,
\bea\label{eq:6.4}
v_n(\tau,x)&&
= \l_n^{\frac 12}(t_n^*) u_n\left(t_\tau,{\l_n(t_n^*)} x+x(t_n^*)\right)
\\ &&= \frac {\l_n^{\frac 12}(t_n^*)}{\l_n^{\frac 12}(t_\tau)}
(Q_{b_n(t_\tau)}+\e_n)\left(t_\tau, 
\frac {{\l_n(t_n^*)} }{\l_n(t_\tau)} x + \frac { x_n(t_n^*)-x_n(t_\tau)}{\l_n(t_\tau)} \right).
\eea
Then $v_n$ is solution of \eqref{kdv} and  belongs to the $L^2$ tube $\mathcal T_{\alpha^*}$ for $\tau\in [\tau_n^*,0]$. Moreover, its  decomposition $(\lambda_{v_n} ,x_{v_n},\e_{v_n})$   satisfies
 on $[\tau_n^*,0]$:
\begin{align}
\l_{v_n}(\tau) = \frac {\l_n(t_\tau)}{\l_n(t_n^*)},  \
x_{v_n}(\tau)= \frac  {x_n(t_\tau)-x_n(t_n^*)}{\l_n(t_n^*)} ,\
b_{v_n}(\tau) = b_n(t_\tau),\
\e_{v_n}(\tau) = \e_n(t_\tau).
\end{align}
 
\medskip

\textbf{step 2} Preliminary estimates on the renormalized sequence. We claim:

\begin{lemma}\label{le:6.1}
There exist $b^*$, $\tau^*$ such that, 
  possibly extracting a subsequence,
\be\label{6.19}
b_n(t_n^*)\to -b^*,\quad 
(\alpha^*)^2 \lesssim b^* \leq \delta (\alpha^*),
\ee
\be\label{6.20}
 \tau_n^* = -\frac{t_n^*}{\l_n^3(t_n^*)}\to - \tau^*,\quad
  \tau^* b^*  \approx 1
 .
\ee
Moreover, for all $n$ large, $\tau\in [ \tau^*_n,0]$, 
\be\label{6.11}
\l_{v_n}(0)= 1, \quad x_{v_n}(0)=0,
\ee
\be\label{6.12}
   | b_{v_n}(\tau)|+
\mathcal N_{v_n}(\tau)+\|\e_{v_n}(\tau)\|_{L^2} \lesssim \delta(\alpha^*),
\ee
\be
\int_{\tau_{1,n}^*}^{\tau_n^*}\left[\int((\partial_y\e_{v_n})^2+\e_{v_n}^2)\frac{\varphi_B}{1+y_+^2}+|b_{v_n}|^4\right]\frac{d\tau}{\lambda_{v_n}^5(\tau)} \lesssim \delta(\alpha^*),
\label{6.perdu}
\ee
\be\label{6.13}
\frac {|b_{v_n}(\tau_n^*)|}{\lambda_{v_n}^2(\tau_n^*)}\lesssim \delta(\alpha^*), \quad
\l_{v_n}\left(  \tau_n^* \right) = \frac 1{\l_n(t_n^*)}
\to 0 \quad \hbox{as $n\to +\infty$},
\ee
\be\label{6.15}
 (\l_{0,v_n})_\tau(\tau) \approx  b^*,
\ee
\be\label{6.100}
x_{v_n}(\tau_n^*) \lesssim - \frac 1{ b^*  \l_{v_n}(\tau_n^*)} \to -\infty 
\quad \hbox{as $n\to +\infty$}.
\ee
 \end{lemma}
\noindent (Recall that $\lambda_0$ is defined in Lemma \ref{le:oubli}. Here, $\l_{0,v_n}$ denotes this quantity for $v_n$. Similarly, $\mathcal N_{v_n}$ denotes the quantity $\mathcal N$ for $v_n$. As usual $y_+=\max(0,y)$.)
\begin{proof}[Proof of Lemma \ref{le:6.1}]
Arguing as in the proof of Lemma \ref{le:2.4}, using conservation of mass and energy of $u_n(t)$, 
we  first obtain
\be\label{6.5}
 (\alpha^*)^2 \lesssim   - b_n(t^*_n) \lesssim \delta(\alpha^*).
\ee
Next, using \eqref{conrolbintegre} in Lemma \ref{le:oubli} on $[t_{1,n}^*,t_n^*]$, and \eqref{deftun}, \eqref{6.0}, one obtains
\be\label{6.7}
 - \frac {b_n(t)}{\l_n^2(t)} \approx -   {b_n(t_{1,n}^*)}  
\ee
and thus, by \eqref{6.5},
\be\label{6.8}
  \frac { (\alpha^*)^2}{|b_n(t_{1,n}^*)|} \lesssim \l_n^2(t_n^*) \lesssim  \frac {\delta(\alpha^*)}{|b_n(t_{1,n}^*)|},
\ee
which implies $ \frac {|b_{v_n}(\tau_n^*)|}{\lambda_{v_n}^2(\tau_n^*)}\lesssim\delta(\alpha^*)$.
Next, by definition of $t_n^*$, $\|\e_{v_n}(\tau)\|_{L^2}  
=\|\e_{n}(t_\tau)\|_{L^2}\leq \delta(\alpha^*)$.
By \eqref{estfondamentale} and \eqref{6.0}, for $n$ large,
\begin{align}
&\mathcal N_n(t) +\int_{t_{1,n}^*}^{t_n^*}\left[\int((\partial_y\e_n)^2+\e_n^2)\varphi_B' \right]\frac{dt}{\lambda_n^3}\nonumber
\\ &\lesssim \mathcal N_n(t_{1,n}^*) + |b_n(t)|^3+ |b_n(t_{1,n}^*)|^3
\leq \delta(\alpha^*).\label{6.10}
\end{align}
Now, we use Lemma 4.3 in \cite{MMR1} to obtain a slightly different estimate. 
From (4.12) in \cite{MMR1}, with $i=1$, using the definition of $\varphi_{i,B}$ in page 84, and then using \eqref{deftun}, we obtain
\begin{align}
&\int_{t_{1,n}^*}^{t_n^*}\left[\int((\partial_y\e_n)^2+\e_n^2)\frac{\varphi_B}{1+y_+^2}+|b_n|^4\right]\frac{dt}{\lambda_n^5}\nonumber \\
&\lesssim  \frac{\int ((\partial_y\e_n)^2(t_{1,n}^*)\psi_B+\e_n^2(t_{1,n}^*)\varphi_B')}{\lambda^2_n(t_{1,n}^*)
}+ \frac {|b_n(t_n^*)|^3}{\l_n^2(t_n^*)}+ \frac{|b_n(t_{1,n}^*)|^3}{\lambda_n^2(t_{1,n}^*)}
\lesssim \frac{|b_n(t_{1,n}^*)|}{\lambda^2_n(t_{1,n}^*)}\nonumber
\\ &\lesssim \delta(\alpha^*).
\label{6.14}\end{align}
Moreover, by \eqref{dvnkoenneoneor} and \eqref{6.7},
\be\label{6.18}
  (\l_{0n})_t(t)
  \approx    \frac {|b_n(t_{n}^*)|}{\l_n^2(t_{n}^*)} .
\ee
By definition of $v_n$, we obtain \eqref{6.11}--\eqref{6.15} from the above estimates.

\medskip

Now, we prove \eqref{6.100}.
Integrating the estimate of $(\l_{0,v_n})_\tau$, we obtain
$$
  \l_{0,v_n}(\tau) - \l_{0,v_n}(\tau_n^*) \approx   b^* (\tau-\tau_n^*)
$$
Finally, since
$(x_{v_n})_\tau \approx \frac 1{\l_{0,v_n}^2}$, we obtain by integration of $[\tau_n^*,0]$ and
$x_{v_n}(0)=0$, for $n$ large,
$$
-x_{v_n}(\tau_n^*)  =    x_{v_n}(0) - x_{v_n}(\tau_n^*) \approx \frac 1{b^*} \frac 1{\l_{v_n}(\tau_n^*)}
$$
and \fref{6.100} is proved.
\end{proof}

\medskip

\textbf{step 3} Monotonicity estimates.
We now claim the following bound on $v_n$ which will allow us to recover $H^1$ bounds in the limit:
\be\label{borneh1vn}
\int_{x>- \l_{n}^{2}(t_n^*)} (\partial_x v_n)^2(0,x) dx \lesssim 1.
\ee
In fact, we prove the following estimate on $\e_{v_n}(0)$ which, together with Lemma \ref{le:6.1} and $\lambda_{v_n}(0)=1$ implies   \eqref{borneh1vn}
\be\label{monotoniepsvn}
\int_{y>- 2 \l_{n}^{2}(t_n^*)}  (\partial_y \e_{v_n}) ^2 (0,y)dy  \lesssim \delta(\alpha^*).
\ee

\emph{Proof of \eqref{monotoniepsvn}.}
For $\tau\in[\tau_n^*,0]$, let $s=-\int_\tau^0 \frac {d\tau}{\l_{v_n}^3(\tau)}$
the rescaled time for $v_n$ and $s_n^*= -\int_{\tau_n^*}^0 \frac {d\tau}{\l_{v_n}^3(\tau)}$.
We perform monotonicity estimates on $\e_{v_n}$ to complement the ones obtained in \eqref{estfondamentale}.
We define $\phi \in \mathcal C^{\infty}(\RR)$, such that
\bea
\label{defphi}
\phi(y) =\left\{\begin{array}{lll}e^{y}\ \ \mbox{for} \ \   y<-1,\\
 1-\tfrac1{10}e^{-y}   \ \ \mbox{for}\ \ y>-\frac 12  ,\\ 
 \end{array}\right.  \ \ \phi'(y) > 0, \ \ \forall y\in \RR,
\eea
and we consider $\psi$ defined as in \eqref{defphi2}.
Let   
$$
\phi_B(s,y)=\phi\left(\frac {y+\frac 12 (s-s_n^*)}{B} \right),
\quad 
\tilde \phi_B(s,y)=\psi\left(\frac {y+\frac 12 (s-s_n^*)}{B} \right),
$$
and
\begin{multline*}
\mathcal F_{B}(s)= \frac{1}{\lambda_{v_n}^2(s)}\int \left[ (\partial_y \e_{v_n})^2 \tilde \phi_B + \frac{\lambda_{v_n}^2(s)}{\lambda_{0,v_n}^2(s)}\e_{v_n}^2 \phi_B \right.\\ \left. 
- \frac 13 \left( (\e_{v_n}+Q_b)^6 - Q_b^6 - 6 \e_{v_n} Q_b^5\right)\tilde \phi_B \right](s,y) dy.
\end{multline*}
We claim the following estimates proved in Appendix \ref{app:B}.

\begin{lemma}\label{le:6.4}
For $B$ large enough,
\begin{align}
\frac d{ds} \mathcal F_{B} &\leq
 - \frac 14\frac{1}{\lambda_{v_n}^2} \int ((\partial_y \e_{v_n})^2 + \e_{v_n}^2) \phi_B' \nonumber\\
&+  \frac{C}{\lambda_{v_n}^2}
\int ((\partial_y \e_{v_n}) ^2+ \e_{v_n}^2) \frac{\varphi_B}{1+y_+^2}+ \frac{C}{\lambda_{v_n}^2} |b_{v_n}|^4,\label{monoepsvn}
\end{align}
\be\label{coerepsvn}
 \mathcal F_B \approx \frac{1}{\lambda_{v_n}^2} \int   (\partial_y \e_{v_n}) ^2 \tilde \phi_B + \e_{v_n}^2 \phi_B 
\ee
\end{lemma}

\medskip

We integrate \eqref{monoepsvn} on $[\tau_n^*,0]$ and then use \eqref{6.14} and Lemma \ref{le:6.1}:
\begin{align*}
\mathcal F_B (0) &\leq \mathcal F_B(\tau_n^*)  + C \int_{\tau_n^*}^0 \left(\int ((\partial_y \e_{v_n})^2+ \e_{v_n}^2) \frac{\varphi_B}{1+y_+^2}+  |b_{v_n}|^4\right) \frac {d\tau} {\l_{v_n}^5(\tau)}\\
&\lesssim \mathcal N_{v_n}(\tau_n^*)   + \frac {|b_{v_n}(\tau_n^*)|^3}{\lambda_{v_n}^2(\tau_n^*)} + |b_{v_n}(0)|^3
\lesssim \delta(\alpha^*).
\end{align*}

And thus, by \eqref{coerepsvn}, $\lambda_{v_n}(0)=1$ and the definition of $\phi_B$, $\tilde \phi_B$,
since 
$$s_n^*= -\int_{\tau_n^*}^0 \frac {d\tau}{\l_{v_n}^3(\tau)}
\approx -\int_{\tau_n^*}^0 \frac {d\tau}{\left(\l_{v_n}(\tau_n^*)+b^*(\tau-\tau_n^*)\right)^3}
\approx -\frac 1{b^*} \frac 1{\l_{v_n}^2(\tau_n^*)},$$
we finally obtain
\be
\int_{y>-  2\l_{v_n}^{-2}(\tau_n^*)}  (\partial_y \e_{v_n})^2 \leq 
\int_{y>\frac 12 s_n^*} (\partial_y \e_{v_n})^2 \lesssim \delta(\alpha^*).
\ee

\medskip

\textbf{step 4} Extraction of the limit. Since $\|v_n(0)-Q\|_{L^2} \leq \delta(\alpha^*)$, there exists
$v(0)\in L^2$ and a subsequence still denoted $(v_{n}(0))$ such that 
$$
v_{n}(0)\rightharpoonup v(0) \hbox{ in $L^2$ weak as $n\to \infty$}.
$$
Moreover, by properties of the weak convergence,
\be\label{deltaproche}
\|v(0)\|_{L^2} \leq \|Q\|_{L^2},\ \ \|v(0)-Q\|_{L^2}\leq \delta(\alpha^*)
\ee
and since $\l_n(t_n^*)\to \infty$, 
it follows from \eqref{borneh1vn} that 
$$v(0)\in H^1.$$

Let $\delta_0>0$ small enough, $\delta_0<\delta$ where $\delta$ is defined in Theorem \ref{KPV} (i).
We consider $\alpha^*$ small enough, but universal, such that 
\be\label{delta}\|v_n(0)-v(0)\|_{L^2}\leq \frac  {\delta_0} 2,
\ee
 In order to exhibit a non trivial weak limit, we decompose the sequence $(v_n(0)-v(0))$ into profiles according to Lemma \ref{le:profile}: there exist
$$
U_n^j(0) = e^{-t_n^j\partial_x^3} \left(g_n^j[{\rm Re}(e^{ix\xi_n^j\lambda_n^j}\phi^{j})]\right)
$$
and
$w_n^J(0) \in L^2
$
such that (up to a subsequence)
\be\label{profff}
v_n(0) - v(0) = \sum_{j=1}^J U_n^j(0)+ w_n^J(0) , \quad 
\lim_{J\to \infty}\limsup_{n\to \infty} 
\bigl\| e^{-t\partial_x^3}w_n^J(0)\bigr\|_{L^5_xL^{10}_t(\RR\times\RR)} 
=0,
\ee
\be\label{orthoprof}
\|v_n(0)-v(0)\|_{L^2}^2 - \sum_{j=1}^J \|U_n^j(0)\|_{L^2}^2 - \|w_n^J(0)\|_{L^2}^2 = o_n(1).
\ee
Moreover,   by weak convergence $v_n(0)\rightharpoonup v(0)$, we have
$$\int v(0) (v_n(0)-v(0))=o_n(1),$$ and thus,
\be\label{orthodeux}
\|v_n(0)\|_{L^2}^2 - \|v(0)\|_{L^2}^2 - \sum_{j=1}^J \|U_n^j(0)\|_{L^2}^2 - \|w_n^J(0)\|_{L^2}^2 = o_n(1)
\ee
In particular, by \eqref{profff} and \eqref{orthodeux}, $v(0)$ is interpreted as the first profile $U^0$ of the decomposition of $v_n(0)$ with $g_n^0=g_{1,0}$, $t_n^0=0$ 
and $\lambda_n^0=1$. By  \eqref{delta} and  \eqref{orthoprof}, for $n$ large,
$$
\sum_{j=1}^J \|U_n^j(0)\|_{L^2}^2 + \|w_n^J(0)\|_{L^2}^2 \lesssim \delta(\alpha^*) \leq \frac {\delta_0^2}2.
$$

Define $U_n^j(\tau)$ and $w_n^J(\tau)$ the (global) solutions of the nonlinear equation \eqref{kdv} corresponding to
the initial data $U_n^j(0)$ and $w_n^J(0)$. Let $\tau_0<0$ be such that $v(\tau)$ exists on $[\tau_0,0]$. 
We claim that,
 for $n$ large, $v_n$ exists on $[\tau_0,0]$ and
\be\label{claimnn}\lim_{J\to +\infty} \limsup_{n\to +\infty}\sup_{\tau \in [\tau_0,0]}
\biggl\|v_n(\tau)-v(\tau) - \sum_{j=1}^J U_n^j(\tau)  - w_n^{J}(\tau)\biggr\|_{L^2} =0.
\ee
Indeed, \eqref{claimnn} is a by now standard corollary of the perturbation Lemma \ref{le:perturbation}
(see e.g. Proposition 2.8 in \cite{DKM}),
 In particular, there exist $n_0>1$ and $J_0\geq 1$ such that
for $n>n_0$,  
$$
\biggl\|v_n(\tau)-v(\tau) - \sum_{j=1}^{J_0} U_n^j(\tau)  - w_n^{J_0}(\tau)\biggr\|_{L^2}\leq \delta_0,
$$
and thus, for all $\tau\in [\tau_0,0]$,
\be\label{tttt}
\|v_n(\tau)-v(\tau) \|_{L^2}\leq  \left\|\sum_{j=1}^{J_0} U_n^j(\tau)\right\|_{L^2}  +\| w_n^{J_0}(\tau) \|_{L^2}+ \delta_0
\lesssim \delta_0< \frac 1{10}\|Q\|_{L^2},
\ee
choosing now $\delta_0$ small but universal. In particular, let $A$ be such that, for all $\tau\in [\tau_0,0]$,
$$
\int_{|x|>A} v^2(\tau,x) dx < \frac 1{100}   \int Q^2
$$
and thus from \eqref{tttt}:
$$\int_{|x|>A} v_n^2(\tau,x)dx \leq \frac 1{25}  \int Q^2.$$ 

Now, recall from \eqref{6.100} that  $x_{v_n}(\tau_n^*)\to -\infty$ as $n\to \infty$, and in particular, for $n$ large enough,
$- x_{v_n}(\tau_n^*) \ll A$ and thus,
$$
\int_{x<-A} v_n^2(\tau_n^*,x)dx\geq \frac 34 \int Q^2.
$$
We conclude that necessarily $\tau_0>\tau_n^*$ for $n$ large enough, and thus $$\tau_0\geq -\tau^*=\lim_n \tau_n^*.$$ It follows that $v(\tau)$ blows up in a finite time $\tau_{\rm max}(v)\geq -\tau^*=\lim_n \tau_n^*$.
Since $\|v(0)\|_{L^2}\leq \|Q\|_{L^2}$ and $v(0)\in H^1$, we have 
$\|v\|_{L^2}= \|Q\|_{L^2}$. In particular, by weak convergence, and 
$\lim_{n\to \infty} \|v_n(0)\|_{L^2} = \|v(0)\|_{L^2}$, we obtain
$\lim_{n \to \infty} \|v_n(0)- v(0)\|_{L^2} = 0$.

From the uniqueness statement in Theorem \ref{th:1}, there exists $\l^*>0$, $x^*\in \RR$
and $\sigma^*>0$ such that
$$
v(0,x)= (\l^*)^{\frac 12} S(\sigma^* ,\l^* x + x^*).
$$
Moreover, denoting by $(b_S,\lambda_S,x_S)$ the parameters of the decomposition of $S$, we observe that 
$$ 
\l_{v}(0)=1 = (\l^*)^{-1}   \l_S(\sigma^*),\quad x_{v}(0)=0=   x_S(\sigma^*) -x^*,
$$ 
and thus
$$v(0,x)=  \l_S^{\frac 12} (\sigma^*)   S\left(\sigma^*, \l_S(\sigma^*) x+x_S(\sigma^*)\right).
$$
In particular, by scaling
$$v(\tau,x)=\l_S^{\frac 12} (\sigma^*)   S\left(\sigma^* + \l_S^{3}(\sigma^*) \tau, \l_S(\sigma^*) x+x_S(\sigma^*)\right).
$$
Since $v$ blows up at $\tau_{\rm max}(v)$, and $S$  blows up at time $0$ (by convention), we have
$$
\sigma^* = -\l_S^{3}(\sigma^*) \tau_{\rm max}(v) < \l_S^{3}(\sigma^*) \tau_*.
$$

From the definition of $t_n^*$ and then strong $L^2$ convergence, we have
\begin{align}
  \alpha^* & = \inf_{\l_1,x_1} \|u_n(t_n^*) - Q_{\l_1}(.-x_1)\|_{L^2}
   = \inf_{\l_1,x_1} \|v_n(0) - Q_{\l_1}(.-x_1)\|_{L^2}\nonumber\\
   & 
   = \inf_{\l_1,x_1} \|S(\sigma^*) - Q_{\l_1}(.-x_1)\|_{L^2}.\label{distance}
\end{align}
Moreover, recall that by definition of $t_n^*$ and $\tau_n^*$, for all $\tau \in [\tau_n^*,0]$,
we have
$$
\inf_{\l_1,x_1} \|v_n(\tau) - Q_{\l_1}(.-x_1)\|_{L^2}\leq \alpha^*,
$$
and so for all $t\in (0,\sigma^*]$, 
\begin{equation}\label{trap}
\inf_{\l_1,x_1} \|S(t) - Q_{\l_1}(.-x_1)\|_{L^2}\leq   \alpha^*.
\end{equation}
From \eqref{derive}, we fix $t_S>0$ such that the distance of $S(t)$
to the family of solitons is increasing on $(0,t_S)$.
Take $\alpha^*>0$ small enough  so that
$$
   \alpha^* \leq \frac 12 \inf_{\l_1,x_1} \|S(t_S) - Q_{\l_1}(.-x_1)\|_{L^2}.
$$
By \eqref{trap}, it is clear that  $\sigma^*\in (0,t_S)$. Moreover, 
$\sigma^*$ is uniquely defined by \eqref{distance} (small for $\alpha^*$  small) and thus does not depend on the subsequence  but only on $\alpha^*$.
In particular, all the sequence converges to the same limit and the proposition is proved.


\appendix



\section{End of the proof of Proposition \ref{sharpbounds}}\label{AA}

In this appendix, we finish the proof of Proposition \ref{sharpbounds} by proving 
\eqref{decayscaling} and 
\eqref{sobolev}-\eqref{sobolevbis} in the framework of Proposition \ref{sharpbounds}. For the reader's convenience, we recall the main estimates proved at this point on $\e$ and the parameters $b,\l,x$:
 for $|s|$ large,
\be\label{estrescbisAA}
\|\e(s)\|_{L^\infty}  \lesssim \|\e(s)\|_{H^1}\lesssim \frac 1{\sqrt{|s|}},\ \ \frac{c_1(u_0)}{\sqrt{|s|}}\leq \lambda(s)\leq \frac{c_2(u_0)}{\sqrt{|s|}},
\quad b(s)\sim \frac{1}{2s},
\ee

\be\label{estrescAAbis}
  \mathcal N(s)+   \int_{-\infty}^s \int (\e_y^2 + \e^2)(s') \varphi_B' ds' \lesssim \frac 1{|s|^3},  
\ee
\be\label{estrescAA}
 \left|\lsl+b\right| + \left| \xsl-1\right| \lesssim
\frac 1 {|s|^{\frac 32}}, \quad |b_s|\lesssim \frac 1{|s|^2}.
\ee

\subsection{Proof of \eqref{decayscaling}}
Since  $u(t)$ is  a minimal mass blowing up solution and $\l(s)$ is increasing for $|s|$ large,  
from Lemma \ref{lemmadecay} and then using the properties  of 
$Q_b$ (see Lemma \ref{cl:2}), we obtain for $|s|$ large,
\be
  \label{controlrighttri}
 \hbox{ $  \forall y>0$, }  |\e(s,y)| \lesssim  e^{-\frac y{20} }
\ee
Thus, by \eqref{estrescbisAA},  
\be\label{tendverszero}
\lim_{s\to -\infty} \int \e^2(s,y) e^{\l(s) y} dy =0.
\ee

Now, to prove \eqref{decayscaling}, we compute the time derivative of $\int \e^2 e^{\l y}$.
Using the equation of $\e$ (see \eqref{eqofeps}), we have
\begin{align*}
\frac 12 \frac d{ds} \int \e^2 e^{\l y}& = \frac {\l_s}{2} \int y e^{\l y} \e^2 + \int \e_s \e e^{\l y} 
\\ & = \frac {\l_s}{2} \int y e^{\l y} \e^2 + \int (L\e)_y\e  e^{\l y} + 
\left( \lsl + b\right) \int \Lambda Q_b \e e^{\l y}   + \lsl \int \e \Lambda \e  e^{\l y} \\ &+
\left(\xsl-1\right) \int (Q_b +\e)_y \e e^{\l y} 
+\int \Phi_b \e e^{\l y} + \int \Psi_b \e e^{\l y}\\ & -\int (R_b(\e))_y \e e^{\l y}
- \int (R_{NL}(\e))_y \e e^{\l y}.
\end{align*}
Since $\int \e \Lambda \e e^{\l y} = -\frac 12 \l \int \e^2 y e^{\l y}$, the scaling terms cancel
(this is because the quantity is scaling invariant).
Next, using \eqref{estrescAA}, we have:
\begin{align*}
\int (L\e)_y \e e^{\l y} & = 
- \frac 32 \l \int \e_y^2 e^{\l y} - \frac 12 \l(1-\l^2) \int \e^2 e^{\l y} 
+ \int (-10 Q^3 Q' - \tfrac 52 \l Q^4) \e^2 e^{\l y}\\
& \leq - \frac {\l} 4 \int (\e_y^2 + \e^2) e^{\l y} + C \mathcal{N}_{\rm loc},
\end{align*}
\bee
&&  \left|\left( \lsl + b\right) \int \Lambda Q_b \e e^{\l y}  \right|+
\left|\left(\xsl-1\right) \int  (Q_b)_y  \e e^{\l y}\right|\\
 & \lesssim &  \frac 1{|s|^{\frac 32}}\left[\left(\int \e^2e^{-\frac{|y|}{10}}\right)^{\frac 12}+|b| \left(\int \e^2 e^{\l y} \right)^{\frac 12}\left(\int_{y<0}e^{\l y}\right)^{\frac 12}\right] \\ 
 &\lesssim& \frac 1{|s|^{\frac 32}}\left[\frac{1}{|s|^{\frac 32}}+\frac{1}{\sqrt{|s|}} \left(\int \e^2 e^{\l y} \right)^{\frac 12}\right] \lesssim  \frac {\l}{100}   \int \e^2 e^{\l y} + \frac 1{|s|^{3}},
\eee
$$
 \left|\left(\xsl-1\right) \int \e_y \e e^{\l y} 
  \right|
  =\frac {\l} 2  \left| \left(\xsl-1\right) \int  \e^2 e^{\l y} \right|
  \leq \frac {\l}{100}\int  \e^2 e^{\l y}.
$$
\begin{align*}
\left|\int \Phi_b \e e^{\l y} \right|&\lesssim |b_s| \int
|P| |\e| e^{\l y}\lesssim \frac 1{|s|^{2}} \left(\int
P^2 e^{\l y}\right)^{\frac 12}  \left(\int \e^2 e^{\l y}\right)^{\frac 12}\\
& \lesssim \frac 1{|s|^{\frac 74}}\left(\int \e^2 e^{\l y}\right)^{\frac 12}
\leq \frac {\l}{100}\int  \e^2 e^{\l y} + \frac C{|s|^{3}}.
\end{align*}
Using \eqref{eq:202} (recall $\gamma =\frac 34$),  
\begin{align*}
&\left|\int \Psi_b \e e^{\l y} \right|
 \lesssim  {|b|^{\frac 74}} \int_{-2|b|^{-\frac 34}<y<-|b|^{-\frac 34}}
  |\e| e^{\l y} +  {|b|^2} \int_{y<0} |\e| e^{\l y} + {|b|^2} \int_{y>0}   |\e|  e^{-\frac y4} \\&
    \lesssim \frac 1{|s|^{\frac 74}}
 \left(\int_{y<-|b|^{-\frac 34}}  e^{\l y}\right)^{\frac 12} \left(\int \e^2 e^{\l y}\right)^{\frac 12} + \frac 1{|s|^2} \left(\int_{y<0}  e^{\l y}\right)^{\frac 12} \left(\int \e^2 e^{\l y}\right)^{\frac 12}+\frac 1{|s|^3}    \\& \lesssim \left( \frac 1{|s|^{\frac 32}} e^{- \l |b|^{-\frac 34}} + \frac 1{|s|^{\frac 74}}\right)\left(\int \e^2 e^{\l y}\right)^{\frac 12}
+\frac {1}{|s|^3}\leq \frac {\l}{100}\int  \e^2 e^{\l y} + \frac C{|s|^{3}}.
\end{align*}
Next, since $|R_b(\e)|=5|Q_b^4 - Q^4||\e| \lesssim |b||\e|$, we have
\begin{align*}
\left| \int (R_b(\e))_y \e e^{\l y} \right| 
& \leq \left| \int R_b(\e) (\l |\e|+|\e_y|) e^{\l y} \right|
\\ & 
\lesssim 
|b| \int    (\e_y^2 + \e^2) e^{\l y} \leq \frac \l{100} \int    (\e_y^2 + \e^2) e^{\l y}.
\end{align*}
Finally, since
$$
|R_{\rm NL}(\e)| \leq \|\e\|_{L^\infty} (\|\e\|_{L^\infty}^3 + |b| + e^{-\frac {y}{10}}) |\e|,
$$
we get
\begin{align*}
\left| \int (R_{\rm NL}(\e))_y \e e^{\l y} \right| 
& \leq  \left| \int  R_{\rm NL}(\e) (\l |\e| + |\e_y|)  e^{\l y} \right| 
\\ & 
\lesssim 
 \frac 1{|s|^{\frac 32}}  \int    (\e_y^2 + \e^2) e^{\l y} 
 \leq
 \frac \l{100} \int    (\e_y^2 + \e^2) e^{\l y}.
  \end{align*}

The collection of above bound ensures $$
  \frac d{ds} \int \e^2 e^{\l y}  \lesssim \frac 1{|s|^3}
$$
which integration on $(-\infty,s]$ using \eqref{tendverszero} yields \eqref{decayscaling}.

\subsection{Proofs of \eqref{sobolev}--\eqref{sobolevbis}}
Note first that by standard arguments,  
$$
\left\| \left( (\partial_y^2 \e)^2+(\partial_y  \e)^2  \right)(s ) e^{\omega  y  }\right\|_{L^\infty}
\lesssim
\int \left((\partial_y^3 \e)^2 +(\partial_y^2 \e)^2+ (\partial_y \e)^2 + \e^2\right)(s,y) e^{\omega  y  } dy,
$$
and so  it is sufficient to prove \eqref{sobolev}.

The proof is similar to Section 3.4 in \cite{Ma1} and involves some computations originally introduced in \cite{Kato}.
To prove \eqref{sobolev}, we need only rough bounds on $\e$ and it is therefore simpler to decompose $$\e+Q_b=\et+Q$$ which satisfies: 
\be
\label{cmneopmep}
\pa_s\et+\pa_y(\pa_y^2\et-\et+F(\et))=\lsl(\Lambda Q+\Lambda \et)+\left(\frac{x_s}{\l}-1\right)(\pa_yQ+\pa_y\et),
\ee
with $$F(\et)=(Q+\et)^5-Q^5.$$

From \eqref{estrescbisAA}, \eqref{estrescAAbis} and $Q_b-Q= b P \chi_b$ (see Lemma \ref{cl:2}) we have the following estimates on $\et$:
\be\label{et}
\|\et\|_{L^\infty} \lesssim \frac 1{\sqrt{|s|}},\quad
\int \et^2 e^{-\frac {|y|}{10}} \leq \frac 1{|s|^2}.
\ee
From \eqref{estrescAAbis},  
$$
\int \e^2(s) \varphi_{B} + \int_{-\infty}^s \int (\e_y^2 + \e^2)(s')\varphi_{B} ds'
\lesssim \frac  1{|s|^3}
$$
and thus, since $|b(s)|\lesssim \frac 1{|s|}$,  for $|s|$ large,
\be\label{initialisation}
\int \et^2(s) \varphi_{B} + \int_{-\infty}^s \int (\et_y^2 + \et^2)(s')\varphi_{B} ds'
\lesssim \frac  1{|s|}.
\ee
Moreover, 
since  $u(t)$ is  a minimal mass blowing up solution and $\l(s)$ is increasing for $|s|$ large,  
from Lemma \ref{lemmadecay} and then using the properties  of 
$Q_b$ (see Lemma \ref{cl:2}), we obtain for $|s|$ large,
\be
  \label{controlrightbis}
 \hbox{ $  \forall y>0$, }  |\e(s,y)| \lesssim  e^{-\frac y{20} }
 \quad \hbox{and so}\quad 
\hbox{$\forall y >0$, }   |\et(s,y)| \lesssim  e^{-\frac y{20} } .
  \ee
  
In particular, it follows that for all $\frac 1{B} \leq\omega < \frac 1{10}$,
\be\label{azero}
\lim_{s\to -\infty} \int \et^2(s,y)  \eoy  dy =0.
\ee

{\bf step 1 } We claim that 
for all $\frac 1{B} < \omega < \frac 1{10}$, for $|s|$ large,
\be\label{Astep1}
 \int  \et^2(s,y)  \eoy dy  + \int_{-\infty}^s \int \left( \et_y^2(s',y)  +  \et^2(s',y) \right) \eoy dy ds' \lesssim \frac 1{|s|}.
\ee

Define
$$
H_0(s) = \frac 12 \int \et^2(s,y) \eoy dy.
$$
Then,
\begin{align*}
&\frac d{ds} H_0  = 
\int \et_s \et \eoy  =-\int \left(-\et_{yy} + \et - F(\et)\right) (\et \eoy)_y
 + \lsl \int (\Lambda Q+ \Lambda \et) \et \eoy \\
 &+ \left(\xsl-1\right) \int (Q' + \et_y) \et \eoy\\
& = -\frac 32 \omega \int \et_y^2 \eoy - \frac 12 \omega(1-\omega^2) \int \et^2 \eoy 
+ \int F(\et) (\et \eoy)_y \\
& + \lsl \int \Lambda Q \et \eoy - \frac \omega 2 \lsl \int \et^2 y \eoy
+ \left(\xsl-1\right) \int Q' \et \eoy - \frac \omega 2 \left(\xsl-1\right) \int \et^2 \eoy.
\end{align*}

First, by decay properties of $Q$ and since $\|\et\|_{L^\infty}^2 \lesssim \frac 1{|s|}$  (by \eqref{et}),
for $|s|$ large,
\begin{align*}
\left| \int F(\et) (\et \eoy)_y \right|&\lesssim
\int (|\et|Q^4 + |\et|^5 ) (|\et_y|+|\et|) \eoy
\\ &\lesssim \frac {\omega} {100}  \int (\et_y^2+ \et^2) \eoy + \int (\et_y^2 + \et^2)\varphi_{B}.
\end{align*}

Second, by $|\lsl|+\left|\xsl-1\right|\lesssim \frac 1{|s|}$, the decay properties of $Q$ and \eqref{et}  
    \be\label{pourri}
\left|\lsl \int \Lambda Q \et \eoy \right|
+ \left|\left(\xsl-1\right) \int Q' \et \eoy \right|
\lesssim \frac 1{|s|}\left(\int \et^2 e^{-\frac {|y|}{10}}\right)^{\frac 12}
\lesssim \frac 1{|s|^2}.
\ee
Finally, for  $\frac 1{B}<\omega''<\omega<\omega'\leq \frac 1{10}$, and then using \eqref{azero},
\begin{align*}
\left|\lsl \int \et^2 y \eoy \right| & \lesssim 
\frac 1{|s|} \left(\int \et^2 y^2 e^{\omega y}\right)^{\frac 12} \left(\int \et^2  e^{\omega y}\right)^{\frac 12} \\
&\lesssim \frac 1{|s|} \left(\int \et^2  \left(e^{\omega''y} + e^{\omega' y}\right)\right)^{\frac 12} \left(\int \et^2  e^{\omega y}\right)^{\frac 12}
\lesssim \frac {\omega}{100} \int \et^2  e^{\omega y} +\frac 1{s^2}.
\end{align*}

In conclusion, we get
$$
\frac d{ds} H_0 \leq  
- \frac {\omega}{4}\int \left( \et_y^2   +  \et^2  \right) \eoy dy +  C\int (\et_y^2 + \et^2)\varphi_{B}  +\frac C{|s|^2}.
$$
Integrating on $(-\infty,s]$,   
  using \eqref{initialisation} and $\lim_{s\to -\infty} H_0(s) =0$ by \eqref{azero}, we get \eqref{Astep1}. In particular, for some sequence $s_n\to -\infty$,
\be\label{azero1}
\lim_{n\to \infty} \int   \et_y^2(s_n,y)\eoy dy =0.
\ee

{\bf step 2 } We claim that 
for all $\frac 1{B} < \omega < \frac 1{10}$, for $|s|$ large,
\be\label{Astep2}
 \int  \et_y^2(s,y)  \eoy dy  + \int_{-\infty}^s \int   \et_{yy}^2(s',y)    \eoy dy ds' \lesssim \frac 1{|s|}.
\ee

Define
$$
H_1(s) = \int \frac 12  \left(\et_y^2(s,y) +\et^2(s,y)\right) \eoy
- \left( \frac {(Q+\et)^6}6 - \et Q^5 - \frac {Q^6}{6}\right) \eoy dy.
$$
Then,
\begin{align*}
& \frac d{ds} H_1=\int \et_s \left( - \et_{yy} + \et - F(\et)\right) \eoy
- \omega \int \et_s \et_y \eoy \\
& = -\frac \omega 2 \int \left(-\et_{yy} + \et-F(\et)\right)^2 \eoy
- \omega \int \left(-\et_{yy} + \et-F(\et)\right)_y \et_y \eoy\\
& + \lsl \int ( \Lambda Q + \Lambda \et) \left(-\et_{yy} + \et-F(\et)\right) \eoy
+ \omega \lsl \int (\Lambda Q + \Lambda \et) \et_y \eoy\\
& + \left(\xsl-1\right) \int (Q'+ \et_y) \left(-\et_{yy} + \et-F(\et)\right) \eoy
+ \omega  \left(\xsl-1\right) \int (Q'+ \et_y) \et_y \eoy\\
& \leq - \omega \int \et_{yy}^2 \eoy - \omega( 1- \tfrac 12 \omega^2) \int \et_y^2 \eoy
-\omega \int F(\et) (\et_y \eoy)_y\\
&+ \lsl \int ( \Lambda Q + \Lambda \et) \left(-\et_{yy} + \et-F(\et)\right) \eoy
+ \omega \lsl \int (\Lambda Q + \Lambda \et) \et_y \eoy\\
& + \left(\xsl-1\right) \int (Q'+ \et_y) \left(-\et_{yy} + \et-F(\et)\right) \eoy
+ \omega  \left(\xsl-1\right) \int (Q'+ \et_y) \et_y \eoy
\end{align*}

First, as in step 1,
for $|s|$ large,
\begin{align*}
\left| \int F(\et) (\et_y \eoy)_y \right|&\lesssim
\int (|\et|Q^4 + |\et|^5 ) (|\et_{yy}|+|\et_y|) \eoy
\\ &\lesssim \frac {\omega} {100}  \int (\et_{yy}^2+ \et_y^2+\et^2) \eoy + \int (\et_y^2 + \et^2)\varphi_{B}.
\end{align*}
Second,  the following estimates are proved as in step 1, \eqref{pourri}, after possible integrations by parts
\begin{align*}
&\left| \lsl \int  \Lambda Q   \left(-\et_{yy} + \et-F(\et)\right) \eoy\right|+
\left| \lsl \int  \Lambda Q  \et_y \eoy\right|\\ &+ \left|
  \left(\xsl-1\right) \int  Q'  \left(-\et_{yy} + \et-F(\et)\right) \eoy\right|+
 \left|  \left(\xsl-1\right) \int Q' \et_y \eoy\right| 
 \lesssim \frac 1{s^2}.
\end{align*}
  For example,
by the decay properties of $Q$ and \eqref{initialisation},
$$
\left|\lsl \int  \Lambda Q   \et_{yy}   \eoy\right|
\lesssim \frac 1{|s|} \int |(\Lambda Q \eoy)_{yy}| |\et| \lesssim \frac 1{s^2}.
$$

Finally, we observe that
$$
\int   (\Lambda \et) \et_{yy}  \eoy
= \int (-\et_y^2\eoy + \tfrac {\omega^2}{2} \et^2 \eoy 
+ \et_y^2 (y\eoy)_y),
$$
and thus  for some $\frac 1{B} <\omega'' <\omega<\omega'<\frac 1{10}$,
$$
\left| \lsl \int   (\Lambda \et) \et_{yy}  \eoy\right|
\lesssim \frac 1{|s|} \int (\et_y^2 + \et^2) \left(e^{\omega' y} + e^{\omega'' y}\right).
$$
All the remaining terms are easier and are treated similarly as in step 1.\\
The collection of above bounds yields:
\begin{align*}
\frac d{ds} H_1 & \lesssim 
- \int \left(\et_{yy}^2+  \et_y^2   +  \et^2  \right) \eoy dy +  \int (\et_y^2 + \et^2)\varphi_{B}\\
& +  \int (\et_y^2 + \et^2) \left(e^{\omega' y} + e^{\omega'' y}\right)   +\frac 1{|s|^2}.
\end{align*}
 Note that   $\lim_{n\to \infty} H_1(s_n)=0$ by \eqref{azero} and \eqref{azero1}.
Integrating on $[s_n,s]$, and then passing to the limit as $n\to +\infty$, using   \eqref{initialisation} and \eqref{Astep1} for $\omega'$ and $\omega''$,   we find \eqref{Astep2}. In particular, there exists a subsequence still denoted $(s_n)$ such that
\be\label{azero2}
\lim_{n\to \infty} \int   \left(\et_{yy}^2+\e_y^2+\et^2\right)(s_n,y)\eoy dy =0.
\ee

{\bf step 3 } We claim that 
for all $\frac 3{B} < \omega < \frac 1{10}$, for $|s|$ large,
\be\label{Astep3}
 \int  \et_{yy}^2(s,y)  \eoy dy  + \int_{-\infty}^s \int   \et_{yyy}^2(s',y)  \eoy dy ds' \lesssim \frac 1{|s|}.
\ee

Define
$$
H_2(s) = \frac 12  \int \et_{yy}^2 \eoy - \frac {25}{6} \int \et_y^2 \et^4 \eoy.
$$
Then,
$$ \frac d{ds} H_2=
   \int \et_{yys}   \et_{yy} \eoy   - \frac {25}3 \int \left(\et_{ys}\et_y \et^4 + 2\et_s  \et_y^2 \et^3\right)\eoy= H_{2,1}+ H_{2,2}
$$

First,
\begin{align*}
H_{2,1}&=   \int \left(-\et_{yy}+\et - F(\et)\right)_{yyy}  \e_{yy} \eoy  \\
&  + \lsl\int (\Lambda Q + \Lambda \et)_{yy}\et_{yy} \eoy  
+\left(\xsl-1\right) \int (Q'+\et_y)_{yy} \et_{yy} \eoy\\
& = -\frac 32 \omega \int \et_{yyy}^2 \eoy - \frac \omega2 (1-\omega^2) \int \et_{yy}^2 \eoy
- 50 \int \et_{yy}^2 \et_y \et^3 \eoy
\\ &
+ \int (F(\et)-\et^5)_{yy} (\e_{yy}\eoy)_y
 \\&+ \frac 52 \omega \int \et_{yy}^2 \et^4 \eoy +30 \int \et_y^5 \et \eoy
+ 15 \omega \int \et_y^4 \et^2 \eoy\\
&  + \lsl\int (\Lambda Q + \Lambda \et)_{yy}\et_{yy} \eoy  
+\left(\xsl-1\right) \int (Q'+\et_y)_{yy} \et_{yy} \eoy.
\end{align*}

Second,
\begin{align*}
 H_{2,2} &
 = \frac {25}3\int \et_{s} \left( (\et_y\et^4 \eoy)_y - 2 \et_y^2 \et^3\eoy\right)
 \\&=
  \frac {25}3\int\left(-\et_{yy}+\et- F(\et)\right)_{y}
\left(\et_{yy} \et^4+ 2 \et_y^2 \et^3 +\omega \et_y \et^4 \right)\eoy \\
& +\frac {25}3 \lsl \int (\Lambda Q+ \Lambda \et) \left(\et_{yy} \et^4+ 2 \et_y^2 \et^3 +\omega \et_y \et^4 \right)\eoy
\\ &+\frac {25}3\left(\xsl-1\right) \int (Q'+\et_y) \left(\et_{yy} \et^4+ 2 \et_y^2 \et^3 +\omega \et_y \et^4 \right)\eoy\\
&=50\int \et_{yy}^2 \et_y\et^3 \eoy  + \frac {25}3\omega \int \et_{yy}^2 \et^4 \eoy 
-\frac {25}2 \int \et_y^5 \et \eoy -\frac {175}{12} \omega \int \et_y^4 \et^2 \eoy\\& -\frac {25}9 {\omega^3}  \int \et_y^3 \et^3 \eoy
 -\frac {25}3\int\left(F(\et)- \et^5\right)_{y}
\left(\et_{yy} \et^4+ 2 \et_y^2 \et^3 +\omega \et_y \et^4 \right)\eoy\\
& +\frac {25}3 \lsl \int (\Lambda Q+ \Lambda \et) \left(\et_{yy} \et^4+ 2 \et_y^2 \et^3 +\omega \et_y \et^4 \right)\eoy
\\ &+\frac {25}3\left(\xsl-1\right) \int (Q'+\et_y) \left(\et_{yy} \et^4+ 2 \et_y^2 \et^3 +\omega \et_y \et^4 \right)\eoy.
\end{align*}

The main observation when looking at   the above expressions of $H_{2,1}$ and $H_{2,2}$ is that
the higher order nonlinear term $\int \et_{yy}^2 \et_y\et^3 \eoy$   cancels in the expression  of $\frac d{ds} H_2$. All other terms are now controlled as follows.

First, by \eqref{et},
\begin{align*}
\left| \int \et_{yy}^2 \et^4 \eoy \right|
\lesssim \frac 1{|s|^2} \int \et_{yy}^2 \eoy.
\end{align*}
Second, by Holder inequality, \eqref{gn}, and then \eqref{Astep1}, \eqref{Astep2}, for 
$\frac 3{100} < \omega < \frac 1{10}$,
\begin{align*}
\left|  \int \et_y^5 \et \eoy \right|
&\lesssim \left(\int \et_y^6 \eoy\right)^{\frac 56} \left(\int \et^6 \eoy\right)^{\frac 16}\\
& \lesssim\left(\int \left(\et_{yy}^2 + \et_y^2+ \et^2\right) \eoyt\right)^{\frac 56}
\left(\int \left(\et_y^2+ \et^2\right) \eoyt\right)^{\frac {13}6}\\
&\lesssim \int \left(\et_{yy}^2 + \et_y^2+ \et^2\right) \eoyt + \frac 1{|s|^{13}}.
 \end{align*}
Similar estimates are proved for  $|\int \et_y^4 \et^2 \eoy|$ and $|\int \et_y^3 \et^3 \eoy|$.
Next, for terms containing $F(\et)-\et^5$, we argue as follows. A first observation is (using \eqref{estrescbisAA}),
$$
|(F(\et)-\et^5)_y|\lesssim (|\et_y| + |\et|) Q,\quad
|(F(\et)-\et^5)_{yy}|\lesssim (|\et_{yy}|+|\et_y|^2 + |\et_y| + |\et|) Q.
$$
Thus,
\begin{align*}
& \left| \int (F(\et)-\et^5)_{yy} (\et_{yy}\eoy)_y\right|
\leq C \int (|\et_{yy}| + |\et_y|^2 + |\et_y| + |\et|) (|\et_{yyy}|+|\et_{yy}|) Q \eoy\\
& \leq \frac 1{100} \int (\et_{yyy}^2+\et_{yy}^2) \eoy +C \int (\et_{yy}^2 + \et_y^4 + \et_y^2 + \et^2) Q\\
& \leq \frac 1{100} \int (\et_{yyy}^2+\et_{yy}^2) \eoy +C \int (\et_{yy}^2 +   \et_y^2 + \et^2) \eoy + \frac C {s^2}.
\end{align*}
The term $|\int\left(F(\et)- \et^5\right)_{y}
\left(\et_{yy} \et^4+ 2 \et_y^2 \et^3 +\omega \et_y \et^4 \right)\eoy|$ is treated similarly and easier.

Finally, terms containing $\lsl$ and $(\xsl-1)$ are treated similarly as in step 1 and step 2.
For example, let us consider the term 
$\lsl\int (\Lambda Q + \Lambda \et)_{yy}\et_{yy} \eoy $.
We first have
$$
\left|\lsl\int  (\Lambda Q)_{yy}\et_{yy} \eoy\right|
= \left|\lsl\int  ((\Lambda Q)_{yy}\eoy)_{yy}\et  \right|\lesssim \frac 1{s^2}.
$$

Since
$$
\int (\Lambda \et)_{yy} \et_{yy} \eoy = \frac 32 \int \et_{yy}^2 \eoy - \omega \int \et_{yy}^2 y \eoy,
$$
we
get, for some $\frac 1{100} < \omega''<\omega< \omega'<\frac 1{10}$,
$$
\left| \lsl\int (\Lambda \et)_{yy}\et_{yy} \eoy\right| \lesssim
\frac 1{s^2} \int  \et_{yy}^2   \left(e^{\omega' y} + e^{\omega'' y}\right).
$$

Gathering all the previous estimates, 
  we obtain
\begin{align*}
\frac d{ds} H_2 & \lesssim 
- \int \left(\et_{yyy}^3 + \et_{yy}^2+  \et_y^2   +  \et^2  \right) \eoy dy \\
& +  \int (\et_{yy}^2 + \et_y^2 + \et^2) \left(e^{\frac \omega 3 y} + e^{\omega'' y}\right)   +\frac 1{|s|^2}.
\end{align*}
Integrating on $[s_n,s]$ and passing to the limit $n\to +\infty$ using \eqref{initialisation}, \eqref{azero2} and \eqref{Astep2},
we get \eqref{Astep3}.

For some   sequence $s_n'\to -\infty$,
it implies
\be\label{azero3}
\lim_{n\to \infty} \int   \left(\et_{yyy}^2+\et_{yy}^2+\e_y^2+\et^2\right)(s_n',y)\eoy dy =0.
\ee
Note also that by standard arguments, \eqref{Astep3} implies directly that
\be\label{Astep3b}
\left\|\e_y^2 e^{\omega y}\right\|_{L^\infty} \lesssim \frac 1{|s|}.
\ee

\medskip

{\bf step 4 } We claim that 
for all $\frac 9{B} < \omega < \frac 1{10}$, for $|s|$ large,
\be\label{Astep4}
 \int  \et_{yyy}^2(s,y)  \eoy dy  + \int_{-\infty}^s \int   (\partial_y^4 \et)^2(s',y)  \eoy dy ds' \lesssim \frac 1{|s|}.
\ee

Define
$$
H_3(s) = \frac 12  \int \et_{yyy}^2 \eoy .
$$
Then,
\begin{align*}
& \frac d{ds} H_3=
   \int \et_{yyys}   \et_{yyy} \eoy   =
  \int \left(-\et_{yy}+\et - F(\et)\right)_{yyyy}  \e_{yyy} \eoy  \\
&  + \lsl\int (\Lambda Q + \Lambda \et)_{yyy}\et_{yyy} \eoy  
+\left(\xsl-1\right) \int (Q'+\et_y)_{yyy} \et_{yyy} \eoy\\
& = -\frac 32 \omega \int \et_{yyyy}^2 \eoy - \frac \omega2 (1-\omega^2) \int \et_{yyy}^2 \eoy
+ \int (F(\et))_{yyy} (\e_{yyy}\eoy)_y\\
&  + \lsl\int (\Lambda Q + \Lambda \et)_{yyy}\et_{yyy} \eoy  
+\left(\xsl-1\right) \int (Q'+\et_y)_{yyy} \et_{yyy} \eoy.
\end{align*}
The last two terms $\lsl\int (\Lambda Q + \Lambda \et)_{yyy}\et_{yyy} \eoy $ and $\left(\xsl-1\right) \int (Q'+\et_y)_{yyy} \et_{yyy} \eoy$ are treated exactly as in the previous steps and thus we omit the estimates.

We focus on the nonlinear term  $\int (F(\et))_{yyy} (\e_{yyy}\eoy)_y$.
Expanding $F(\et) = 5 Q^4 \et + 10 Q^3 \et^2 + 10 Q^2 \et^3 + 5 Q \et^4 + \et^5$ and integrating by parts,
we obtain many different terms. We check the worst terms and we claim that the other terms can be checked
similarly. See also Section 3.4 in \cite{Ma1} for similar arguments.

First, we remark that the following term which is only quadratic in $\et$, is easily controlled 
$$
\left| \int \et_{yyy}^2 (Q^4)' e^{\omega y}\right| \lesssim \int  \et_{yyy}^2  e^{\omega y}.
$$

Second, we treat some terms coming from $\et^5$:
$$
\left| \int \et_{yyy}^2 \et_y \et^3 e^{\omega y}\right|
\lesssim \|\et\|_{L^\infty}^3 \left\|\et_y e^{\frac \omega 2 y}\right\|_{L^\infty} \int \et_{yyy}^2  e^{\frac \omega 2y}
\lesssim \frac 1{|s|^2}  \int \et_{yyy}^2  e^{\frac \omega 2y};
$$

\begin{align*}
 \left| \int \et_{yy}^3 \et_y \et^2 e^{\omega y}\right|& \lesssim 
\|\et\|_{L^\infty}^2 \left\|\et_y e^{\frac \omega 4 y}\right\|_{L^\infty}  \left|
\int \et_{yy}^3   e^{\frac {3\omega} 4 y}\right|  \\
&\lesssim 
\frac 1{|s|^{\frac 32}} \left(\int \left(\et_{yyy}^2+ \et_{yy}^2 \right)e^{\frac \omega 2 y}\right)^{\frac 14}
\left(\int \et_{yy}^2e^{\frac \omega 2 y} \right)^{\frac 54} \\
& \lesssim \frac 1{|s|^{\frac {11}3}} + \int \et_{yyy}^2 e^{\frac \omega 2 y}.m
\end{align*}
 \begin{align*}
 \left| \int \et_{yy}^2 \et_y^3 \et e^{\omega y}\right|& \lesssim 
\|\et\|_{L^\infty} \left\|\et_y e^{\frac \omega 5 y}\right\|_{L^\infty}^3  \left|
\int \et_{yy}^2   e^{\frac {2\omega} 5 y}\right|  \lesssim \frac 1{|s|^3}  .
\end{align*}

Thus, we get 
$$
\frac d{ds} H_3 \lesssim - \int (\partial_y^4 \e)^2 e^{\omega y}+ \frac 1{|s|^2} + \int \left(\et_{yyy}^2+\et_{yy}^2 + \et_y^2 + \et^2\right) \left(e^{\frac {2\omega} 5y}+e^{\omega y}\right).
$$
Integrating on $[s_n',s]$ and passing to the limit as $n\to +\infty$, using \eqref{azero3} and \eqref{Astep3}, we obtain \eqref{Astep4}, for the following range of values of $\omega$: $\frac {15}{2 B} <\omega <\frac 1{10}$.

\section{Proof of Lemma \ref{le:6.4}}\label{app:B}

 For simplicity of notation, we denote $\e_{v_n}$, $\lambda_{v_n}$, $b_{v_n}$ and $x_{v_n}$
 simply by $\e$, $\lambda$, $b$ and $x$.
\\

\textbf{step 1} Algebraic computations.
We follow closely the computations of the proof of Proposition 3.1 in \cite{MMR1}.
First,
\begin{align*}
  \frac d{d{s}}  {\cal F_B}   &=\frac{1}{\lambda^2}\left(
\frac 12 \int  \left[ \e_y^2 \tilde \phi_B' + \frac{\lambda^2}{\lambda_{0}^2}\e^2 \phi_B' 
- \frac 13 \left( (\e+Q_b)^6 - Q_b^6 - 6 \e Q_b^5\right)\tilde \phi_B' \right]\right.  \\
 &  + 2\int \tilde \phi_B(\varepsilon_y)_s \varepsilon_y + 2\varepsilon_s \left[  \frac{\lambda^2}{\lambda_{0}^2}\varepsilon \phi_B - \tilde \phi_B\left((\varepsilon+ Q_b)^5 -Q_b^5\right)\right]\\
&   \left.-2 \int \tilde \phi_B(Q_b)_s \left((\varepsilon+Q_b)^5 - Q_b^5 - 5 \varepsilon Q_b^4\right) \right)\\
& - 2  \frac{\lambda_s}{\lambda^3}   \int  \left[ \e_y^2 \tilde \phi_B  
- \frac 13 \left( (\e+Q_b)^6 - Q_b^6 - 6 \e Q_b^5\right)\tilde \phi_B \right]
-  2  \frac{(\lambda_{0})_s}{\lambda_{0}^3}   \int    \e^2 \phi_B
\\ &
 =  \frac{1}{\lambda^2}\left(f_{1}+f_2+f_3+f_4\right) ,
\end{align*}
  where 
\begin{align*}
  f_{1} & =\frac 12 \int   \left[\e_y^2 \tilde \phi_B' + \e^2 \phi_B' 
- \frac 13 \left( (\e+Q_b)^6 - Q_b^6 - 6 \e Q_b^5\right)\tilde \phi_B'  \right] \\
 &  +
      2 \int \left( \varepsilon_{s} - {\frac{{\lambda}_s}{{\lambda}}} {\Lambda} \varepsilon\right) \left( - (\tilde \phi_B\e_y)_{y} + \varepsilon \phi_B -\tilde \phi_B\left((\varepsilon +Q_b )^5 - Q_b ^5\right) \right) ,\\
f_2 & = 2 \left(1-\frac{\lambda^2}{\lambda_{0}^2}\right)   \int\varepsilon_s     \varepsilon \phi_B -  2  \frac{(\lambda_{0})_s}{\lambda_{0}}   \int    \e^2 \phi_B\\
f_3 & = 
 2 {\frac{{\lambda}_s}{{\lambda}}}  \int    {\Lambda} \varepsilon  \left( - (\tilde \phi_B\varepsilon_y)_y + \varepsilon\phi_B -\tilde \phi_B\left((\varepsilon +Q_b )^5 - Q_b ^5\right)\right)\\
 & - 2  \frac{\lambda_s}{\lambda}   \int  \left[ \e_y^2 \tilde \phi_B  
- \frac 13 \left( (\e+Q_b)^6 - Q_b^6 - 6 \e Q_b^5\right)\tilde \phi_B \right]\\
 f_4 & =
  - 2  \int \tilde \phi_B(Q_b)_{s} \left( (\varepsilon +Q_b)^5 - Q_b ^5 - 5 \varepsilon Q_b^4\right).
\end{align*}
We  use the equation of $\e$ under the following form
\bea
\label{eqebis}
 \varepsilon_{s} -  {\frac{{\lambda}_s}{{\lambda}}} {\Lambda} \varepsilon
 &  = &  \left(-\varepsilon_{yy} + \varepsilon - (\varepsilon +Q_b )^5 + Q_b ^5\right)_y      \\
  \nonumber  & + & \left(\frac {{\lambda}_{s}}{{\lambda}}+{b}\right) {\Lambda} Q_b
+ \left(\frac { x_{{s}}}{\lambda} -1\right) (Q_b  + \varepsilon)_y  + \Phi_{{b}} + \Psi_{{b}},
\eea
where
$\Phi_b = - {b}_{s} \left(\chi_b   + \gamma   y (\chi_b)_y\right) P$
and
$-\Psi_b=\left(Q_b''- Q_b+ Q_b^5\right)'+b {\Lambda} Q_b$.

\medskip

\textbf{step 2} Control of $f_1$.
\begin{align*}
  \rm f_{1}& = \frac 12 \int  \left[ \e_y^2 \tilde \phi_B' + \e^2 \phi_B' 
- \frac 13 \left( (\e+Q_b)^6 - Q_b^6 - 6 \e Q_b^5\right)\tilde \phi_B'  \right] \\
 &  +2 \int   \left(  -\varepsilon_{yy} {+} \varepsilon  {-}\left((\varepsilon +Q_b )^5 {-} Q_b ^5\right)  \right)_y \left(-(\tilde \phi_B\e_y)_y{+}\e\phi_B{-}\tilde \phi_B[(Q_b+\e)^5-Q_b^5]\right) 
\\ & +2\left(\frac {{\lambda}_{s}}{{\lambda}}+{b}\right) \int {\Lambda} Q_b  \left( - (\tilde \phi_B\varepsilon_y)_y + \varepsilon \phi_B-\tilde \phi_B\left((\varepsilon +Q_b )^5 - Q_b ^5\right)   \right)
\\ &   + 2 \left(\frac { x_{{s}}}{\lambda} -1\right) \int (Q_b  + \varepsilon)_y
 \left( -(\tilde \phi_B \varepsilon_y)_y + \varepsilon \phi_B-\tilde \phi_B\left((\varepsilon +Q_b )^5 - Q_b ^5\right)   \right)
 \\ & + 2 \int  \Phi_{{b}}  \left( -(\tilde \phi_B \varepsilon_y)_y + \varepsilon \phi_B-\tilde \phi_B\left((\varepsilon +Q_b )^5 - Q_b ^5\right)   \right)
\\ & + 2 \int  \Psi_{{b}}   \left( -(\tilde \phi_B \varepsilon_y)_y + \varepsilon \phi_B-\tilde \phi_B\left((\varepsilon +Q_b )^5 - Q_b ^5\right)   \right) \\
& = {{\rm f}_{1,1}}+ {{\rm f}_{1,2}}+{{\rm f}_{1,3}}+{{\rm f}_{1,4}}+{{\rm f}_{1,5}}.
\end{align*}

As in \cite{MMR1}, we obtain after some computations,
\bee
f_{1,1} &=&  -\int\left[3\tilde \phi_B'\e_{yy}^2+(3\phi_B'+\frac 12 \tilde \phi_B'-\tilde \phi_B''')\e_y^2+(\frac 12\phi_B' - \phi_B''')\e^2\right]\\
\nonumber &
-& \frac 16 \int  \left( (\e+Q_b)^6 - Q_b^6 - 6 \e Q_b^5\right)\tilde \phi_B'
\\ \nonumber & - &2\int \left[\frac {(\varepsilon + Q_b )^6} 6-\frac {Q_b ^6} 6 - Q_b ^5 \varepsilon -\left((\varepsilon +Q_b )^5 - Q_b ^5\right) \varepsilon \right](\phi_B'-\tilde \phi_B')\\
\nonumber & + & 2  \int  \left[(\varepsilon+Q_b )^5 -Q_b ^5- 5 Q_b ^4 \varepsilon\right](Q_b )_y (\tilde \phi_B-\phi_B)\\
& + & 10\int\tilde \phi_B'\e_y\left\{(Q_b)_y[(Q_b+\e)^4-Q_b^4]+(Q_b+\e)^4\e_y\right\}\\
& + & \int\tilde \phi_B' \left[-2\e_{yy}+ 2\varepsilon  -\left((\varepsilon +Q_b )^5 - Q_b ^5\right)\right]  \left[ (\varepsilon +Q_b )^5 - Q_b^5\right]  \\
\eee
Using the following estimates (see \cite{MMR1} for more details), 
\begin{align}
& \tilde \phi_B''' \lesssim \frac 1{B^2}   \phi_B',\quad
  \phi_B''' \lesssim \frac 1{B^2}   \phi_B',   \quad  \hbox{ for all $y\in \RR$,}
  \label{un}\\&
  |Q_b(y)|+|(Q_b)_y(y)|\lesssim e^{-|y|} + |b|,  \quad  \hbox{ for all $y\in \RR$,} \label{deux}\\
 & \int \e^6 \phi_B' \lesssim \delta(\alpha^*)\int (\e_y^2 + \e^2) \phi_B',\\
 &
  \int \e_y^2 \e^4 \tilde \phi_B' \leq \delta(\alpha^*) \left(\int \e_{yy}^2 \tilde \phi_B' + 
  \int (\e_y^2 + \e^2) \phi_B'\right),\label{trois}
\end{align}
and the bound on the $L^2$ norm of $\e$ (see Lemma \ref{le:6.1}), 
we obtain for $B$ large and $\alpha^*$ small:
\begin{align*}
f_{1,1} &
\leq - \int \tilde \phi_B'\e_{yy}^2 - \frac 14 \int (\e_y^2 + \e^2) \phi_B'
+ C \int(\e_y^2 + \e^2) e^{-\frac{|y|}{10}} +C|b|^4
\\ & \leq - \int \tilde \phi_B'\e_{yy}^2 - \frac 14 \int (\e_y^2 + \e^2) \phi_B'
+ C \int(\e_y^2 + \e^2) \frac{\varphi_B}{1+y_+^2} +C|b|^4.
\end{align*}
 
Next,
\begin{align*}
f_{1,2} & =  2 \left(\frac {{\lambda}_{s}}{{\lambda}}+{b}\right)  \int {\Lambda} Q  (L \varepsilon) -2\left(\frac {{\lambda}_{s}}{{\lambda}}+{b}\right)\int\varepsilon (1-\phi_B)\Lambda Q
\\
&+ 2b \left(\frac {{\lambda}_{s}}{{\lambda}}+{b}\right)  \int {\Lambda} ( \chi_b P)  \left(-(\tilde \phi_B\e_y)_{y}  + \e\phi_B - \tilde \phi_B[(Q_b+\e)^5-Q_b^5)]\right)\\
& + 2 \left(\frac {{\lambda}_{s}}{{\lambda}}+{b}\right)  \int {\Lambda} Q      \left(-(\tilde \phi_B)_y\e_y
-(1-\tilde \phi_B)\e_{yy}+(1-\tilde \phi_B)[(Q_b+\e)^5-Q_b^5]\right) \\&
+ 2\left(\frac {{\lambda}_{s}}{{\lambda}}+{b}\right)  \int {\Lambda} Q  \left[(Q_b+\e)^5-Q_b^5-5Q^4\e\right].
\end{align*}
The main term  $\int {\Lambda} Q  (L \varepsilon)$ is zero by the orthogonality conditions on 
$\e$ and the other terms are controled as in \cite{MMR1} using \eqref{eq:2002}, \eqref{un}, \eqref{deux} and
\eqref{trois},
 to obtain
$$
| f_{1,2} | \leq \frac 1{100} \int (\e_y^2 + \e^2) \phi_B' + C \int(\e_y^2 + \e^2) \frac{\varphi_B}{1+y_+^2} +C|b|^4.
$$

The next term is  
\bee
f_{1,3}& = & 2\left(\xsl-1\right)\int\frac 16\tilde \phi_B'\left[(Q_b+\e)^6-Q_b^6-6Q_b^5\e\right]\\
& + & 2\left(\xsl-1\right)\int(b\chi_bP+\e)_y\left[-\tilde \phi_B'\e_y-\tilde \phi_B\e_{yy}+\e\phi_B\right]\\
& + &  2\left(\xsl-1\right) \int Q'\left[L\e-\tilde \phi_B' \e_y + (1-\tilde \phi_B)\e_{yy}-\e(1-\phi_B)\right]\\
& + &  10 \left(\xsl-1\right)  \int\e\tilde \phi_B(Q_b^4(Q_b)_y-Q^4Q_y) 
\eee
Using $LQ'=0$ and arguing similarly as before, we obtain
$$
| f_{1,3} | \leq \frac 1{100} \int (\e_y^2 + \e^2) \phi_B' + C \int(\e_y^2 + \e^2) \frac{\varphi_B}{1+y_+^2} +C|b|^4.
$$

 \medskip
 
\textbf{step 3} Control of $f_2  $.

First, by \eqref{6.15}, we have 
$$-\frac{(\lambda_{0})_s}{\lambda_{0}}   \int    \e^2 \phi_B<0.$$
     
Next, by the definition of $\lambda_0$ in Lemma \ref{le:oubli}, we have
$$
 \left|1-\frac{\lambda^2}{\lambda_{0}^2}\right| \lesssim \mathcal N^{\frac 12}
 \lesssim \delta(\alpha^*),
$$
and thus, proceeding for $\int\varepsilon_s     \varepsilon \phi_B$ as in the previous step, we find
$$  \left|1-\frac{\lambda^2}{\lambda_{0}^2}\right| \left|  \int\varepsilon_s     \varepsilon \phi_B\right|
\leq \frac 1{100} \int (\e_y^2 + \e^2) \phi_B' +C\int  (\e_y^2 + \e^2) \frac{\varphi_B}{1+y_+^2}+ C|b|^4 .$$
\medskip

 \medskip
 
\textbf{step 4} Control of $f_3$.
From computations in \cite{MMR1},
\begin{align*}
f_3 
&   =  \lsl\int[2 \tilde \phi_B-y\tilde \phi_B']\e_y^2
- \lsl \int y \tilde \phi_B' \e^2 \\ &-\frac 13\lsl\int[2 \tilde \phi_B-y\tilde \phi_B']\left[(\varepsilon+Q_b )^6 - Q_b ^6 - 6 Q_b ^5 \varepsilon \right]\\
& +2 {\frac{{\lambda}_s}{{\lambda}}} \int \tilde \phi_B {\Lambda} Q_b  \left( (\varepsilon +Q_b )^5 - Q_b ^5 - 5 Q_b ^4 \varepsilon\right) \\
& - 2  \frac{\lambda_s}{\lambda}   \int  \left[ \e_y^2 \tilde \phi_B  
- \frac 13 \left( (\e+Q_b)^6 - Q_b^6 - 6 \e Q_b^5\right)\tilde \phi_B \right].
\end{align*}
After simplification of the last line with terms in the first and second lines, we obtain
\begin{align*}
f_3 
&   =  -\lsl\int y\tilde \phi_B' \left[\e_y^2+ \e^2 - \frac 13\left((\varepsilon+Q_b )^6 - Q_b ^6 - 6 Q_b ^5 \varepsilon\right) \right]\\
& +2 {\frac{{\lambda}_s}{{\lambda}}} \int \tilde \phi_B {\Lambda} Q_b  \left( (\varepsilon +Q_b )^5 - Q_b ^5 - 5 Q_b ^4 \varepsilon\right).
\end{align*}
For this term we  observe, from the definition of $\phi_B$ and $\tilde \phi_B$,
\begin{align*}
 \int |y| \tilde \phi_B' (\e_y^2 +\e^2+|\e|^6) \lesssim  \int (\e_y^2+ \e^2) \phi_B'
\end{align*}
and $\left| \lsl\right| \lesssim \delta(\alpha^*)$.
The other terms in the expression of $f_3$ are treated as before, so that  we   obtain:
$$
|f_3| \leq \frac 1{100} \int \left(\e_y^2 +\e^2\right) \phi_B'
+ C \int(\e_y^2 + \e^2) \frac{\varphi_B}{1+y_+^2} +C|b|^4
$$

\medskip

\textbf{step 5} Control of $f_4$.
Arguing exactly as in \cite{MMR1} (using \eqref{eq:2003}), we obtain
$$
|f_4| \leq \frac 1{100} \int \left(\e_y^2 +\e^2\right) \phi_B'+ C \int(\e_y^2 + \e^2) \frac{\varphi_B}{1+y_+^2} +C|b|^4.
$$

Gathering these estimates, we get \eqref{monoepsvn}.

\medskip

\textbf{step 6} Proof of \eqref{coerepsvn}.
This is a standard fact by localization arguments
(see e.g. Appendix A of \cite{MMannals}).

\end{document}